\documentclass[11pt,reqno,a4paper]{amsart}

\usepackage{graphicx}
\usepackage{epstopdf}
\usepackage{amsmath}
\usepackage{amssymb,amsthm,hyperref,caption,tikz, colonequals}
\usetikzlibrary{arrows,positioning}
\usepackage{color}
\definecolor{blau}{rgb}{0,0,0.75} 
\hypersetup{colorlinks,linkcolor=blau,citecolor=blue}

\allowdisplaybreaks

\newtheorem{theorem}{Theorem}[section]
\newtheorem{lemma}[theorem]{Lemma}
\newtheorem{coroll}[theorem]{Corollary}
\newtheorem{prop}[theorem]{Proposition}

\theoremstyle{definition}
\newtheorem{remark}[theorem]{Remark}

\newcommand{\rt}{\textsf{\textup{root}}}
\newcommand{\str}{\textsf{\textup{star}}}
\newcommand{\chain}{\textsf{\textup{chain}}}
\newcommand{\cycle}{\textsf{\textup{cycle}}}
\newcommand{\id}{\textsf{\textup{id}}}
\newcommand{\reallocate}[4]{\textsf{\textup{reallocate}}\!\left(#1 \left| \begin{smallmatrix} #2\\ \downarrow\\ #3 \end{smallmatrix} \mapsto \begin{smallmatrix} #2\\ \downarrow\\ #4 \end{smallmatrix}\right.\right)}

\begin{document}

\author{Marie-Louise Bruner and Alois Panholzer}
\address{Marie-Louise Bruner\\
Institut f{\"u}r Diskrete Mathematik und Geometrie\\
Technische Universit\"at Wien\\
Wiedner Hauptstr. 8-10/104\\
1040 Wien, Austria} \email{marie-louise.bruner@tuwien.ac.at}
\address{Alois Panholzer\\
Institut f{\"u}r Diskrete Mathematik und Geometrie\\
Technische Universit\"at Wien\\
Wiedner Hauptstr. 8-10/104\\
1040 Wien, Austria} \email{Alois.Panholzer@tuwien.ac.at}

\thanks{The authors were supported by the Austrian Science Foundation FWF, grant P25337-N23.}

\title{Parking functions for trees and mappings}

\begin{abstract}
We apply the concept of parking functions to rooted labelled trees and functional digraphs of mappings (i.e., functions $f : [n] \to [n]$) by considering the nodes as parking spaces and the directed edges as one-way streets: Each driver has a preferred parking space and starting with this node he follows the edges in the graph until he either finds a free parking space or all reachable parking spaces are occupied. If all drivers are successful we speak about a parking function for the tree or mapping.
We transfer well-known characterizations of parking functions to trees and mappings.
Especially, this yields bounds and characterizations of the extremal cases for the number of parking functions with $m$ drivers for a given tree $T$ of size $n$.
Via analytic combinatorics techniques we study the total number $F_{n,m}$ and $M_{n,m}$ of tree and mapping parking functions, respectively, i.e., the number of pairs $(T,s)$ (or $(f,s)$), with $T$ a size-$n$ tree (or $f : [n] \to [n]$ an $n$-mapping) and $s \in [n]^{m}$ a parking function for $T$ (or for $f$) with $m$ drivers, yielding exact and asymptotic results.
We describe the phase change behaviour appearing at $m=\frac{n}{2}$ for $F_{n,m}$ and $M_{n,m}$, respectively, and relate it to previously studied combinatorial contexts. Moreover, we give a bijective proof of the occurring relation $n F_{n,m} = M_{n,m}$.
\end{abstract}

\maketitle

\section{Introduction}

Parking functions are combinatorial objects originally introduced by Konheim and Weiss \cite{KonWei1966} during their studies of the so-called linear probing collision resolution scheme for hash tables. Since then, parking functions have been studied extensively and many connections to various other combinatorial objects such as forests, hyperplane arrangements, acyclic functions and non-crossing partitions have been revealed, see, e.g., \cite{Stanley1997}.

An illustrative description of parking functions is as follows: consider a one-way street with $n$ parking spaces numbered from $1$ to $n$ and a sequence of $m$ drivers with preferred parking spaces $s_{1}, s_{2}, \dots, s_{m}$. The drivers arrive sequentially and each driver $k$, $1 \le k \le m$, tries to park at his preferred parking space with address $s_{k} \in [n]$, where $[n] \colonequals \{1, 2, \dots, n\}$. If it is free he parks.
Otherwise he moves further in the allowed direction (thus examining parking spaces $s_{k}+1, s_{k}+2, \dots$) until he finds a free parking space, where he parks.
If there is no such parking space he leaves the street without parking. A parking function is then a sequence $(s_{1}, \dots, s_{m}) \in [n]^{m}$ of addresses such that all $m$ drivers are able to park. It has been shown already in \cite{KonWei1966} that there are exactly $P_{n,m} = (n+1-m) \cdot (n+1)^{m-1}$ parking functions, for $n$ parking spaces and $0 \le m \le n$ drivers.

The notion of parking functions has been generalized in various ways, yielding, e.g., $(a,b)$-parking functions~\cite{Yan2001}, bucket parking functions~\cite{BlaKon1977}, $x$-parking functions~\cite{PitSta2002}, or $G$-parking functions~\cite{PosSha2004}. 
Another natural generalization that has however not been considered yet is the following:
Starting with the original definition of parking functions as a vivid description of a simple collision resolution scheme, we apply it to other objects of interest, namely, to rooted trees and mappings, respectively.

First, when allowing branches in the road net, this collision resolution scheme leads to a natural generalization of parking functions to rooted trees. Consider a rooted labelled tree $T$ of size $|T|=n$, i.e., we assume (for simplicity) that the vertices of $T$ are labelled by distinct integers of $[n]$. Furthermore, we assume that the edges of the tree are oriented towards the root node, which we will often denote by $\rt(T)$.
We thus view edges as ``one-way streets''. 
Now, we consider a sequence of $m$ drivers, where again each driver has his preferred parking space, which in this case is a node in the tree, respectively its label (throughout this work, we will always identify a node with its label). The drivers arrive sequentially and each driver $k$, $1 \le k \le m$, tries to park at his preferred parking space with address $s_{k} \in [n]$. If it is free he parks.
Otherwise he follows the edges towards the root node and parks at the first empty node, if there is such one.
If there is no empty node, he leaves the road net, i.e., the tree without parking. A sequence $s \in [n]^{m}$ of addresses (i.e., a function $s : [m] \to [n]$) is then called a parking function for the tree $T$, if all drivers are successful, i.e., if all drivers are able to find a parking space. More precisely, we will call a pair $(T,s)$ an $(n,m)$-tree parking function, if $T$ is a rooted labelled tree of size $n$ and $s \in [n]^{m}$ is a parking function for $T$ with $m$ drivers. In Figure~\ref{fig:treepark} we give an example of a tree parking function.
\begin{figure}
\begin{center}
\includegraphics[height=3.5cm]{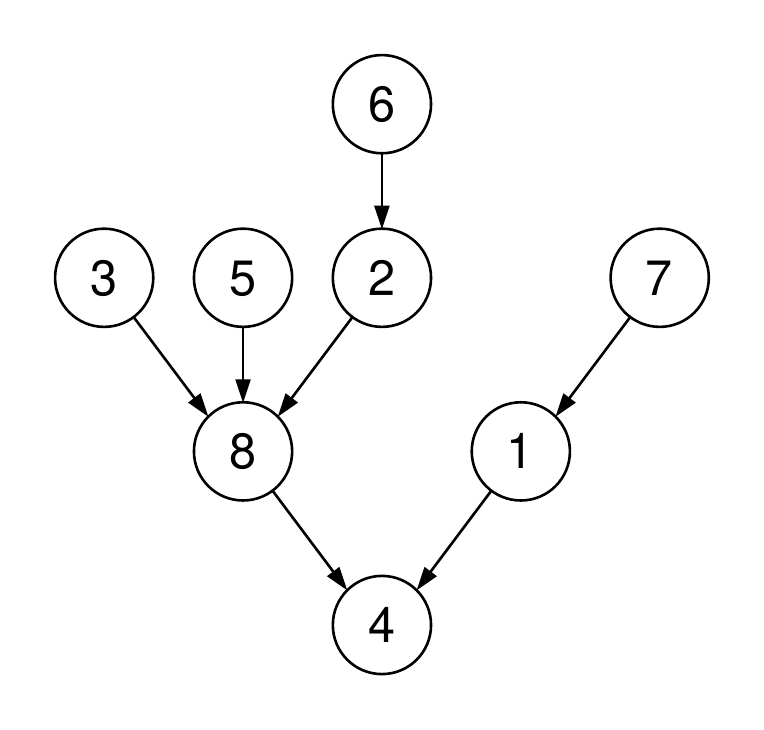}
\includegraphics[height=3.5cm]{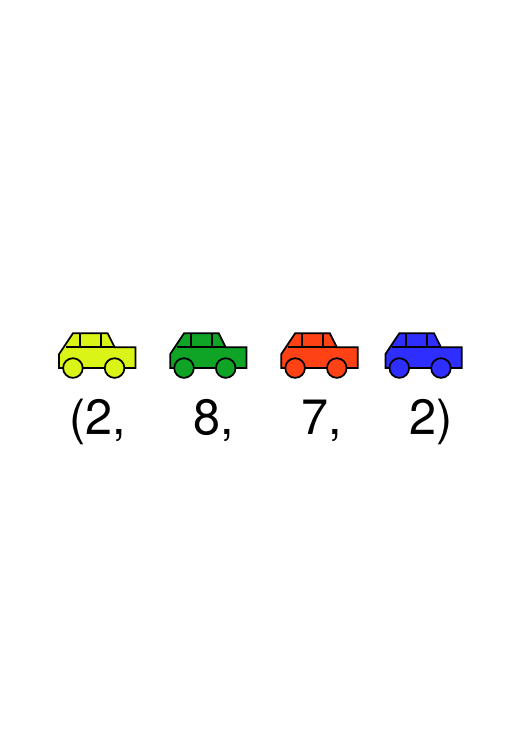}
\raisebox{1.6cm}{$\Rightarrow$}
\includegraphics[height=3.5cm]{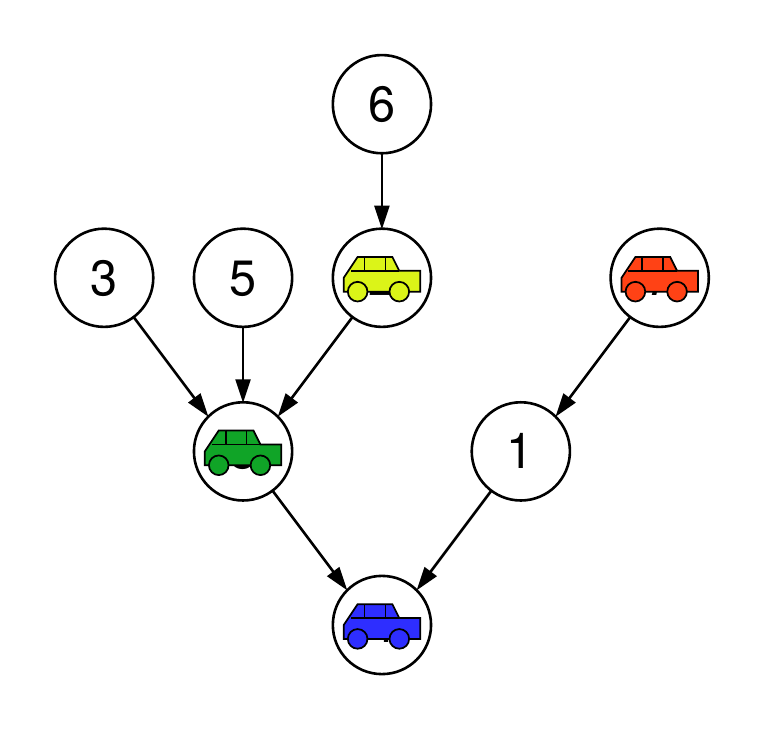}
\end{center}
\caption{A size-$8$ rooted labelled tree $T$ and a sequence $s= (2,8,7,2)$ of addresses of preferred parking spaces for $4$ drivers. All drivers are successful, thus $(T,s)$ yields a $(8,4)$-tree parking function; the parking positions of the drivers are given by the sequence $(2,8,7,4)$ defining the output-function $\pi_{(T,s)}$. Conversely, the sequence $(2,8,8,2)$ is not a parking function for $T$, since the fourth driver is not able to park.\label{fig:treepark}}
\end{figure}

Second, one can go a step further and consider structures in general for which the simple collision resolution scheme is applicable, i.e., for which a driver reaching an occupied parking space can move on in a unique way to reach a new parking space. This naturally leads to a generalization of parking functions to mappings: Consider the set $[n]$ of addresses and a mapping $f: [n] \to [n]$ (which we call here $n$-mapping).
If a driver reaches address $i$ and is unable to park there, he moves on to the parking space with address $j=f(i)$ for his next trial. The road net is then the functional digraph $G_{f}$ of the mapping $f: [n] \to [n]$, i.e., the directed graph $G_{f} = (V, E)$, with $V = [n]$ and $E = \{(i,f(i)) : i \in [n]\}$. Since functional digraphs are obviously characterized by the property that each node has out-degree $1$, the parking procedure described above can be applied. Again, the drivers arrive sequentially and each driver has his preferred parking space (a node in the graph).
If it is empty he will park, otherwise he follows the edges and parks at the first empty node, if such one exists.
Otherwise he cannot park since he would be caught in an endless loop. A sequence $s \in [n]^{m}$ of addresses (i.e., a function $s : [m] \to [n]$) is then called a parking function for the graph $G_{f}$, or alternatively, a parking function for the mapping $f$ (in this context we will always identify a mapping $f$ with its functional digraph $G_{f}$), if all drivers are successful, i.e., all drivers find a parking space.
A pair $(G_{f},s)$ (or alternatively $(f,s)$) is called an $(n,m)$-mapping parking function, if $f$ is an $n$-mapping and $s \in [n]^{m}$ is a parking function for $G_{f}$ with $m$ drivers. In Figure~\ref{fig:mappingpark} we give an example of a mapping parking function.
\begin{figure}
\begin{scriptsize}
$f : \left(\begin{array}{ccccccccccccccccccc}
1 & 2 & 3 & 4 & 5 & 6 & 7 & 8 & 9 & 10 & 11 & 12 & 13 & 14 & 15 & 16 & 17 & 18 & 19\\
\downarrow & \downarrow & \downarrow & \downarrow & \downarrow & \downarrow & \downarrow & \downarrow & \downarrow & \downarrow & \downarrow & \downarrow & \downarrow & \downarrow & \downarrow & \downarrow & \downarrow & \downarrow & \downarrow \\
5 & 7 & 1 & 12 & 13 & 10 & 14 & 10 & 2 & 13 & 5 & 18 & 12 & 7 & 5 & 14 & 13 & 5 & 14
\end{array}
\right)$
\end{scriptsize}
\begin{minipage}{14cm}
\begin{center}
\includegraphics[height=3.5cm]{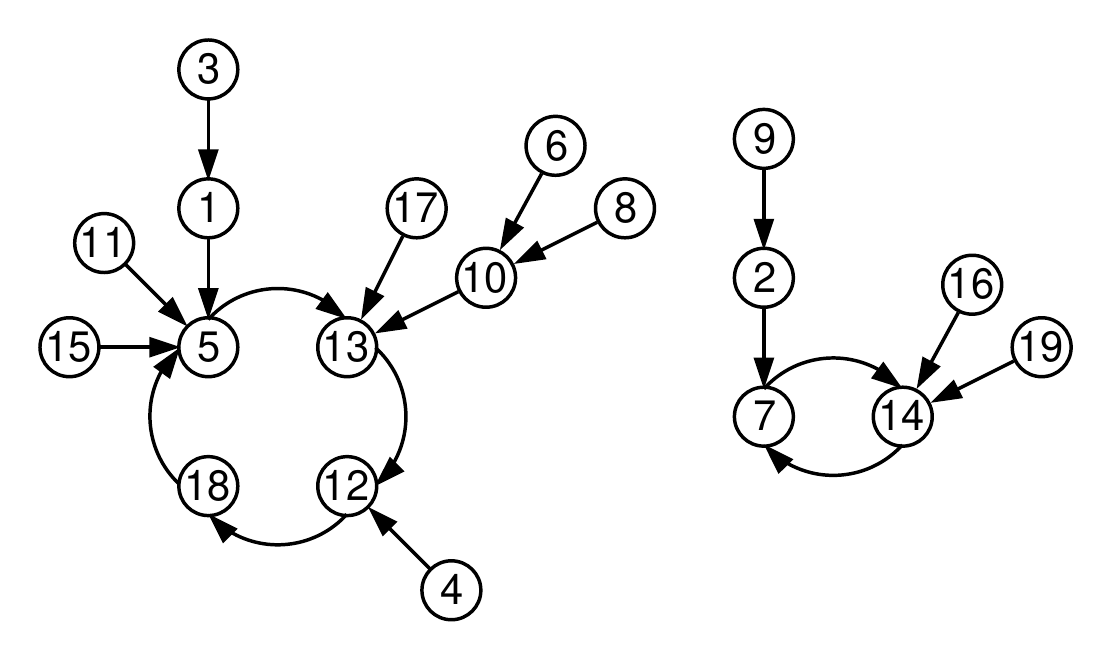}
\includegraphics[height=3.5cm]{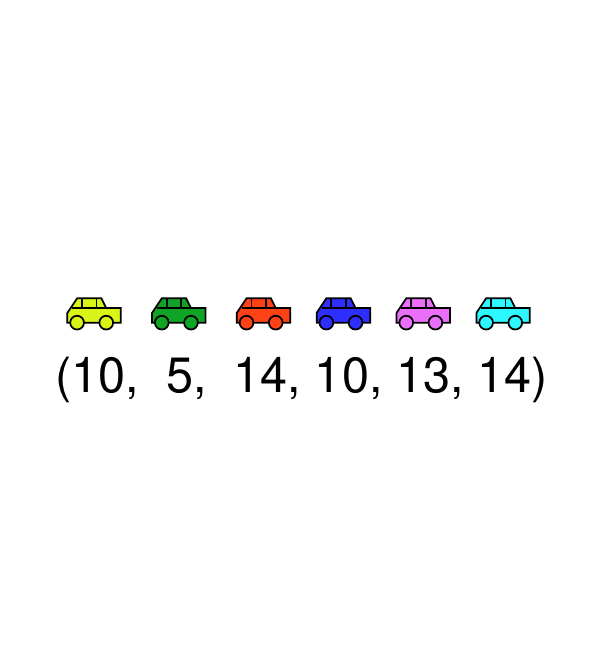}\\[-5mm]
$\Downarrow$\\[-2mm]
\includegraphics[height=3.5cm]{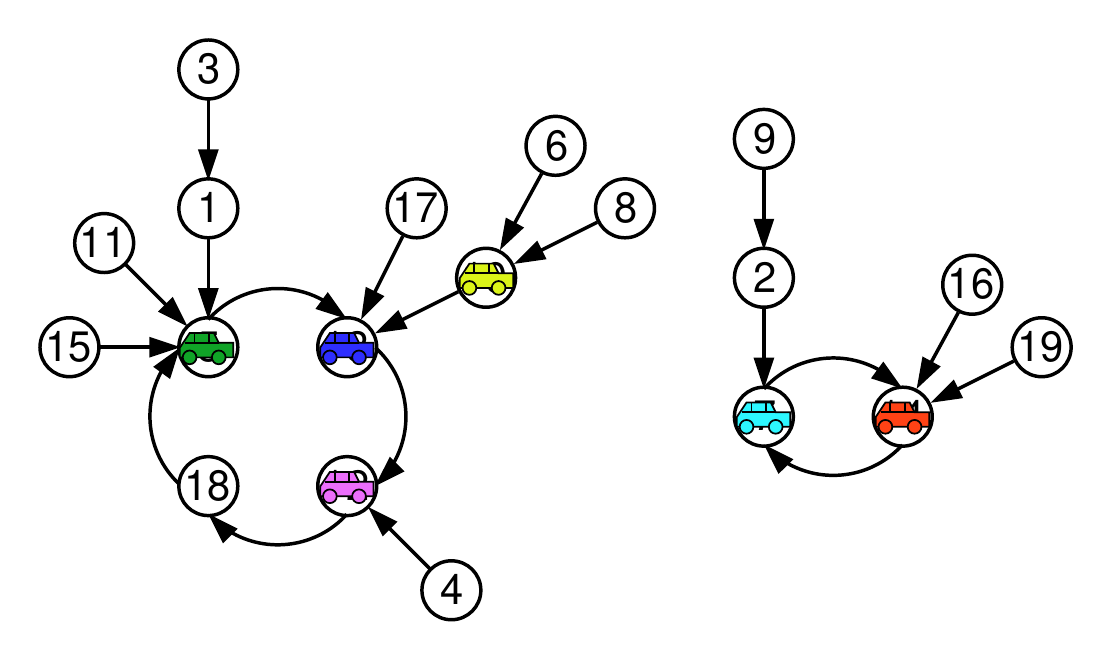}
\end{center}
\end{minipage}
\caption{The functional digraph $G_{f}$ of a $19$-mapping $f$ and a sequence $s= (10, 5, 14, 10, 13, 14)$ of addresses of preferred parking spaces for $6$ drivers. All drivers are successful, thus $(G_{f},s)$ yields a $(19,6)$-mapping parking function with output-function $\pi_{(f,s)}$ defined by the sequence $(10, 5, 14, 13, 12, 7)$ of parking positions of the drivers. When a new driver with preferred parking space $7$ arrives, he is not able to park, thus $(10, 5, 14, 10, 13, 14, 7)$ is not a parking function for $f$.\label{fig:mappingpark}}
\end{figure}

To each $(n,m)$-tree parking function $(T,s)$ (or $(n,m)$-mapping parking function $(f,s)$), we associate the corresponding \textit{output}-function $\pi \colonequals \pi_{(T,s)}$ (or $\pi \colonequals \pi_{(f,s)}$),
with $\pi : [m] \to [n]$, where $\pi(k)$ is the address of the parking space (i.e., the label of the node) in which the $k$-th driver ends up parking. Of course, $\pi$ is an injection and for the particular case $m=n$ a bijection; thus in the latter case one may speak about the output-permutation $\pi$. This notion will be useful in subsequent considerations describing characterizations and bijections for tree and mapping parking functions.

Obviously, both concepts of parking functions for trees and mappings, respectively, generalize ordinary parking functions: first, each ordinary parking function on $[n]$ can be identified with a parking function for the linear tree (i.e., the chain) $1-2- \cdots - n$ with root $n$, and second, each parking function for a rooted labelled tree $T$ of size $n$ can be identified with a parking function for the functional digraph, which is obtained from $T$ by adding a loop-edge to the root.

We start our studies of tree and mapping parking functions by giving some of their basic properties and characterizations in Section~\ref{sec:Basic}.
This extends corresponding properties and characterizations for ordinary parking functions. 
As a direct application we can characterize the extremal values for the number of parking functions with $0 \le m \le n$ drivers amongst all size-$n$ trees.
The minimal number $n^{\underline{m}} + (n-1)^{\underline{m-1}} \binom{m}{2}$ of parking functions occurs for star-like trees, where $x^{\underline{k}} \colonequals x (x-1) \cdots (x-k+1)$ is the notation used throughput this paper for the falling factorials.
The maximal number of parking functions $(n+1-m) (n+1)^{m-1}$ occurs for linear trees. 
Of course, for mappings this problem is trivial:
The minimal number $n^{\underline{m}}$ of parking functions occurs for the case where $f$ is the identity (only sequences of distinct addresses are  parking functions).
Conversely, the maximal number $n^{m}$ of parking functions occurs whenever $f$ is a cyclic permutation (every sequence of addresses is a parking function).

The main focus of this paper lies on the exact and asymptotic enumeration of the total number of $(n,m)$-tree parking functions and $(n,m)$-mapping parking functions, respectively. Let us denote by $\mathcal{T}$ the combinatorial family of rooted unordered labelled trees (so-called Cayley trees), where unordered means that we do not impose any left-to-right ordering of the subtrees of a node in the tree.
We may thus assume that to each node in the tree a (possibly empty) set of children is attached.
Furthermore, $\mathcal{T}_{n} \colonequals \{T \in \mathcal{T} : |T| = n\}$ denotes the family of size-$n$ Cayley trees. Moreover, let $\mathcal{M} \colonequals \bigcup_{n \ge 0} \mathcal{M}_{n}$ and $\mathcal{M}_{n} \colonequals \{f : [n] \to [n]\}$ denote the combinatorial family of mappings and $n$-mappings, respectively. Sections~\ref{sec:drivers_equal_nodes}-\ref{sec:general_case} are then devoted to a study of the exact and asymptotic behaviour of the quantities
\begin{align*}
  F_{n,m} & \colonequals \lvert \{(T,s) \: : \: \text{$T \in \mathcal{T}_{n}, s \in [n]^{m}$ and $s$ a parking function for $T$}\} \rvert,\\
	M_{n,m} & \colonequals |\{(f,s) \: : \: \text{$f \in \mathcal{M}_{n}, s \in [n]^{m}$ and $s$ a parking function for $f$}\}|,
\end{align*}
counting the total number of $(n,m)$-tree parking functions and $(n,m)$-map\-ping parking functions, respectively.

In order to get exact enumeration results we use suitable combinatorial decompositions of the objects, which give recursive descriptions of the quantities of interest. The recurrences occurring can be treated by a generating functions approach yielding partial differential equations.
These differential equations allow for implicit characterizations of the generating functions via the solution of a certain functional equation (conceptually, such a treatment is related to \cite{KubPan2010}). Exact counting formul{\ae} are then obtained by applying the Lagrange inversion formula~\cite{Stanley1997}. This treatment is divided into two main steps: first, in Section~\ref{sec:drivers_equal_nodes} we treat the important particular case $m=n$, i.e., we consider parking functions where the number of drivers is equal to the number of parking spaces. A combinatorial decomposition with respect to the last empty parking space before the final driver appears is the starting point for the exact enumeration of $F_{n} \colonequals F_{n,n}$ and $M_{n} \colonequals M_{n,n}$. Asymptotic results can be obtained easily by applying standard singularity analysis of generating functions~\cite{flajolet2009analytic}.

The general case in which the number of drivers $m$ is less than or equal to the number of parking spaces $n$ is then treated in Section~\ref{sec:general_case}.
Here a decomposition of the objects with respect to the free parking space with largest label in the final configuration is applied. From the exact results for $F_{n,m}$ and $M_{n,m}$ it follows somewhat surprisingly that $M_{n,m} = n F_{n,m}$, for $1 \le m \le n$. 
Of course, the numbers $T_{n} \colonequals |\mathcal{T}_{n}| = n^{n-1}$ of size-$n$ Cayley trees and the numbers $|\mathcal{M}_{n}| = n^{n}$ of $n$-mappings themselves satisfy such a relationship.
However, standard constructions such as Pr\"{u}fer codes do not seem to give a simple explanation, why this carries over to the total number of parking functions. 
In Section~\ref{sec:drivers_equal_nodes}-\ref{sec:general_case} we construct a bijection which maps each triple $(T, s, w)$, with $(T,s)$ an $(n,m)$-tree parking function and $w$ a node in $T$, to an $(n,m)$-mapping parking function $(f,s)$  and thus implies and explains the stated relation.
Note that indeed $s$ remains fixed in this correspondence.

To give a complete picture of the asymptotic behaviour of $M_{n,m}$ (and thus also $F_{n,m}$) depending on the growth of $m$ w.r.t.\ $n$ requires a more detailed study using saddle point methods. 
We consider the probability $p_{n,m} \colonequals M_{n,m}/n^{n+m} = F_{n,m}/n^{n+m-1}$ that a randomly chosen pair $(f,s)$ of an $n$-mapping $f$ and a sequence $s$ of $m$ addresses is indeed a parking function and thus the probability that all drivers are successful.
For $m \sim \frac{n}{2}$ there occurs a phase change behaviour in this probability:
If $\frac{m}{n} < \frac{1}{2} - \epsilon$, then there is asymptotically a positive probability that all drivers can park successfully, whereas for $\frac{m}{n} > \frac{1}{2} + \epsilon$ the probability that all drivers are successful is exponentially small. 
Qualitatively, the transient behaviour at $m \sim \frac{n}{2}$ is the same as observed previously in other combinatorial contexts, such as, e.g., in the analysis of random graphs during the phase where a giant component has not yet emerged.
See \cite{BanFlaSchaefSor2001,FlaKnuPit1989} or \cite[Ch.~VIII.10.]{flajolet2009analytic}.

In Section~\ref{sec:further_research} we conclude this paper by giving some remarks on open problems and possible further research directions.

\section{Basic properties of parking functions for trees and mappings\label{sec:Basic}}

In this section we will state and prove some basic facts on parking functions for trees and mappings. 
The following notation will turn out to be useful: Given an $n$-mapping $f$, we define a binary relation $\preceq_{f}$ on $[n]$ via
\[
i \preceq_f j: \Longleftrightarrow \exists k \in \mathbb{N}: f^k(i)=j.
\]
Thus $i \preceq_{f} j$ holds if there exists a directed path from $i$ to $j$ in the functional digraph $G_{f}$, and we say that $j$ is a successor of $i$ or that $i$ is a predecessor of $j$.
In this context a one-way street represents a total order, a tree represents a certain partial order, where the root node is the maximal element (to be precise, a partially ordered set with maximal element, where every interval is a chain - this is also called tree in set theory) and a mapping represents a certain pre-order (i.e., binary relation that is transitive and reflexive).

Furthermore, the combinatorial structure of the functional digraph $G_{f}$ of an arbitrary mapping function $f$ is well known \cite{flajolet2009analytic}: the weakly connected components of $G_{f}$ are cycles of rooted labelled trees.
That is, each connected component consists of rooted labelled trees (with edges oriented towards the root nodes) whose root nodes are connected by directed edges such that they form a cycle (see Figure~\ref{fig:mappingpark} for an example). We call a node $j$ lying on a cycle, i.e., for which there exists a $k \ge 1$ such that $f^{k}(j) = j$, a cyclic node.

\subsection{Changing the order in a parking function}

For ordinary parking functions the following holds:
changing the order of the elements of a sequence does not affect its property of being a parking function or not. This fact can easily be generalized to parking functions for mappings (which might also be trees).

\begin{lemma}
A function $s : [m] \rightarrow [n]$ is a parking function for a mapping $f : [n] \to [n]$ if and only if $s \circ \sigma$ is a parking function for $f$ for any permutation $\sigma$ on $[m]$.
\end{lemma}

\begin{proof}
Since each permutation $\sigma$ on $[m]$ can be obtained by a sequence of transpositions of consecutive elements, i.e., $\sigma = \tau_{r} \circ \tau_{r-1} \circ \cdots \circ \tau_{1}$, with $\tau_{i} = (k_{i} \: k_{i}+1)$, $1 \le k_{i} \le m-1$, $1 \le i \le r$, it suffices to prove the following: if $s$ is a parking function for $f$, then $s \circ \sigma$ is a parking function for $f$ for any transposition $\sigma$ of consecutive elements, i.e., for any permutation $\sigma$ on $[m]$ that swaps two consecutive elements and leaves the other elements fixed. The statement for general $\sigma$ follows from this by iteration.

Thus, let $s : [m] \rightarrow [n]$ be a parking function for $f : [n] \to [n]$ and $s' = s \circ \sigma$, where $\sigma = (k \: k+1)$, with $1 \le k \le m-1$, i.e., $\sigma(k) = k + 1$,
$\sigma(k + 1) = k$, and $\sigma(j) = j$ otherwise. In other words, $s'$ is the parking sequence obtained from $s$ by changing the order of the $k$-th and the $(k+1)$-th car.

In the following the mapping $f$ is fixed and we denote by $\pi_s=\pi_{(f,s)}$ the output-function of the parking function $s$ and consider the parking paths of the drivers: the path $y_{j} = s_{j} \leadsto \pi_{s}(j)$ denotes the parking path of the $j$-th driver of $s$ in the mapping graph $G_{f}$ starting with the preferred parking space $s_{j}$ and ending with the parking position $\pi_{s}(j)$.
In order to show that $s'$ is still a parking function, we have to show that all cars can successfully be parked using $s'$. In the following we do this and also determine the output-function $\pi_{s'}$ of $s'$ (and thus the parking paths $y_{j}' = s_{j}' \leadsto \pi_{s'}(j)$, $1 \le j \le m$, of the drivers of $s'$).

Clearly, the parking paths of the first $(k-1)$ cars are not affected by the swapping of the $k$-th and the $(k+1)$-th car and we have that $\pi_{s'}(1)=\pi_{s}(1)$, \dots, $\pi_{s'}(k-1)=\pi_{s}(k-1)$.
For the $k$-th and the $(k+1)$-th car we will distinguish between two cases according to the parking paths $y_{k} = s_{k} \leadsto \pi_{s}(k)$ and $y_{k+1} = s_{k+1} \leadsto \pi_{s}(k+1)$.
\begin{enumerate}
\item[$(a)$] Case $y_{k} \cap y_{k+1} = \emptyset$: Since the parking paths $y_{k}$ and $y_{k+1}$ are disjoint,
swapping the $k$-th and the $(k+1)$-th car simply also swaps the corresponding parking paths, i.e., $y_{k}' = y_{k+1}$ and $y_{k+1}' = y_{k}$, and in particular $\pi_{s'}(k) = \pi_{s}(k+1)$ and $\pi_{s'}(k+1) = \pi_{s}(k)$.
\item[$(b)$] Case $y_{k} \cap y_{k+1} \neq \emptyset$: Let us denote by $v$ the first node in the path $y_{k}$ that also occurs in $y_{k+1}$. Then, according to the parking procedure, the parking paths can be decomposed as follows: 
\begin{equation*}
  y_{k} = s_{k} \leadsto v \leadsto \pi_{s}(k) \quad \text{and} \quad y_{k+1} = s_{k+1} \leadsto v \leadsto \pi_{s}(k) \leadsto \pi_{s}(k+1),
\end{equation*}
i.e. $\pi_{s}(k+1)$ is a proper successor of $\pi_{s}(k)$, i.e., $\pi_{s}(k) \prec_{f} \pi_{s}(k+1)$. Thus, when swapping the $k$-th and the $(k+1)$-th car, both cars can also be parked yielding the parking paths
\begin{equation*}
  y_{k}' = s_{k+1} \leadsto v \leadsto \pi_{s}(k) \quad \text{and} \quad y_{k+1}' = s_{k} \leadsto v \leadsto \pi_{s}(k) \leadsto \pi_{s}(k+1).
\end{equation*}
In particular, we obtain $\pi_{s'}(k) = \pi_{s}(k)$ and $\pi_{s'}(k+1) = \pi_{s}(k+1)$.
\end{enumerate}
Thus, in any case we get $\{\pi_{s}(k), \pi_{s}(k+1)\} = \{\pi_{s'}(k), \pi_{s'}(k+1)\}$, and consequently swapping the $k$-th and the $(k+1)$-th car does not change the parking paths of the subsequent cars and we obtain $\pi_{s'}(j)=\pi_{s}(j)$, $k+2 \le j \le m$.
So all drivers in the parking sequence $s'$ are successful and $s'$ is indeed a parking function for $f$.
\end{proof}

\subsection{Alternative characterizations of parking functions}

Using the fact that reordering the elements of a function does not have any influence on whether it is a parking function or not, one can obtain the following well-known simpler characterization of ordinary parking functions $s: [n] \to [n]$ (see, e.g., \cite{Stanley1997}): A sequence $s \in [n]^{n}$ is a parking function if and only if it is a major function, i.e., the sorted rearrangement $s'$ of the sequence $s$ satisfies:
\begin{equation*}
  s_{j}' \leq j, \quad \text{for all $j \in [n]$}.
\end{equation*}

This statement can be reformulated in the following way: A sequence $s \in [n]^{n}$ is a parking function if and only if for every $j \in [n]$, it does not contain more than $(n-j)$ elements that are larger than $j$. Or again in other words, there must be at least $j$ elements that are not larger than $j$:
\begin{equation}\label{eqn:charactPF}
  |\left\lbrace k \in [n]: s_{k} \leq j \right\rbrace| \geq j, \quad \text{for all $j \in [n]$}.
\end{equation}

Now, this characterization of parking functions can easily be generalized to parking functions for trees and mappings.
Indeed, in \eqref{eqn:charactPF} we merely need to replace the $\leq$ and $\geq$ relation which come from the order on the elements $1, 2, \ldots, n$ represented by the one-way street of length $n$ by the binary relation given by the respective tree or mapping.
The following characterization of $(n,n)$-mapping parking functions (which might also be trees) is now possible. 
See Figure~\ref{fig:ex_characterization} for an illustration of Lemma~\ref{lem:char_parking} with a tree.

\begin{figure}
\begin{center}
\begin{tabular}[b]{|c|c|c|c|} \hline 
$j$ & $p(j)$ & $q_{1}(j)$ & $q_{2}(j)$ \\\hline
1 & 2& 3& 3\\\hline
2 & 2& 2& 2\\\hline
3 & 4& 5& 4\\\hline
4 & 1& 1& 3\\\hline
5 & 7& 7& 7\\\hline
6 & 1& 0& 1\\\hline
7 & 1& 1& 1\\\hline
\end{tabular}
\includegraphics[height=2.8cm]{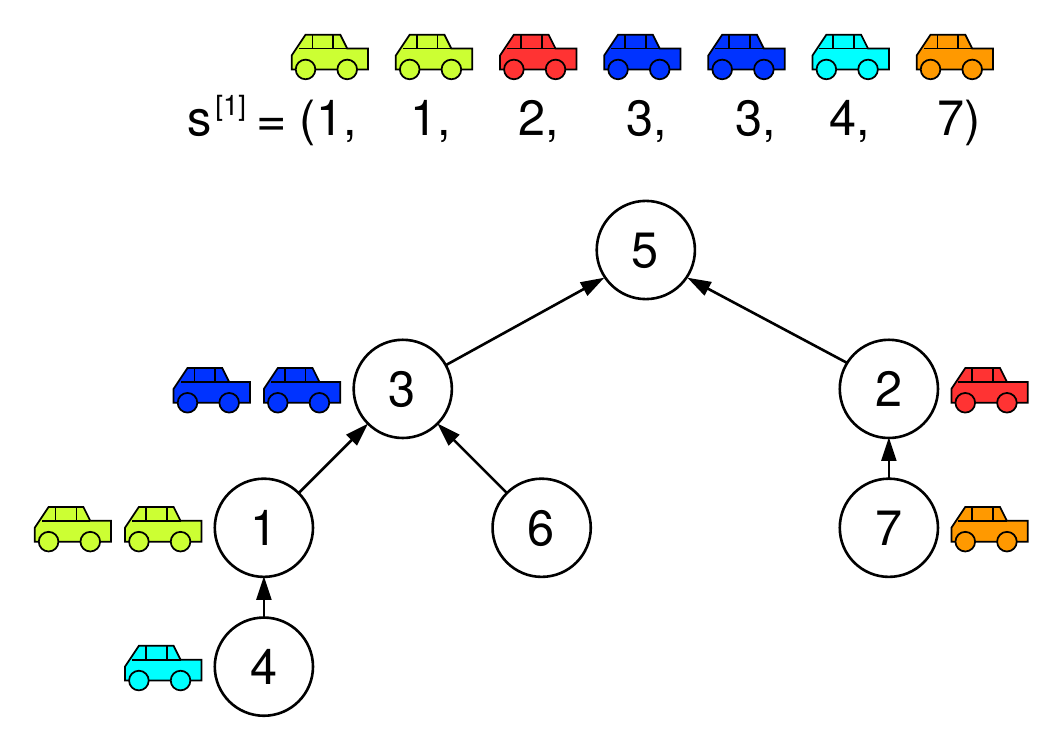}
\includegraphics[height=2.8cm]{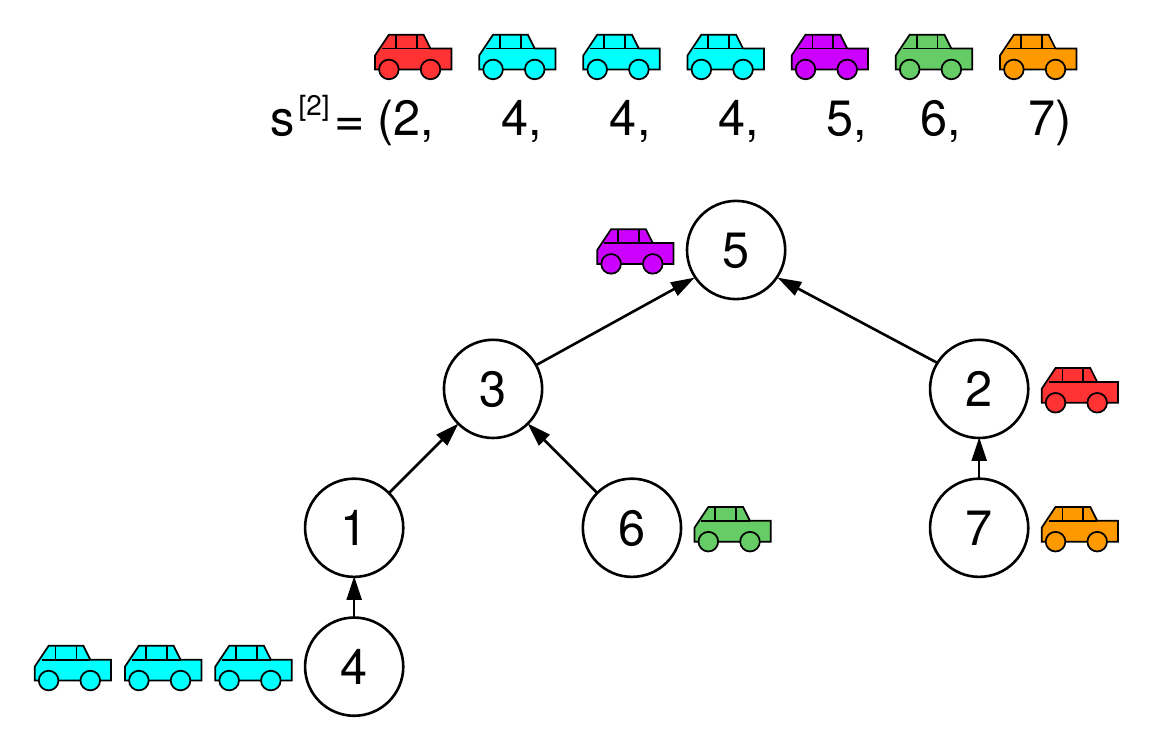}
\end{center}
\caption{Exemplifying the characterization of generalized parking functions given in Lemma~\ref{lem:char_parking} for a tree $T$ of size $7$. The (sorted) sequence $s^{[1]} = (1,1,2,3,3,4,7)$ represented on the left-hand-side does not give a parking function for $T$, whereas the (sorted) sequence $s^{[2]} = (2,4,4,4,5,6,7)$ represented on the right-hand-side does. This can be seen in the following way, where we denote by $q_{1}(j)$ and $q_{2}(j)$ respectively, the quantity $q(j)$ for $s^{[1]}$ and $s^{[2]}$, respectively: for the sequence $s^{[1]}$ we have $q_{1}(6)=0<1=p(6)$ thus violating the condition, whereas each element $p(j)$ is smaller or equal to the corresponding element $q_{2}(j)$.\label{fig:ex_characterization}}
\end{figure}

\begin{lemma}\label{lem:char_parking}
Given an $n$-mapping $f$ and a sequence $s \in [n]^{n}$, let $p(j)$ denote the number of predecessors of $j$, i.e.,  $p(j)\colonequals|\left\lbrace i \in [n]: i \preceq_f j \right\rbrace|$  and $q(j)$ denote the number of drivers whose preferred parking spaces are predecessors of $j$, i.e., $q(j)\colonequals|\left\lbrace k \in [n]: s_{k} \preceq_f j \right\rbrace|$ .
Then $s$ is a mapping parking function for $f$ if and only if
\begin{equation*}
  q(j) \ge p(j), \quad \text{for all $j \in [n]$}.
\end{equation*}
\end{lemma}

\begin{proof}
First, assume that $q(j) < p(j)$ holds for some $j \in [n]$. Let us denote by $P(j) \colonequals \{i \in [n] : i \preceq_{j} j\}$ the set of predecessors of $j$. Obviously, if $s_{k} \not\in P(j)$ then the $k$-th driver will not get a parking space in $P(j)$. Thus, at most $q(j) = |\{k \in [n]: s_{k} \in P(j)\}|$ drivers are able to park in $P(j)$. In other words, at least $p(j) - q(j) > 0$ parking spaces in $P(j)$ remain free. Since there is the same number of cars and of parking spaces, this means that at least one driver will not be able to park successfully. Thus $s$ is not a parking function for $f$.

Next, assume that $q(j) \ge p(j)$ holds for all $j \in [n]$. 
It will be sufficient to show the following: Let $s \in [n]^{n}$ be a parking sequence such that $q(j) \ge p(j)$ for some $j$, then node $j$ will be occupied after applying the parking procedure. Due to the assumption above we may then conclude that all nodes will be occupied after applying the parking procedure and thus that all $n$ drivers are successful, which means that $s$ indeed is a parking function for $f$.

To prove the assertion above we distinguish between two cases:
\begin{enumerate}
\item[$(a)$] $j$ is not a cyclic node: Then the set $P(j)$ of predecessors of $j$ is a tree. If there is a driver $k$ with preferred parking space $s_{k} = j$ then in any case node $j$ will be occupied. Thus let us assume that $s_{k} \neq j$, for all $k \in [n]$. Let us further assume that $i_{1}, \dots, i_{r}$ are the preimages of $j$, i.e., $f(i_{t})=j$, $1 \le t \le r$. Since $q(j) \ge p(j)$, but no driver wishes to park at $j$, it holds that there exists a preimage $i_{\ell}$, such that $q(i_{\ell}) > p(i_{\ell})$. This means that there must be at least one driver appearing in $P(i_{\ell})$ that is not able to get a parking space in $P(i_{\ell})$ and thus, according to the parking procedure, he has to pass the edge $(i_{\ell}, j)$. Consequently, node $j$ will be occupied.
\item[$(b)$] $j$ is a cyclic node: 
If the edge $(j, f(j))$ is passed by some driver during the application of the parking procedure this necessarily implies 
that the node $j$ is occupied. 
Thus let us assume that the edge $(j, f(j))$ will never be passed while carrying out the parking procedure. Then we may remove the edge $(j, f(j))$ from $G_{f}$ without influencing the outcome of the parking procedure. By doing so, node $j$ becomes a non-cyclic node and according to case $(a)$, node $j$ will be occupied.
\end{enumerate}
\end{proof}

Now let us turn to parking functions, where the number of drivers does not necessarily coincide with the number of parking spaces.
It is well-known and easy to see that a parking sequence $s : [m] \to [n]$ on a one-way street is a parking function if and only if 
\begin{equation}
  |\left\lbrace k \in [m]: s_{k} \ge j \right\rbrace| \le n-j+1, \quad \text{for all $j \in [n]$}.
\end{equation}
This characterization can be generalized to $(n,m)$-tree parking functions as follows.

\begin{lemma}\label{lem:char_parking_general}
Given a rooted labelled tree $T$ of size $|T|=n$ and a sequence $s \in [n]^{m}$.
Then $s$ is a tree parking function for $T$ if and only if
\begin{equation*}
  \lvert \left\lbrace k \in [m]: s_{k} \in T' \right\rbrace \rvert \le |T'|, \quad \text{for all subtrees $T'$ of $T$ containing $\rt(T)$}.
\end{equation*}
\end{lemma}

Recall that $T'$ is called a subtree of $T$ if $T'$ is a subgraph of $T$ that is a tree itself.

\begin{proof}
First, let us assume that there exists a subtree $T'$ of $T$ containing $\rt(T)$, such that $q(T') \colonequals |Q(T')| \colonequals |\left\lbrace k \in [m]: s_{k} \in T' \right\rbrace| > |T'|$. 
Clearly, the possible parking spaces for any driver $k$ with preferred parking space $s_{k} \in T'$ form a subset of $T'$.
Here, the number of such drivers $q(T')$ exceeds the amount of parking spaces $|T'|$ and thus at least one of the drivers in $Q(T')$ will be unsuccessful. Thus $s$ is not a parking function for $T$.

Next, let us assume that $s$ is not a parking function for $T$. Let us further assume that $\ell \in [m]$ is the first unsuccessful driver in $s$ when applying the parking procedure. We consider the situation after the first $\ell$ drivers: Define $T'$ as the maximal subtree of $T$ containing $\rt(T)$ and only such nodes that are occupied by one of the first $\ell-1$ cars. Of course, since the $\ell$-th driver is unsuccessful, the $\rt(T)$ has to be occupied by one of the first $\ell-1$ cars, anyway. Due to the maximality condition of $T'$, it holds that each driver $k$ that has parked in $T'$ must have had his preferred parking space in $T'$, thus $|\{k \in [\ell-1] : s_{k} \in T'\}| = |T'|$. Since the $\ell$-th driver is unsuccessful, his preferred parking space is also in $T'$, yielding $|\{k \in [\ell] : s_{k} \in T'\}| > |T'|$. Of course, this implies $|\{k \in [m] : s_{k} \in T'\}| > |T'|$, for the subtree $T'$ of $T$ containing $\rt(T)$.
\end{proof}

We remark that the characterization above could be also extended to mapping parking functions (where one has to consider connected subgraphs of $G_{f}$ containing all cyclic nodes of the respective component).
Since we will not make use of it in the remainder of this paper we omit it here.

\subsection{Extremal cases for the number of parking functions}

Given an $n$-mapping $f : [n] \to [n]$ (which might be a tree), let us denote by $S(f,m)$ the number of parking functions $s \in [n]^{m}$ for $f$ with $m$ drivers. So far we are not aware of enumeration formul{\ae} for the numbers $S(f,m)$  for general $f$.
In Sections~\ref{sec:drivers_equal_nodes} and \ref{sec:general_case} however, we will compute the total number $F_{n,m} \colonequals \sum_{T \in \mathcal{T}_{n}}S(T,m)$ and $M_{n,m} \colonequals \sum_{f \in \mathcal{M}_{n}}S(f,m)$ of $(n,m)$-tree and $(n,m)$-mapping parking functions, respectively. 

Before continuing, we first state the obvious fact that isomorphic mappings (or trees) yield the same number of mapping (or tree) parking functions, since one simply has to adapt the preferred parking spaces of the drivers according to the relabelling.

\begin{prop}
  Let $f$ and $f'$ two isomorphic $n$-mappings, i.e., there exists a bijective function $\sigma : [n] \to [n]$, such that $f' = \sigma \circ f \circ \sigma^{-1}$. Then for $0 \le m \le n$ it holds
	\begin{equation*}
	  S(f,m) = S(f',m).
	\end{equation*}
\end{prop}

\begin{proof}
First note that the corresponding functional digraphs $G_{f} = ([n],E)$ and $G_{f'} = ([n],E')$ are isomorphic in the graph theoretic sense, since
\begin{align*}
  e = (i,j) \in E & \Leftrightarrow j=f(i)  \Leftrightarrow \sigma(j)= \sigma(f(i)) \Leftrightarrow \sigma(j)= f'(\sigma(i)) \\
  & \Leftrightarrow \sigma(e) = (\sigma(i),\sigma(j)) \in E'.
\end{align*}
It is then an easy task to show via induction that a function $s = (s_{1}, \dots, s_{m}) \in [n]^{m}$ is a parking function for $f$ if and only if $s'; = \sigma \circ s = (\sigma(s_{1}), \dots, \sigma(s_{m}))$ is a parking function for $f'$.
\end{proof}

In the following we consider the extremal cases of $S(f,m)$. Obviously, each surjective function $s \in [n]^{m}$ is a parking function for every mapping $f \in \mathcal{M}_{n}$, which yields the trivial bounds
\begin{equation}
  n^{\underline{m}} \le S(f,m) \le n^{m}, \quad \text{for $f \in \mathcal{M}_{n}$}.
\end{equation}
These bounds are actually tight.
Indeed, for the identity $\id_{n} : j \mapsto j$, for $j \in [n]$, we have $S(\id_{n}, m) = n^{\underline{m}}$ since no collisions may occur.
Moreover, for
\[
\cycle_{n} : j \mapsto \begin{cases} j+1, & \quad \text{for $1 \le j \le n-1$},\\ 1, & \quad \text{for $j=n$}, \end{cases}
\]
a cycle of length $n$, it holds that $S(\cycle_{n}, m) = n^{m}$.

The situation becomes more interesting when we restrict ourselves to trees. The following simple tree operation will turn out to be useful in order to identify the extremal cases. Let $T$ be a rooted labelled tree and $v$ a node of $T$.
Furthermore, let $U$ be a subtree of $T$ attached to $v$ such that $T \setminus U$ is still a tree, i.e., the graph consisting of all edges not contained in $U$ has one connected component.
For a node $w$ not contained in $U$, we denote by $\reallocate{T}{U}{v}{w}$ the tree operation of first detaching the subtree $U$ from $v$ and then attaching it to $w$.
See Figure~\ref{fig:reallocate} for an illustration. 

\begin{figure}
\begin{center}
\parbox{4cm}{
\begin{center}
\includegraphics[height=5cm]{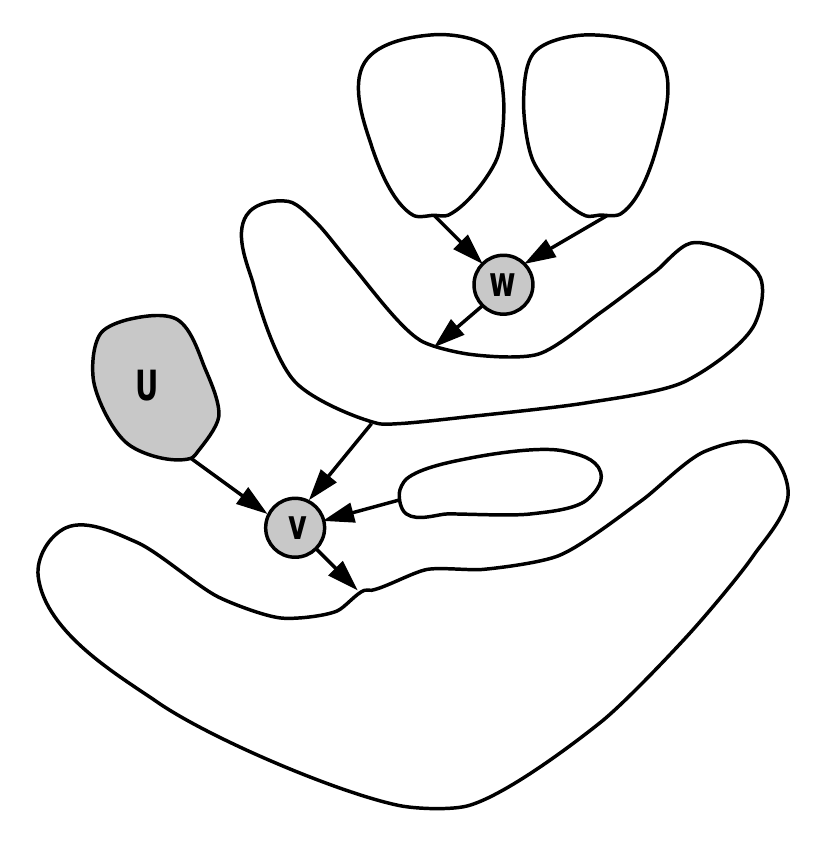}\\
{\large{$T$}}
\end{center}
}
\quad \quad \quad 
\raisebox{0cm}{$\Rightarrow$}
\quad
\parbox{5cm}{
\begin{center}
\includegraphics[height=5cm]{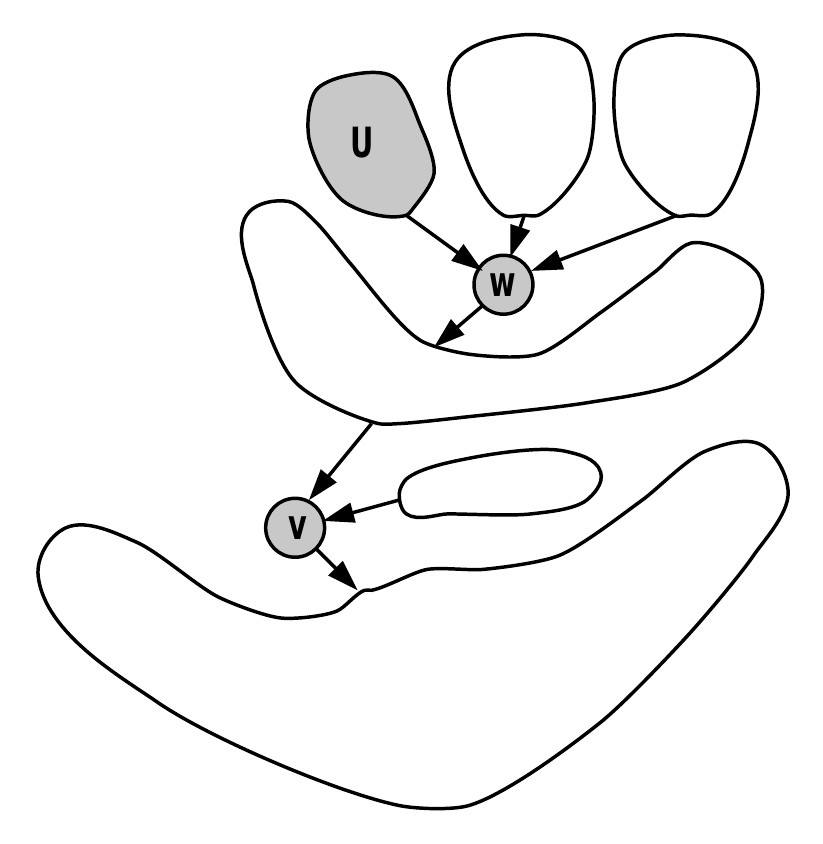}\\
{\large{$\tilde{T}$}}
\end{center}
}
\end{center}
\caption{Illustrating the tree operation of reallocating the subtree $U$ from $v$ to $w$ in $T$ which yields the tree $\tilde{T}$.
Here the nodes $v$ and $w$ satisfy $w \preceq_{T} v$, as required in the proof of Lemma~\ref{lem:reallocate_tree}.\label{fig:reallocate}}
\end{figure}

\begin{lemma}\label{lem:reallocate_tree}
Let $T$ be a rooted labelled tree and $w \preceq_{T} v$, for two nodes $v, w \in T$. Furthermore, let $U$ be a subtree of $T$ attached to $v$ that does not contain $w$ such that $T \setminus U$ is still a tree.
Let us denote by $\tilde{T}$ the tree which is obtained by reallocating $U$ from $v$ to $w$, i.e., 
\begin{equation*}
  \tilde{T} = \reallocate{T}{U}{v}{w}. 
\end{equation*}
Then it holds that
\begin{equation*}
  S(\tilde{T}, m) \ge S(T,m).
\end{equation*}
\end{lemma}

\begin{proof}
 By applying Lemma~\ref{lem:char_parking_general} we will show that each parking function $s \in [n]^{m}$ for $T$ is also a parking function for $\tilde{T}$. For this purpose, let $s$ be a parking function for $T$ and consider a subtree $\tilde{T}'$ of $\tilde{T}$ containing the root of $\tilde{T}$. Note that by construction $\rt(\tilde{T}) = \rt(T)$. We distinguish between two cases to show that $\left|\left\lbrace k \in [m]: s_{k} \in \tilde{T}' \right\rbrace\right| \le |\tilde{T}'|$.
\begin{itemize}
\item[$(a)$] Case $U \cap \tilde{T}' = \emptyset$: In this case $\tilde{T}'$ is also a subtree of $T$ containing $\rt(T)$. Since $s$ is a parking function for $T$ it holds that $\left|\left\lbrace k \in [m]: s_{k} \in \tilde{T}' \right\rbrace\right| \le |\tilde{T}'|$ according to Lemma~\ref{lem:char_parking_general}.
\item[$(b)$] Case $U \cap \tilde{T}' = R \neq \emptyset$: According to the construction of $\tilde{T}$, $R$ is a subtree of $U$ that is attached to node $w$, which is itself a predecessor of $v$. 
Within the tree $\tilde{T}'$, let us reallocate the subtree $R$ from $w$ to $v$.
Then the resulting tree $T' \colonequals \reallocate{\tilde{T}'}{R}{w}{v}$ is a subtree of $T$ containing $\rt(T)$. According to Lemma~\ref{lem:char_parking_general} it holds that $|\left\lbrace k \in [m]: s_{k} \in T' \right\rbrace| \le |T'|$.
Since $T'$ and $\tilde{T}'$ have equal size and the nodes in the corresponding trees have the same labels, this also implies that $\left|\left\lbrace k \in [m]: s_{k} \in \tilde{T}' \right\rbrace\right| \le |\tilde{T}'|$.
\end{itemize}
\end{proof}

With this lemma we can easily obtain tight bounds on $S(T,m)$.
\begin{theorem}
  Let $\str_{n}$ be the rooted labelled tree of size $n$ with root node $n$ and the nodes $1, 2, \dots, n-1$ attached to it. Furthermore let $\chain_{n}$ be the rooted labelled tree of size $n$ with root node $n$ and node $j$ attached to node $(j+1)$, for $1 \le j \le n-1$. Then, for any rooted labelled tree $T$ of size $n$ it holds
	\begin{equation}\label{eqn:extremal_cases_FTm}
	  S(\str_{n},m) \le S(T,m) \le S(\chain_{n},m),
	\end{equation}
	yielding the bounds
	\begin{equation}\label{eqn:bounds_FTm}
	  n^{\underline{m}} + \binom{m}{2} (n-1)^{\underline{m-1}} \le S(T,m) \le (n-m+1) (n+1)^{m-1}, \quad \text{for $0 \le m \le n$}.
	\end{equation}
\end{theorem}

\begin{proof}
Each tree $T$ of size $n$ can be constructed from a tree $T_{0}$, which is isomorphic to $\str_{n}$, by applying a sequence of reallocations $T_{i+1} \colonequals \reallocate{T_{i}}{U_{i}}{v_{i}}{w_{i}}$, with $w_{i} \preceq_{T_{i}} v_{i}$, for $0 \le i \le k$, with $k \ge 0$.
Furthermore, starting with $T =: \tilde{T}_{0}$, there always exists a sequence of reallocations $\tilde{T}_{i+1} \colonequals \reallocate{\tilde{T}_{i}}{\tilde{U}_{i}}{\tilde{v}_{i}}{\tilde{w}_{i}}$, with $\tilde{w}_{i} \preceq_{\tilde{T}_{i}} \tilde{v}_{i}$, for $0 \le i \le \tilde{k}$, with $\tilde{k} \ge 0$, such that the resulting tree is isomorphic to $\chain_{n}$. Thus, equation~\eqref{eqn:extremal_cases_FTm} follows immediately from Lemma~\ref{lem:reallocate_tree}.

The upper bound in \eqref{eqn:bounds_FTm} is the well-known formula for the number parking functions in a one-way street (which corresponds to the number of tree parking functions for $\chain_{n}$). For the lower bound one has to compute the number of parking functions with $m$ drivers for $\str_{n}$: there are only two possible cases, namely either $s$ is injective or exactly two drivers have the same non-root node as preferred parking space, whereas all remaining drivers have different non-root nodes as preferred parking spaces. Elementary combinatorics yields the stated result.
\end{proof}

\section{Total number of parking functions: number of drivers coincides with number of parking spaces\label{sec:drivers_equal_nodes}}

In this section, we  consider  the total number of parking functions for trees and mappings for the case that the number of drivers $m$ is equal to the number of parking spaces (i.e., nodes) $n$. As for ordinary parking functions, this case is not only interesting in its own.
It will also occur during the studies of the general case via initial values for recurrence relations.

\subsection{Tree parking functions\label{ssec:Tree_parking_functions_m=n}}

We study the total number $F_{n} \colonequals F_{n,n}$ of $(n,n)$-tree parking functions, i.e., the number of pairs $(T,s)$, with $T \in \mathcal{T}_{n}$ a Cayley tree of size $n$ and $s \in [n]^{n}$ a parking sequence of length $n$ for the tree $T$, such that all drivers are successful.
To obtain a recursive description of the total number $F_{n}$ of tree parking functions we use the decomposition of a Cayley tree $T \in \mathcal{T}_{n}$ w.r.t.\ the last empty node.
We thus consider the situation just before the last driver starts searching a parking space. 

\begin{figure}
\begin{minipage}[hbt]{0.47\textwidth}
\includegraphics[scale=0.58]{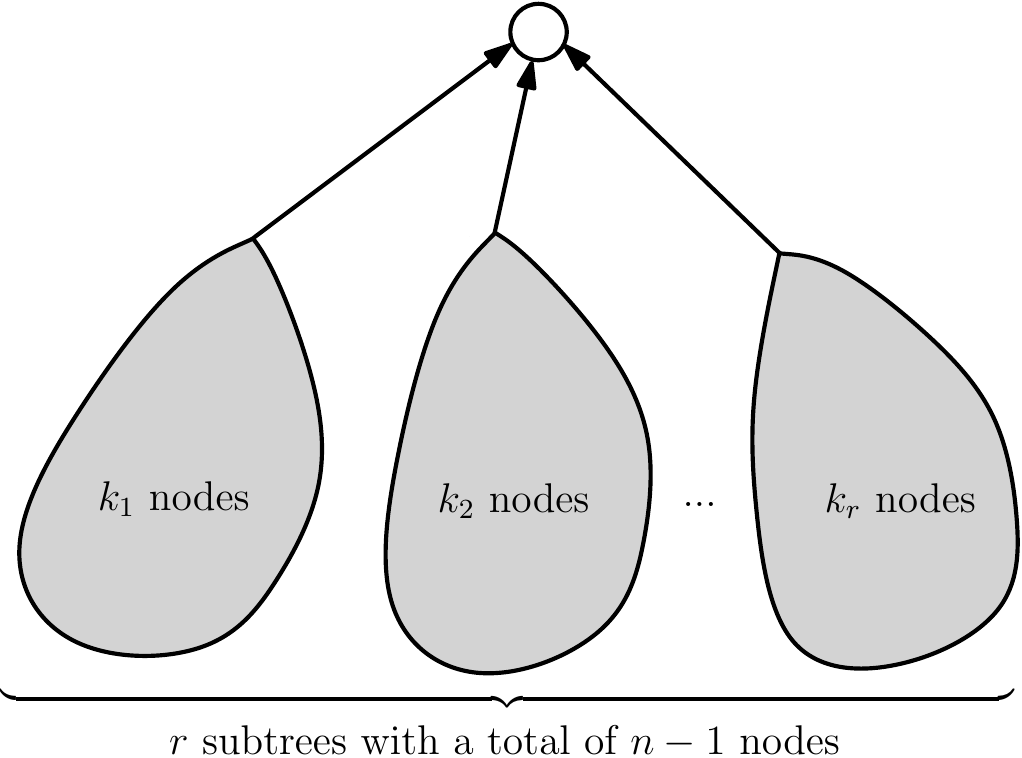}
\end{minipage}
\hspace{-0.1cm}
\begin{minipage}[hbt]{0.47\textwidth}
\includegraphics[scale=0.58]{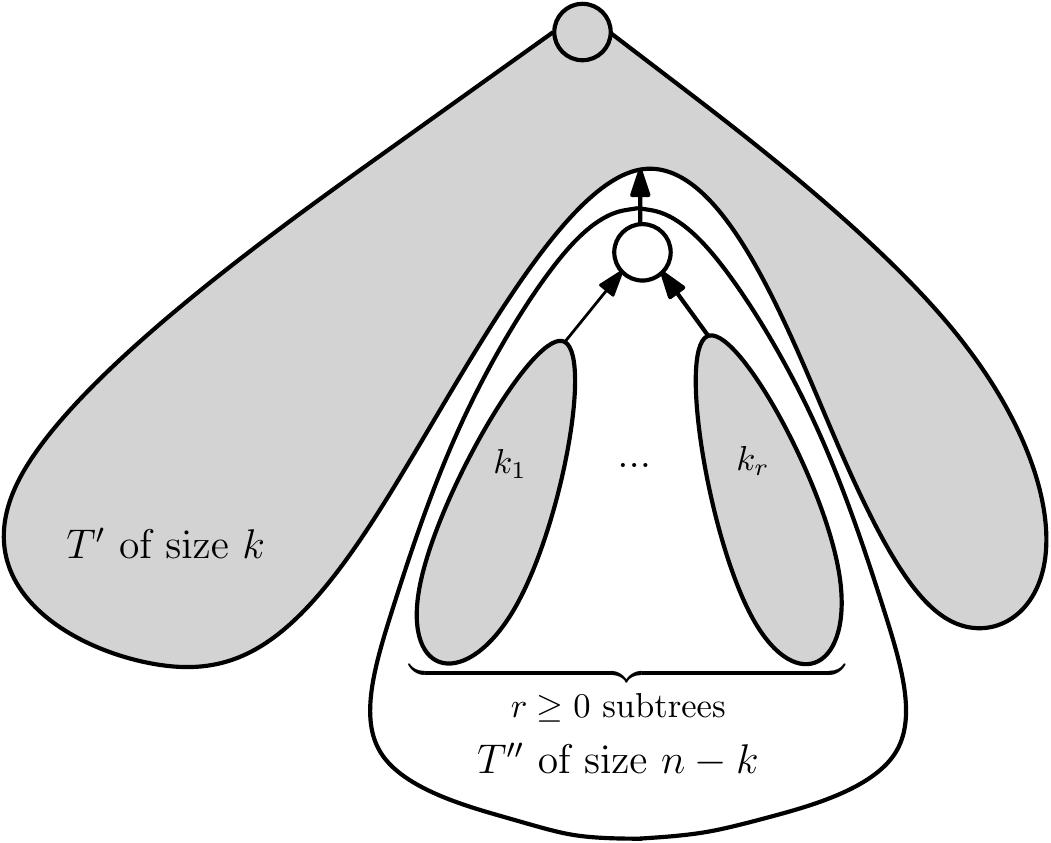}
\end{minipage}
\caption{Schematic representation of the two situations that might occur when considering parking functions with $n$ drivers for a Cayley tree with $n$ nodes. The last empty node, which is marked in white in the trees, is either the root node of the tree (see left hand side) or a non-root node (see right hand side).}
\label{fig:tree_parking_n=m}
\end{figure}

Two different situations might occur: $(i)$ the empty node is the root node of the tree $T$, or $(ii)$ the empty node is a non-root node.
See Figure~\ref{fig:tree_parking_n=m} for a schematic representation of these two situations, where case $(i)$ is depicted to the left and case $(ii)$ to the right.

In case $(i)$ the last driver will always find a free parking space regardless of the $n$ possible choices of his preferred parking space. In case $(ii)$ the last driver will only find a free parking space, if his preferred parking space is contained in the subtree (call it $T''$) rooted at the node corresponding to the free parking space. If we detach the edge linking this subtree $T''$ with the rest of the tree we get two unordered trees; let us assume the tree containing the original root of the tree (denote it with $T'$) has size $k$, whereas the remaining tree $T''$ has size $n-k$. Then there are $n-k$ choices for the preferred parking space of the last driver such that he is successful. Furthermore, it is important to take into account that, given $T'$ and $T''$, the original tree $T$ cannot be reconstructed, since there are always $k$ different trees in $\mathcal{T}_{n}$ leading to the same pair $(T',T'')$; in other words, given $T'$ and $T''$, we have $k$ choices of constructing trees $\tilde{T} \in \mathcal{T}_{n}$ by attaching the root of $T''$ to any of the $k$ nodes of $T'$. Taking into account the order-preserving relabellings of the subtrees and also the merging of the parking sequences for the subtrees, we obtain the following recursive description of $F_{n}$. Here $r$ denotes the number of subtrees of the free parking space (i.e., of the empty node), thus the factor $\frac{1}{r!}$ occurs, since each of the $r!$ orderings of the subtrees of the empty node represent the same tree.
\begin{align}
  F_{n} & = \sum_{r \ge 1} \frac{1}{r!} \sum_{\stackrel{\sum_{i=1}^r k_{i}=n-1}{k_{i} \ge 1}} F_{k_{1}} \cdot F_{k_{2}} \cdot \cdots \cdot F_{k_{r}} \binom{n}{k_{1}, k_{2}, \dots, k_{r}, 1}
  \binom{n-1}{k_{1}, k_{2}, \dots, k_{r}} n\notag\\
  & \quad \mbox{} + \sum_{r \ge 0} \frac{1}{r!} \sum_{\stackrel{k+ \sum_{i=1}^r k_{i}=n-1}{k \ge 1, k_{i} \ge 1}} F_{k} \cdot F_{k_{1}} \cdot F_{k_{2}} \cdot \cdots \cdot F_{k_{r}} \cdot \label{eqn:Fn_rec}\\
  & \qquad \qquad \quad \cdot \binom{n}{k, k_{1}, k_{2}, \dots, k_{r}, 1} \binom{n-1}{k, k_{1}, k_{2}, \dots, k_{r}} k (n-k), \quad \text{for $n \ge 2$},\notag
\end{align}
with initial value $F_{1} = 1$.
In order to treat this recurrence we introduce the following generating function
\begin{equation*}
  F(z) \colonequals \sum_{n \ge 1} F_{n} \frac{z^{n}}{(n!)^{2}}.
\end{equation*}
Then, after straightforward computations which are omitted here, \eqref{eqn:Fn_rec} can be transferred into the following differential equation:
\begin{equation}\label{eqn:Fz_DEQ}
  F'(z) = \exp(F(z)) \cdot \left(1+z F'(z)\right)^{2}, \quad F(0)=0.
\end{equation}
This differential equation can be solved by standard methods and it can be checked easily that the solution of \eqref{eqn:Fz_DEQ} is given as follows:
\begin{equation}\label{eqn:Fz_sol}
  F(z) = T(2z) + \ln\left(1-\frac{T(2z)}{2}\right).
\end{equation}
Here and in the following, $T(z) \colonequals \sum_{n \ge 1} T_{n} \frac{z^{n}}{n!} = \sum_{n \ge 1} n^{n-1} \frac{z^{n}}{n!}$
denotes the so-called tree function, i.e., the exponential generating function of the number $T_{n} = n^{n-1}$ of size-$n$ Cayley trees. In this context we note that the tree function $T(z)$ satisfies the functional equation
\begin{equation}\label{eqn:Tz_feq}
  T(z) = z e^{T(z)}
\end{equation}
and is thus related to the so-called Lambert $W$-function~\cite{flajolet2009analytic}.

We shall not extract coefficients from \eqref{eqn:Fz_sol} at this point yet, since we will soon see in Theorem~\ref{thm:link_trees_mappings_m=n} that the total number $F_{n}$ of parking functions for trees of size $n$ is directly linked to the total number $M_{n}$ of parking functions for mappings of size $n$. The latter quantity is treated in the next section and thus also yields exact and asymptotic enumeration formul{\ae} for $F_{n}$.

\subsection{Mapping parking functions\label{ssec:Mapping_parking_functions_m=n}}

Now we study the total number $M_{n} \colonequals M_{n,n}$ of $(n,n)$-mapping parking functions, i.e., the number of pairs $(f,s)$, with $f \in \mathcal{M}_{n}$ an $n$-mapping and $s \in [n]^{n}$ a parking sequence of length $n$ for the mapping $f$, such that all drivers are successful.
First, consider the well-known structure of the functional digraph $G_{f}$ of a mapping: the connected components of such a graph are cycles of Cayley trees, i.e., the root nodes of the involved Cayley trees are linked in a cyclic way.
For an example, see Figure~\ref{fig:mappingpark}: this graph consists of two connected components, which are cycles of four and two Cayley trees, respectively. It is thus natural to introduce connected mapping graphs (we simply say connected mappings) as auxiliary objects and study parking functions for them; after that the general situation can be treated easily. Let $\mathcal{C}$ and $\mathcal{C}_{n} \colonequals \{f \in \mathcal{C} : |f|=n\}$ denote connected mappings and connected $n$-mappings, respectively. Using the \textsc{Set} and \textsc{Cycle} construction for combinatorial families, mappings, connected mappings and Cayley trees are related via the symbolic equations
\begin{equation*}
  \mathcal{M} = \textsc{Set}(\mathcal{C}), \qquad \mathcal{C} = \textsc{Cycle}(\mathcal{T}).
\end{equation*}
Whereas the relation between mappings and connected mappings can be translated immediately into connections between parking functions for these objects, this is not the case for connected mappings and trees.
Indeed,the decomposition of connected mappings $\mathcal{C}$ into Cayley trees $\mathcal{T}$ is not consistent with the parking procedure. Instead of using this composition, we will therefore apply a decomposition of connected mappings w.r.t.\ the last empty node in the parking procedure. 
So, let us introduce the total number $C_{n}$ of parking functions of length $n$ for connected $n$-mappings, i.e., the number of pairs $(f,s)$, with $f \in \mathcal{C}_{n}$ a connected $n$-mapping and $s \in [n]^{n}$ a parking sequence of length $n$ for $f$, such that all drivers are successful.
We will then obtain a recursive description of $C_{n}$  in which the quantity $F_{n}$ counting the number of $(n,n)$-tree parking functions which was introduced in Section~\ref{ssec:Tree_parking_functions_m=n} appears. 

\begin{figure}[tb]
\begin{minipage}{0.32\textwidth}
\raisebox{2.2cm}{
\includegraphics[scale=0.58]{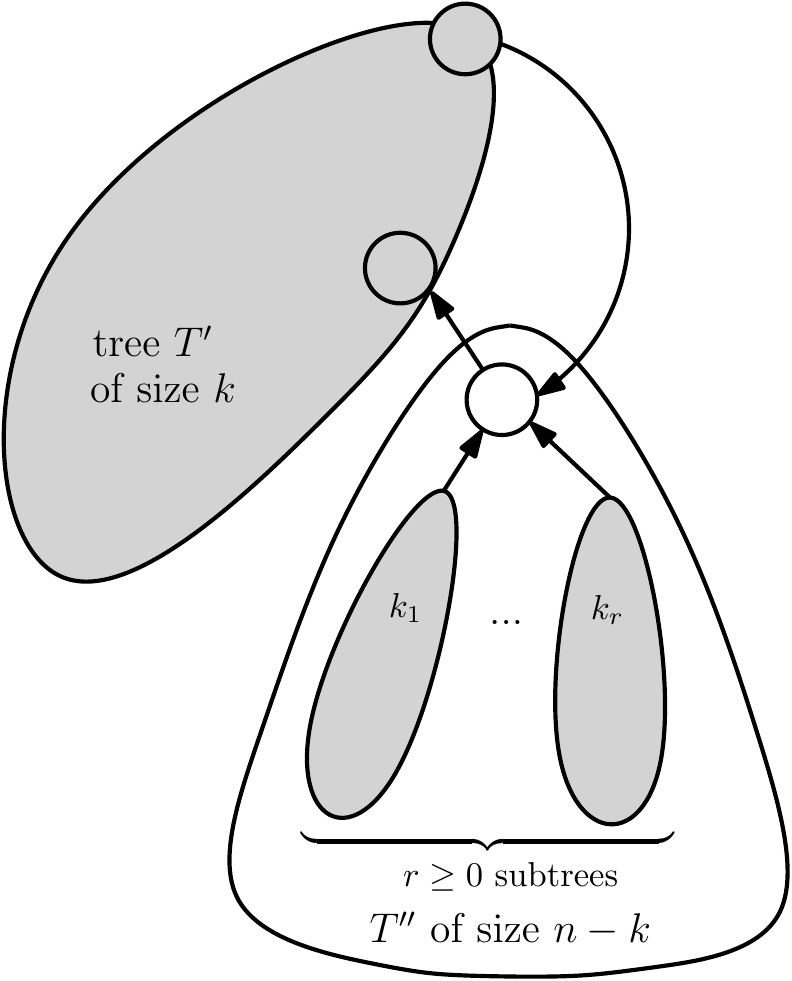}}
\end{minipage}
\hspace{0.1cm}
\begin{minipage}{0.62\textwidth}
\includegraphics[scale=0.58]{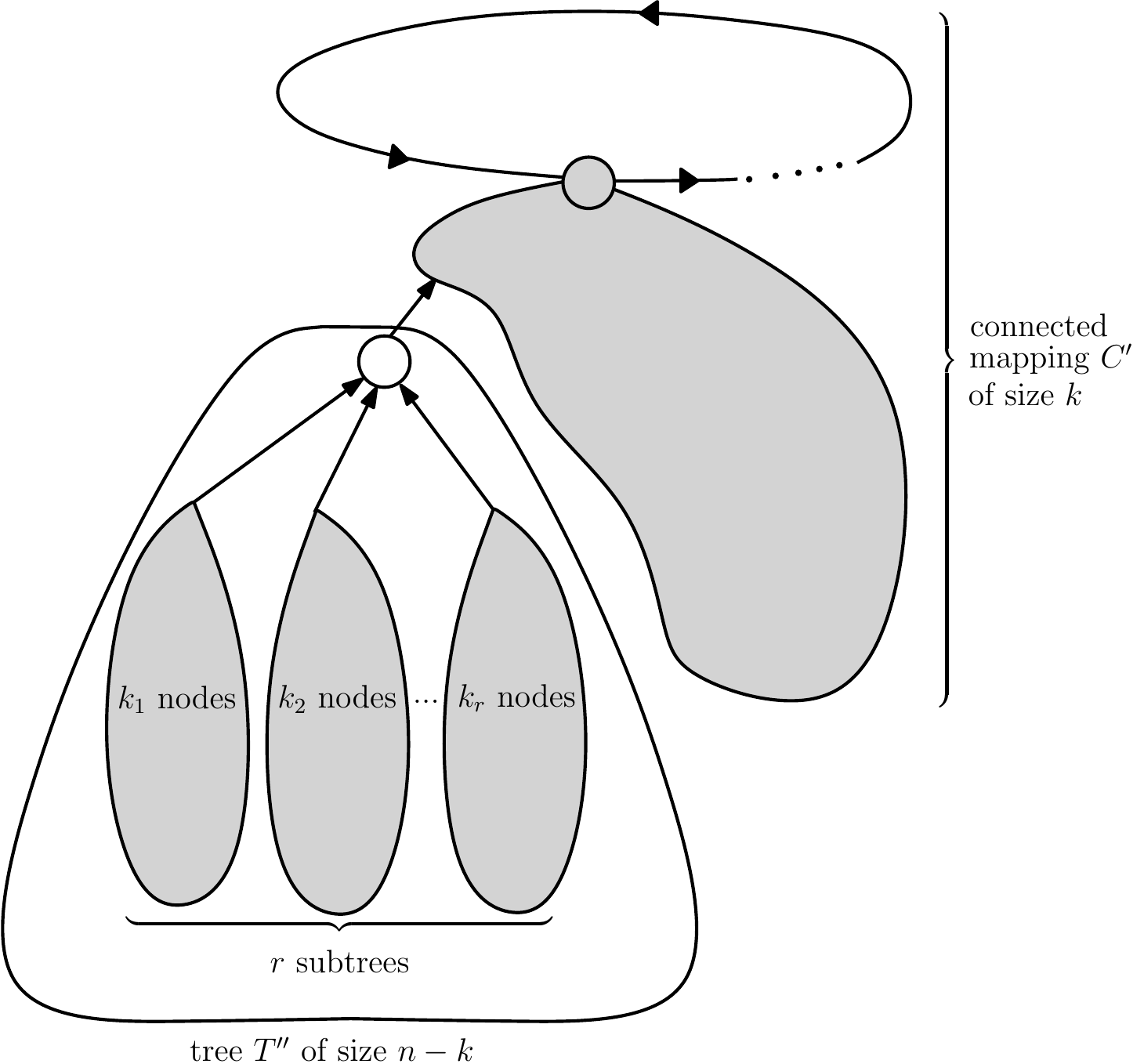}
\end{minipage}
\caption{Schematic representation of two of the three situations that might occur when considering parking functions with $n$ drivers for a connected $n$-mapping. The last empty node is either a non-cyclic node (see right hand side), a node lying on the cycle which has at least length two (see left hand side), or the root node of the single tree constituting the mapping (see the left hand side of Figure~\ref{fig:tree_parking_n=m}).}
\label{fig:mapping_parking_n=m}
\end{figure}

Three situations may occur: $(i)$ the last empty node is the root node of the Cayley tree that forms a length-$1$ cycle, 
$(ii)$ the last empty node is the root node of a Cayley tree lying in a cycle of at least two trees, $(iii)$ the last empty node is not a cyclic node, i.e., it is not one of the root nodes of the Cayley trees forming the cycle.

A schematic representation of these situations can be found in Figure~\ref{fig:mapping_parking_n=m}, where case $(ii)$ is represented on the left hand side and case $(iii)$ on the right hand side.
Case $(i)$ is the same as case $(i)$ for parking functions for trees and has been depicted on the left hand side of Figure~\ref{fig:tree_parking_n=m}.

To treat these cases only slight adaptions to the considerations made in Section~\ref{ssec:Tree_parking_functions_m=n} have to be done; case $(i)$ is explained already there.
In case $(ii)$ the last driver will always find a free parking space regardless of the $n$ possible choices of his preferred parking space.
Let us denote by $T''$ the tree whose root node is the last free parking space. When we detach the two edges linking $T''$ with the rest of the mapping graph, we cut the cycle and the graph decomposes into two trees: the tree $T''$ and the unordered tree (call it $T'$), which we may consider rooted at the former predecessor of the free parking space in the cycle of the original graph. Let us assume that $T'$ has size $k$, whereas $T''$ has size $n-k$. Then, given $T'$ and $T''$, there are $k$ different choices of constructing graphs in $\mathcal{C}_{n}$ by adding an edge from the root of $T'$ to the root of $T''$ and attaching the root of $T''$ to any of the $k$ nodes of $T'$.

In case $(iii)$ the last driver will only find a free parking space if his preferred parking space is contained in the subtree (call it $T''$) rooted at the node corresponding to the free parking space. If we detach the edge linking this subtree $T''$ with the rest of the graph, a connected mapping graph remains (call it $C'$).
Let us assume that $C'$ has size $k$, whereas the tree $T''$ has size $n-k$. Then there are $n-k$ possibilities of preferred parking spaces for the last driver such that he is successful. Again, given $C'$ and $T''$, there are $k$ different choices of constructing graphs in $\mathcal{C}_{n}$ by attaching the root of $T''$ to any of the $k$ nodes of $C'$.

Again, taking into account the order-preserving relabellings of the substructures and also the merging of the parking sequences for them, we obtain the following recursive description of $C_{n}$, valid for all $n \ge 1$:
\begin{align}
  C_{n} & = \sum_{r \ge 0} \frac{1}{r!} \sum_{\sum k_i=n-1} \hspace{-0.5cm}F_{k_{1}} F_{k_{2}} \cdots F_{k_{r}} \binom{n}{k_{1}, k_{2}, \dots, k_{r}}  \binom{n-1}{k_{1}, k_{2}, \dots, k_{r}}
  \cdot n\label{eqn:Cn_rec}\\
  & \mbox{} + \sum_{r \ge 0} \frac{1}{r!} \sum_{k+\sum k_i=n-1} \hspace{-0.5cm} F_{k} F_{k_{1}}  \cdots F_{k_{r}}  \binom{n}{k, k_{1}, \dots, k_{r}} \binom{n-1}{k, k_{1}, \dots, k_{r}}
  \cdot k n\notag\\
  & \mbox{} + \sum_{r \ge 0} \frac{1}{r!} \sum_{k+\sum k_i=n-1} \hspace{-0.5cm} C_{k} F_{k_{1}}  \cdots F_{k_{r}}  \binom{n}{k, k_{1}, \dots, k_{r}}  \binom{n-1}{k, k_{1}, \dots, k_{r}} \cdot k (n-k).\notag
\end{align}
Now we introduce the generating function
\begin{equation*}
  C(z) \colonequals \sum_{n \ge 1} C_{n} \frac{z^{n}}{(n!)^{2}}.
\end{equation*}
Then, recurrence~\eqref{eqn:Cn_rec} yields the following differential equation for $C(z)$,
\begin{multline}
  C'(z) \cdot \left(1-z \exp(F(z)) - z^{2} F'(z) \exp(F(z))\right)\\
	= \left(\left(1+z F'(z)\right)^{2} + z F'(z) + z^{2} F''(z)\right) \exp(F(z)), \quad C(0)=0,
\end{multline}
where $F(z)$ denotes the generating function of the number of tree parking functions given in \eqref{eqn:Fz_sol}.
This differential equation has the following simple solution:
\begin{equation}\label{eqn:Cz_sol}
  C(z) = \ln\left(\frac{1}{1-\frac{T(2z)}{2}}\right),
\end{equation}
as can be checked easily by using the functional equation \eqref{eqn:Tz_feq} of the tree function $T(z)$.

Extracting coefficients from \eqref{eqn:Cz_sol} gives the following auxiliary result.
\begin{lemma}
  The total number $C_{n}$ of parking functions of length $n$ for connected $n$-mappings is, for $n \ge 1$, given as follows:
  \begin{equation*}
    C_{n} = n! (n-1)! \sum_{j=0}^{n-1} \frac{(2n)^{j}}{j!}.
  \end{equation*}
\label{thm:n=m_connected}
\end{lemma}
\begin{proof}
  Using \eqref{eqn:Tz_feq}, a standard application of the Lagrange inversion formula yields
	\begin{align*}
	  [z^{n}] \ln\left(\frac{1}{1-\frac{T(2z)}{2}}\right) & = 2^{n} [z^{n}] \ln\left(\frac{1}{1-\frac{T(z)}{2}}\right)
		= \frac{2^{n}}{n} [T^{n-1}] \frac{e^{n T}}{2(1-\frac{T}{2})}\\
		& = \frac{2^{n-1}}{n} \sum_{j=0}^{n-1} \frac{n^{n-1-j}}{2^{j} (n-1-j)!}
		= \frac{1}{n} \sum_{j=0}^{n-1} \frac{(2n)^{j}}{j!},
	\end{align*}
	and further
	\begin{equation*}
	  C_{n} = (n!)^{2} [z^{n}] C(z) = (n!)^{2} [z^{n}] \ln\left(\frac{1}{1-\frac{T(2z)}{2}}\right)
		= n! (n-1)! \sum_{j=0}^{n-1} \frac{(2n)^{j}}{j!}.
	\end{equation*}
\end{proof}

Now we are in the position to study the total number $M_{n}$ of $(n,n)$-mapping parking functions.
Again we introduce the generating function
\begin{equation*}
  M(z) \colonequals \sum_{n \ge 0} M_{n} \frac{z^{n}}{(n!)^{2}}.
\end{equation*}
Since the functional digraph of a mapping can be considered as the set of its connected components and furthermore a parking function for a mapping can be considered as a shuffle of the corresponding parking functions for the connected components, we get the following simple relation between the generating functions $M(z)$ and $C(z)$ of parking functions for mappings and connected mappings, respectively:
\begin{equation*}
  M(z) = \exp(C(z)).
\end{equation*}
Thus, by using~\eqref{eqn:Cz_sol}, the generating function $M(z)$ is given as follows:
\begin{equation}\label{eqn:Mz_sol}
  M(z) = \frac{1}{1-\frac{T(2z)}{2}}.
\end{equation}

Next, we remark that the following relation between $M(z)$ and $F(z)$, the generating functions for the number of parking functions for mappings and trees, holds:
\begin{equation*}
1+z F'(z) = 1+\frac{T(2z)}{1-T(2z)}\cdot \left( 1-\frac{1}{2-T(2z)}\right) = 1+\frac{\frac{T(2z)}{2}}{1-\frac{T(2z)}{2}} =
M(z),
\end{equation*}
where we used $T'(z)=T(z)/(z\cdot (1-T(z)))$ obtained by differentiating \eqref{eqn:Tz_feq}.
At the level of coefficients, this immediately shows the following somewhat surprising connection between $F_n$ and $M_n$.

\begin{theorem}\label{thm:link_trees_mappings_m=n}
For all $n \geq 1$ it holds that the total numbers $F_{n}$ and $M_{n}$ of $(n,n)$-tree parking functions and $(n,n)$-mapping parking functions, respectively, satisfy:
\begin{equation*}
  M_{n} = n \cdot F_{n}.
\end{equation*}
\end{theorem}

Since it also holds that the number of mappings of size $n$ is exactly $n$ times the number of Cayley trees of size $n$, this implies that the average number of parking functions per mapping of a given size is exactly equal to the average number of parking functions per tree of the same size. Later, in Section~\ref{ssec:bijection_m=n} we establish a combinatorial explanation for this interesting fact. 

Extracting coefficients from the generating function solution \eqref{eqn:Mz_sol} of $M(z)$ easily yields exact formul{\ae} for $M_{n}$ and, due to Theorem~\ref{thm:link_trees_mappings_m=n}, also for $F_{n}$.
\begin{theorem}\label{thm:m=n_mappings}
  The total number $M_{n}$ of $(n,n)$-mapping parking functions is for $n \ge 1$ given as follows:
  \begin{equation*}
    M_{n} = n! (n-1)! \cdot \sum_{j=0}^{n-1} \frac{(n-j) \cdot (2n)^{j}}{j!}.
  \end{equation*}
\end{theorem}
\begin{coroll}\label{cor:m=n_trees}
  The total number $F_{n}$ of $(n,n)$-tree parking functions is for $n \ge 1$ given as follows:
  \begin{equation*}
    F_{n} = ((n-1)!)^2 \cdot \sum_{j=0}^{n-1} \frac{(n-j) \cdot (2n)^{j}}{j!}.
  \end{equation*}
\end{coroll}
\begin{proof}[Proof of Theorem~\ref{thm:m=n_mappings}]
  Again, using \eqref{eqn:Tz_feq} and the Lagrange inversion formula, we obtain
	\begin{align*}
	  [z^{n}] \frac{1}{1-\frac{T(2z)}{2}} & = 2^{n} [z^{n}] \frac{1}{1-\frac{T(z)}{2}}
		= \frac{2^{n}}{n} [T^{n-1}] \frac{e^{n T}}{2 \left(1-\frac{T}{2}\right)^{2}} \\
		&
		= \frac{2^{n-1}}{n} \sum_{k=0}^{n-1} \frac{(k+1) n^{n-1-k}}{2^{k} (n-1-k)!} = \frac{1}{n} \sum_{j=0}^{n-1} \frac{(n-j) (2n)^{j}}{j!},
	\end{align*}
	and thus
	\begin{equation*}
	  M_{n} = (n!)^{2} [z^{n}] M(z) = (n!)^{2} [z^{n}] \frac{1}{1-\frac{T(2z)}{2}} = 
		n! (n-1)! \sum_{j=0}^{n-1} \frac{(n-j) \cdot (2n)^{j}}{j!}.
	\end{equation*}
\end{proof}

The asymptotic behaviour of the numbers $M_{n}$ and $F_{n}$ for $n \to \infty$ could be deduced from these exact formul{\ae}; however, it seems easier to start with the generating function solution \eqref{eqn:Mz_sol} of $M(z)$.
Using the well-known asymptotic expansion of the tree function $T(z)$ in a complex neighbourhood of its unique dominant singularity $\frac{1}{e}$ (see \cite{flajolet2009analytic}),
\begin{equation}
  T(z) = 1- \sqrt{2} \sqrt{1-e z} + \frac{2}{3} (1-ez) + \mathcal{O}((1-ez)^{\frac{3}{2}}),
\end{equation}
one immediately obtains that $M(z)$ inherits a singularity from $T(z)$ at $\rho=\frac{1}{2e}$.
According to \eqref{eqn:Mz_sol}, there might be another singularity at the point $z_0$ where $T(2z_0)=2$.
Due to the functional equation \eqref{eqn:Tz_feq}, this would imply $2=2z_0e^2$, i.e.\ $z_0=1/e^2$. 
It is easy to check that $T(2/e^2)\approx 0.4 \neq 2$. 
Therefore, $M(z)$ has its unique dominant singularity at $\rho=\frac{1}{2e}$. 
Its local expansion in a complex neighbourhood of $\rho$ can easily be obtained as follows:
\begin{align*}
M(z) & = \frac{2}{2-T(2z)} = \frac{2}{1 +\sqrt{2} \sqrt{1-2ez} - \frac{2}{3}(1-2ez) + \mathcal{O}((1-2ez)^{\frac{3}{2}})}\\
& = 2 \Bigg(1-\sqrt{2}\sqrt{1-2ez} + \frac{2}{3}(1-2ez) + \mathcal{O}((1-2ez)^{\frac{3}{2}}) \\
& \qquad \quad + \left(-\sqrt{2} \sqrt{1-2ez} +\frac{2}{3}(1-2ez) + \mathcal{O}((1-2ez)^{\frac{3}{2}})\right)^{2}\Bigg)\\
& = 2 -2 \sqrt{2} \sqrt{1-2ez} + \frac{16}{3} (1-2ez) + \mathcal{O}((1-2ez)^{\frac{3}{2}}).
\end{align*}

A standard application of singularity analysis of generating functions, i.e., transfer lemmata which allow to deduce the asymptotic behaviour of the coefficients from the local behaviour of the generating function around its dominant singularity, shows the following asymptotic equivalent of the numbers $M_{n}$.
We get
\begin{equation*}
[z^{n}] M(z) \sim \frac{\sqrt{2}}{\sqrt{\pi}} \frac{(2e)^{n}}{n^{\frac{3}{2}}}
\end{equation*}
and the following corollary, which follows directly when applying Stirling's approximation formula for the factorials \cite{flajolet2009analytic}.
\begin{coroll}\label{cor:asympM}
  The total number $M_{n}$ of $(n,n)$-mapping parking functions and the total number $F_{n}$ of $(n,n)$-tree parking functions, respectively, are asymptotically, for $n \to \infty$, given as follows:
  \begin{equation*}
    M_{n} \sim \frac{\sqrt{2 \pi} \, 2^{n+1} n^{2n}}{\sqrt{n} \, e^{n}}, \quad \text{and} \quad
		F_{n} \sim \frac{\sqrt{2 \pi} \, 2^{n+1} n^{2n}}{n^{\frac{3}{2}} \, e^{n}}.
  \end{equation*}
\end{coroll}

\subsection{Bijective relation between parking functions for trees and mappings\label{ssec:bijection_m=n}}

The simple relation between the total number of parking functions of a given size for trees and mappings, respectively, stated in Theorem~\ref{thm:link_trees_mappings_m=n} was proved by algebraic manipulations of the corresponding generating functions.
This does not provide a combinatorial explanation of this fact. 
We thus present a bijective proof of this result in the following.
The bijection $\varphi$ is illustrated in Figure~\ref{fig:bij_tree_mapping_n=m} where an example involving a tree of size $8$ is given.

\begin{figure}[bt]
\begin{center}
\includegraphics[height=4cm]{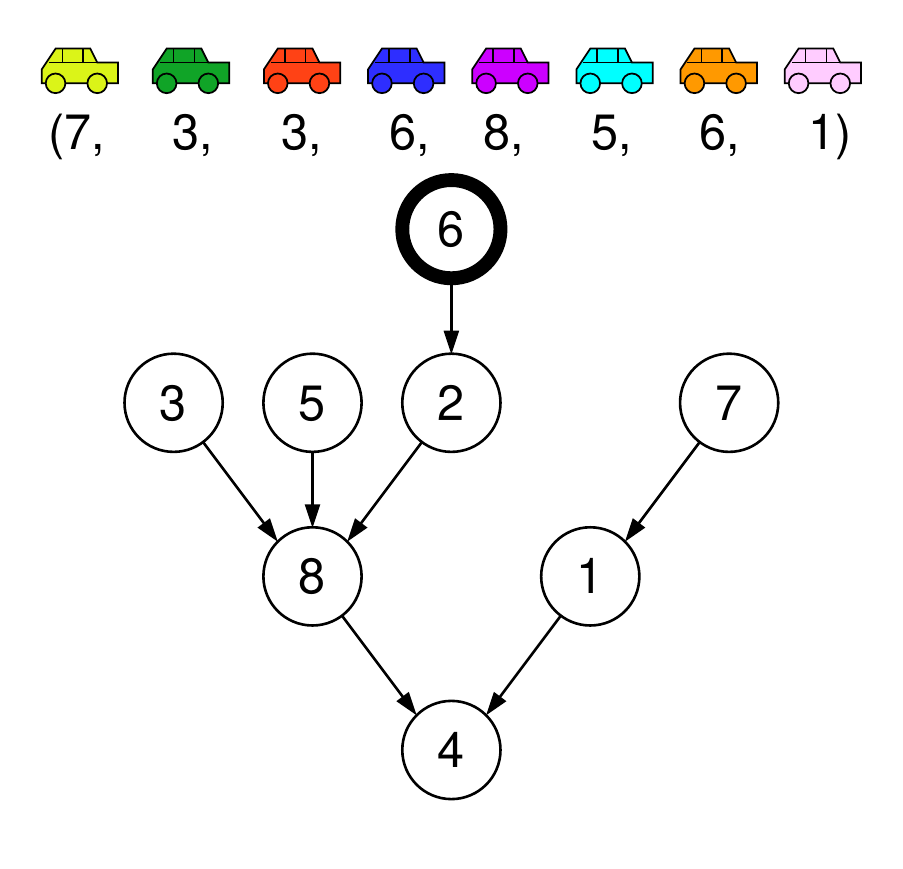}
\raisebox{1.6cm}{$\Rightarrow$}
\includegraphics[height=4cm]{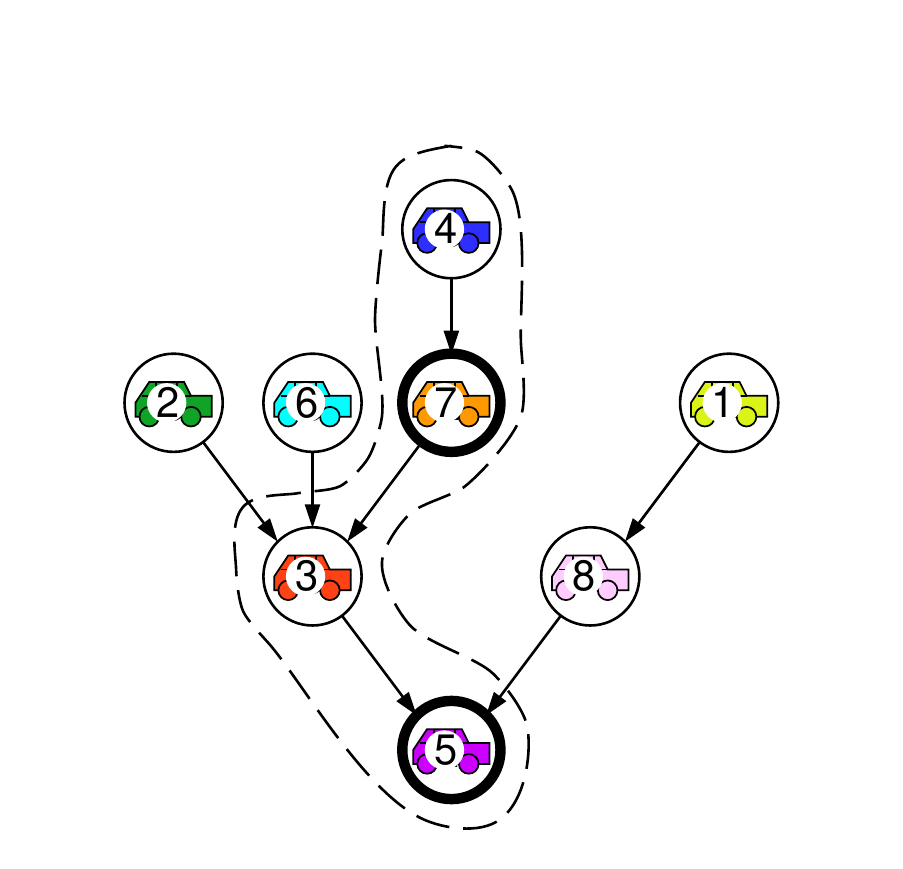}
\raisebox{1.6cm}{$\Rightarrow$}\\
\raisebox{1.6cm}{$\Rightarrow$}
\includegraphics[height=4cm]{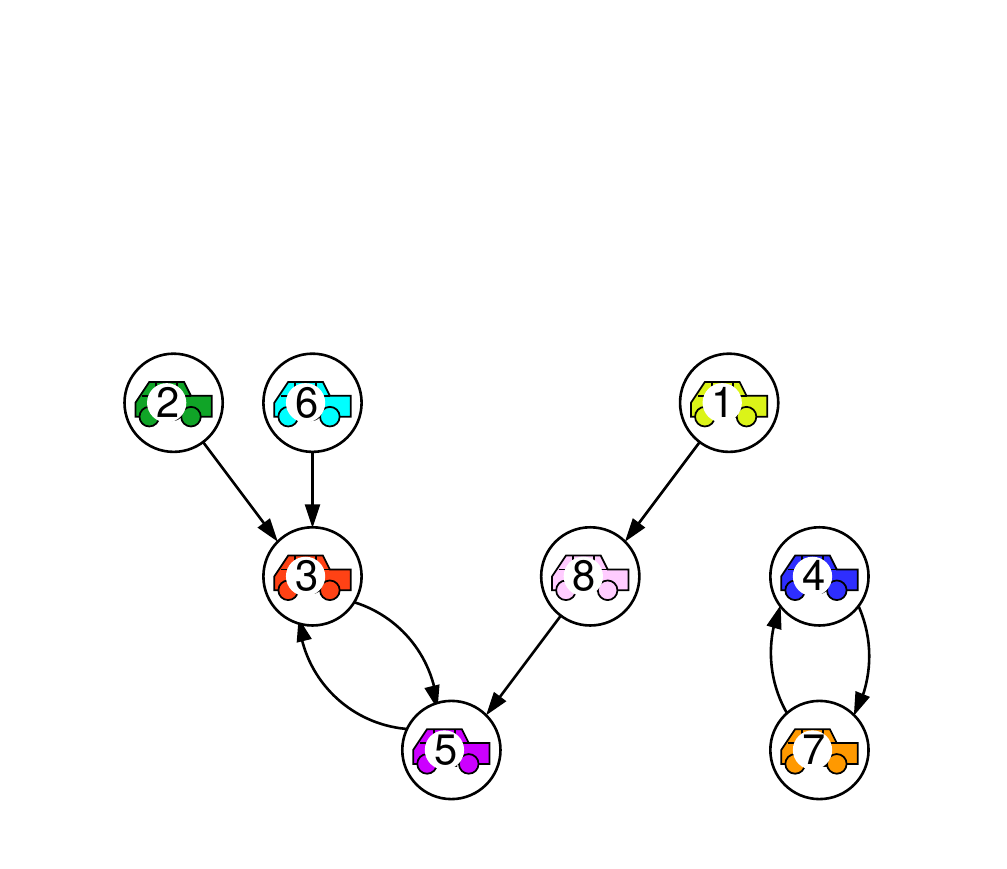}
\raisebox{1.6cm}{$\Rightarrow$}
\includegraphics[height=4cm]{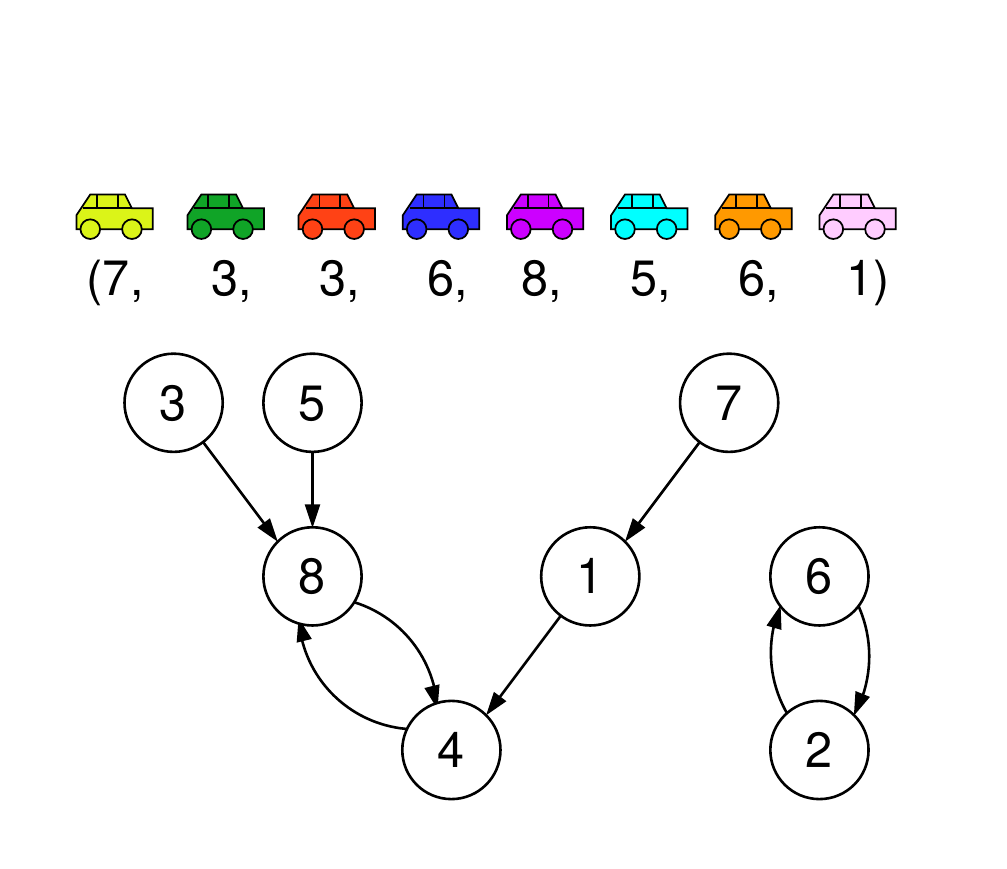}
\end{center}
\caption{The bijection $\varphi$ described in Theorem~\ref{thm:bijection_parking_m=n} is applied to the triple $(T,s,w)$ with $T$ a size-$8$ tree, $s=(7,3,3,6,8,5,6,1)$ a parking function for $T$ with $8$ drivers and the node $w=6$ which is marked in $T$ in the top left corner.
It yields the mapping parking function $(f,s)$ with $f : [8] \to [8]$ an $8$-mapping represented in the bottom right corner. The labels of the cars denote their ranks as defined in the proof of Theorem~\ref{thm:bijection_parking_m=n}. The marked nodes in the second picture correspond to the right-to-left maxima in the sequence of ranks of the drivers on the path from $w$ to the root.\label{fig:bij_tree_mapping_n=m}}
\end{figure}

\begin{theorem}\label{thm:bijection_parking_m=n}
  For each $n \ge 1$, there exists a bijection $\varphi$ from the set of triples $(T,s,w)$, with $T \in \mathcal{T}_{n}$ a tree of size $n$, $s \in [n]^{n}$ a parking function for $T$ with $n$ drivers, and $w \in T$ a node of $T$, to the set of pairs $(f,s)$ where $f \in \mathcal{M}_{n}$ is an $n$-mapping and $s \in [n]^{n}$ is a parking function for $f$ with $n$ drivers. Thus
\begin{equation*}
  n \cdot F_{n} = M_{n}, \quad \text{for $n \ge 1$}.
\end{equation*}
\end{theorem}

\begin{remark}
	The parking function $s$ remains unchanged under the bijection $\varphi$. Thus, when denoting by $\hat{F}_{n}(s)$ and $\hat{M}_{n}(s)$ the number of trees $T \in \mathcal{T}_{n}$ and mappings $f \in \mathcal{M}_{n}$, respectively, such that a given $s \in [n]^{n}$ is a parking function for $T$ and $f$, respectively, it holds:
\begin{equation*}
  n \cdot \hat{F}_{n}(s) = \hat{M}_{n}(s), \quad \text{for $n \ge 1$}.
\end{equation*}
\end{remark}

\begin{proof}[Proof of Theorem~\ref{thm:bijection_parking_m=n}.]
Let us start by defining the rank of a node in $T$: the rank $k(v)$ is defined as $\pi^{-1}(v)$, where the output-function $\pi$ of $(T,s)$  is a bijection since $s$ is a parking function for $T$ with $n$ drivers.
That is, $k(v)=i$ if and only if the $i$-th car in the parking sequence ends up parking at node $v$ in $T$.
For an example, see the second picture in Figure~\ref{fig:bij_tree_mapping_n=m}.
Furthermore, we will denote by $T(v)$ the parent of node $v$ in the tree $T$ in what follows.
That is, for $v \neq \rt(T)$, $T(v)$ is the unique node such that $(v, T(v))$ is an edge in $T$.

Given a triple $(T,s,w)$, we consider the unique path $w \leadsto \rt(T)$ from the node $w$ to the root of $T$.
It consists of the nodes $v_{1}=w$, $v_2=T(v_1), \ldots, v_{i+1}=T(v_i), \ldots, v_{r} = \rt(T)$ for some $r \geq 1$.
To this sequence $v_1, v_2, \ldots, v_r$ of nodes in $T$ we associate its sequence of ranks $k_{1}, \dots, k_{r}$ where $k_i \colonequals k(v_i)$.
We denote by $I = (i_{1}, \dots, i_{t})$, with $i_{1} < i_{2} < \cdots < i_{t}$ for some $t \ge 1$, the indices of the right-to-left maxima in this sequence, i.e.,
\begin{equation*}
  i \in I \Longleftrightarrow k_{i} > k_{j}, \quad \text{for all $j > i$}.
\end{equation*}
The corresponding set of nodes in the path $w \leadsto \rt(T)$ will be denoted by $V_{I} \colonequals \{v_{i} : i \in I\}$. Of course, if follows from the definition that the root node is always contained in $V_{I}$, i.e., $v_{r} \in V_{I}$.

We can now describe the function $\varphi$ by constructing an $n$-mapping $f$, such that $s$ is a parking function for $f$. 
The $t$ right-to-left maxima in the sequence $(k_{1}, \dots, k_{r})$ will give rise to $t$ connected components in the functional digraph $G_{f}$.
Moreover, the nodes on the path $w \leadsto \rt(T)$ in $T$ will correspond to the cyclic nodes in $G_{f}$. We describe $f$ by defining $f(v)$ for all $v \in [n]$, where we distinguish whether $v \in V_{I}$ or not.
\begin{itemize}
\item[$(a)$] Case $v \notin V_{I}$: We set $f(v)=T(v)$.
\item[$(b)$] Case $v \in V_{I}$: We have $v=v_{i_{\ell}}$ for some  $1 \le \ell \le t$.
The crucial observation is that the edge $(v_{i_{\ell}}, T \left( v_{i_{\ell}}\right) )$, is never used by any of the drivers of $s$.
Since $k_{i_{\ell}}$ is a right-to-left maximum in the sequence $k_{1}, \dots, k_{r}$, all nodes that lie on the path from $v_{i_{\ell}}$ to the root are already occupied when the $k_{i_{\ell}}$-th driver parks at $v_{i_{\ell}}$.
Thus, no driver before $k_{i_{\ell}}$ (then he would have parked at $v_{i_{\ell}}$) nor after $k_{i_{\ell}}$ (then he would not be able to park anywhere) could have reached and thus left the node $v_{i_{\ell}}$. 
We may thus delete this edge and attach the node $v_{i_{\ell}}$ to an arbitrary node without violating the property that $s$ is a parking function. 
Since we want to be able to reconstruct $T$ from $f$ we will do this in the following way:
\begin{equation*}
  f(v_{i_{\ell}}) \colonequals T\left( v_{i_{\ell-1}}\right),
\end{equation*}
where we set $T(v_{i_0})=v_1=w$.
This means that the nodes on the path $w \leadsto \rt(T)$ in $T$ form $t$ cycles $C_{1} \colonequals (v_{1}, \dots, v_{i_{1}})$, \dots, $C_{t} \colonequals (T(v_{i_{t-1}}), \dots, v_{r}=v_{i_t})$ in $G_{f}$.
\end{itemize}
Having defined the mapping $f$ in this way, the sequence $s$ is also a parking function for $f$ and it holds that the parking paths of the drivers coincide for $T$ and $f$. 
In particular, it holds that $\pi_{(f,s)} = \pi_{(T,s)}$ for the corresponding output-functions.

Moreover, it is easy to describe the inverse function $\varphi^{-1}$.
Given a pair $(f,s)$, we start by computing the rank of every node in $G_f$.
Then we sort the connected components of $G_f$ in decreasing order of their cyclic elements with highest rank.
That is, if $G_f$ consists of $t$ connected components and $c_i$ denotes the cyclic element in the $i$-th component with highest rank, we have $k(c_1) > k(c_2) > \ldots > k(c_t)$.
Then, for every $1 \leq i \leq t$, we remove the edges $(c_i, d_i)$ where $d_i= f(c_i)$.
Next we reattach the components to each other by establishing the edges $(c_i, d_{i+1})$ for every $1 \leq i \leq t-1$.
This leads to the tree $T$.
Note that the node $c_t$ is attached nowhere since it constitutes the root of $T$.
Setting $w=d_1$, we obtain the preimage $(T,s,w)$ of $(f,s)$.
\end{proof}

\section{Total number of parking functions: the general case\label{sec:general_case}}

In this section we study the exact and asymptotic behaviour of the total number of tree and mapping parking functions for the general case of $n$ parking spaces and $0 \le m \le n$ drivers. In what follows we will always use $\tilde{m} \colonequals n-m$, i.e., $\tilde{m}$ denotes the number of empty parking spaces (i.e., empty nodes) in the tree or mapping graph after all $m$ drivers have parked. The case $\tilde{m}=0$ has already been treated  in Section~\ref{sec:drivers_equal_nodes} and the results obtained there will be required here.

\subsection{Tree parking functions\label{ssec:Tree_parking_functions_general}}

We analyze the total number $F_{n,m}$ of $(n,m)$-tree parking functions, i.e., the number of pairs $(T,s)$, with $T \in \mathcal{T}_{n}$ a Cayley tree of size $n$ and $s \in [n]^{m}$ a parking sequence of length $m$ for the tree $T$, such that all drivers are successful.
Furthermore, as introduced in Section~\ref{sec:drivers_equal_nodes}, $F_{n} = F_{n,n}$ denotes the number of tree parking functions when the number of parking spaces $n$ coincides with the number of drivers $m$.

Let us now consider tree parking functions for the case that $\tilde{m}$ parking spaces will remain free.
In the following it is advantageous to use the abbreviation $\tilde{F}_{n,\tilde{m}} \colonequals F_{n,n-\tilde{m}}$, thus $\tilde{F}_{n,0} = F_{n}$.
Let us assume that $1 \le \tilde{m} \le n$. To get a recursive description for the numbers $\tilde{F}_{n,\tilde{m}}$, we use the combinatorial decomposition of a Cayley tree $T \in \mathcal{T}_{n}$ w.r.t.\ the free node which has the largest label amongst all $\tilde{m}$ empty nodes in the tree. 

Again, the two situations depicted in Figure~\ref{fig:tree_parking_n=m} have to be considered.
The argumentation given in Section~\ref{ssec:Tree_parking_functions_m=n} for the case $\tilde{m}=0$ can be adapted easily: in case $(i)$, the root node is the empty node with largest label and we assume that the $r$ subtrees of the root are of sizes $k_{1}, \dots, k_{r}$ (with $\sum_{i} k_{i} = n-1$) and contain $\ell_{1}, \dots, \ell_{r}$ (with $\sum_{i} \ell_{i} = \tilde{m}-1$) empty nodes, respectively. 
In case $(ii)$, a non-root node is the empty node with largest label.
We denote by $T''$ the subtree of $T$ rooted at this empty node.
After detaching $T''$ from the remaining tree we obtain a tree $T'$ that is of size $k$ and has $\ell$ empty nodes for some $1 \leq k \leq n-1$ and $0 \leq \ell \leq \tilde{m}-1$. Furthermore, we assume that the $r$ subtrees of the root of $T''$ are of sizes $k_{1}, \dots, k_{r}$ (with $k + \sum_{i} k_{i} = n-1$) and contain $\ell_{1}, \dots, \ell_{r}$ (with $\ell + \sum_{i} \ell_{i} = \tilde{m}-1$) empty nodes, respectively. In the latter case one has to take into account that there are $k$ possibilities of attaching the root of $T''$ to one of the $k$ nodes in $T'$ yielding the same decomposition. The following recursive description of the numbers $\tilde{F}_{n,\tilde{m}}$ follows by considering the order-preserving relabellings of the subtrees and also the merging of the parking sequences for the subtrees. Moreover, one uses the simple fact that, when fixing an empty node $v$ and considering all possible labellings of the $\tilde{m}$ empty nodes, only a fraction of $\frac{1}{\tilde{m}}$ of all labellings leads to $v$ having the largest label amongst all empty nodes.

We then get the following recurrence
\begin{align}
  \tilde{F}_{n,\tilde{m}} & = \frac{1}{\tilde{m}} \sum_{r \ge 0} \frac{1}{r!} \sum_{k_{1}+\cdots+k_{r}=n-1} \sum_{\ell_{1}+\cdots+\ell_{r}=\tilde{m}-1}
  \tilde{F}_{k_{1},\ell_{1}} \cdot \tilde{F}_{k_{2},\ell_{2}} \cdots \tilde{F}_{k_{r},\ell_{r}} \cdot\notag\\
  & \qquad \cdot \binom{n}{k_{1}, k_{2}, \dots, k_{r}} \binom{n-\tilde{m}}{k_{1}-\ell_{1}, k_{2}-\ell_{2}, \dots, k_{r}-\ell_{r}}\label{eqn:Fnm_rec}\\
	& \quad \mbox{} + \frac{1}{\tilde{m}} \sum_{r \ge 0} \frac{1}{r!} \sum_{k+k_{1}+\cdots+k_{r}=n-1} \cdot \sum_{\ell+\ell_{1}+\cdots+\ell_{r}=\tilde{m}-1} \tilde{F}_{k,\ell} \tilde{F}_{k_{1},\ell_{1}} \cdots \tilde{F}_{k_{r},\ell_{r}} \cdot\notag\\
  & \qquad \cdot \binom{n}{k, k_{1}, \dots, k_{r}} \binom{n-\tilde{m}}{k-\ell,k_{1}-\ell_{1}, \dots, k_{r}-\ell_{r}} \cdot k,
		\quad \text{for $1 \le \tilde{m} \le n$},\notag
\end{align}
with initial values $\tilde{F}_{n,0} = F_{n}$. It is advantageous to introduce the generating function
\begin{equation}\label{eqn:Fzu_def}
  \tilde{F}(z,u) \colonequals \sum_{n \ge 1} \sum_{\tilde{m} \ge 0} \tilde{F}_{n,\tilde{m}} \frac{z^{n} u^{\tilde{m}}}{n! (n-\tilde{m})!}
	= \sum_{n \ge 1} \sum_{0 \le m \le n} F_{n,n-m} \frac{z^{n} u^{n-m}}{n! m!}.
\end{equation}
The recurrence relation~\eqref{eqn:Fnm_rec} then yields, after straightforward computations, the following partial differential equation
for $\tilde{F}(z,u)$:
\begin{equation}\label{eqn:Fzu_pde}
  \tilde{F}_{u}(z,u) = z^{2} \tilde{F}_{z}(z,u) \exp(\tilde{F}(z,u)) + z \exp(\tilde{F}(z,u)),
\end{equation}
with initial condition $\tilde{F}(z,0) = F(z)$ and $F(z) = \sum_{n \ge 1} F_{n} \frac{z^{n}}{(n!)^{2}}$ given by \eqref{eqn:Fz_sol}.
A suitable representation of the solution of this PDE as given next is crucial for further studies.

\begin{prop}\label{prop:Fzu_sol}
  The generating function $\tilde{F}(z,u)$ defined in \eqref{eqn:Fzu_def} is given by
  \begin{equation*}
    \tilde{F}(z,u) = Q \cdot \left(2+u(1-Q)\right) + \ln\left(1-Q\right)= \ln\left(\frac{Q(1-Q)}{z} \right),
  \end{equation*}
  where the function $Q = Q(z,u)$ is given implicitly as the solution of the functional equation
  \begin{equation}\label{eqn:Qzu_feq}
	  Q = z \cdot e^{Q \cdot \left(2+u(1-Q)\right)}.
  \end{equation}
\end{prop}
\begin{proof}
Of course, once a solution is found, it can be checked easily after some computations that this solution indeed satisfies the PDE \eqref{eqn:Fzu_pde} as well as the initial condition $\tilde{F}(z,0) = F(z)$. However, we find it useful to carry out solving this first order quasilinear partial differential equation via the so-called ``method of characteristics''. To start with we assume that we have an implicit description of a solution $\tilde{F} = \tilde{F}(z,u)$ of \eqref{eqn:Fzu_pde} via the equation
\begin{equation*}
  g(z,u,\tilde{F}) = c,
\end{equation*}
with a certain differentiable function $g$ and a constant $c$. Taking derivatives of this equation w.r.t.\ $z$ and $u$ we obtain $g_{z} + g_{\tilde{F}}\tilde{F}_{z}=0$ and $g_{u} + g_{\tilde{F}}\tilde{F}_{u}=0$. After plugging these equations into \eqref{eqn:Fzu_pde} we get the following linear PDE in reduced form for the function $g(z,u,\tilde{F})$:
\begin{equation}\label{eqn:gzF_PDE}
  g_{u} - z^{2} e^{\tilde{F}} g_{z} + z e^{\tilde{F}} g_{\tilde{F}} = 0.
\end{equation}
To solve it we consider the following system of so-called characteristic differential equations,
\begin{equation}\label{eqn:Fzn_characteristicDE}
  \dot{u} = 1, \quad \dot{z} = -z^{2} e^{\tilde{F}}, \quad \dot{\tilde{F}} = z e^{\tilde{F}},
\end{equation}
where we regard $z = z(t)$, $u = u(t)$, and $\tilde{F} = \tilde{F}(t)$ as dependent of a variable $t$, i.e., $\dot{z} = \frac{d z(t)}{dt}$, etc. Now we search for first integrals of the system of characteristic differential equations, i.e., for functions $\xi(z,u,\tilde{F})$, which are constant along any solution curve (a so-called characteristic curve) of \eqref{eqn:Fzn_characteristicDE}.

We may proceed as follows. The second and third equation of \eqref{eqn:Fzn_characteristicDE} yield the differential equation
\begin{equation*}
  \frac{dz}{d\tilde{F}} = -z,
\end{equation*}
leading to the general solution $z = c_{1} e^{-\tilde{F}}$; thus, we get the following first integral of \eqref{eqn:Fzn_characteristicDE}:
\begin{equation*}
  \xi_{1}(z,u,\tilde{F}) = c_{1} = z e^{\tilde{F}}.
\end{equation*}
To get another first integral (independent from this one) we consider the first and third differential equation of \eqref{eqn:Fzn_characteristicDE} and get, after the substitution $z = c_{1} e^{-\tilde{F}}$, simply
\begin{equation*}
  \frac{du}{d\tilde{F}} = \frac{1}{c_{1}}.
\end{equation*}
The general solution $u = \frac{\tilde{F}}{c_{1}} + c_{2}$ yields, after backsubstituting $c_{1} = z e^{\tilde{F}}$ the following first integral:
\begin{equation*}
  \xi_{2}(z,u,\tilde{F}) = c_{2} = u - \frac{\tilde{F}}{z e^{\tilde{F}}}.
\end{equation*}
Thus the general solution of \eqref{eqn:gzF_PDE} is given as follows:
\begin{equation}\label{eqn:gzuF_sol}
  g(z,u,\tilde{F}) = H\left(\xi_{1}(z,u,\tilde{F}), \xi_{2}(z,u,\tilde{F})\right) 
	= H\left(z e^{\tilde{F}}, u-\frac{\tilde{F}}{z e^{\tilde{F}}}\right) = c,
\end{equation}
with $H$ an arbitrary differentiable function in two variables and $c$ a constant. We can solve \eqref{eqn:gzuF_sol} w.r.t.\ the variable $u$ and obtain that
the general solution of the PDE \eqref{eqn:Fzu_pde} is implicitly given by
\begin{equation}\label{eqn:Fzu_general}
  u = \frac{\tilde{F}(z,u)}{z e^{\tilde{F}(z,u)}} + h\left( z e^{\tilde{F}(z,u)}\right) ,
\end{equation}
with $h(x)$ an arbitrary differentiable function in one variable.
It remains to characterize the function $h(x)$ by adapting the general solution \eqref{eqn:Fzu_general} to the initial condition $\tilde{F}(z,0) = F(z)$. First, we obtain
\begin{equation*}
  h(z e^{F(z)}) = - \frac{F(z)}{z e^{F(z)}},
\end{equation*}
with $F(z)$ given by \eqref{eqn:Fz_sol}. To get an explicit description of $h(x)$ we require some manipulations. Using the abbreviations $F=F(z)$, $T=T(2z)$ and introducing $R=R(z) \colonequals z e^{F(z)}$, we get
\begin{equation*}
  R = z e^{F} = z e^{T + \ln \left(1-\frac{T}{2}\right)} = z e^{T} \left( 1-\frac{T}{2}\right)  = \frac{T}{2} \left( 1-\frac{T}{2}\right) ,
\end{equation*}
where we applied \eqref{eqn:Tz_feq} for the last identity. Thus 
\begin{equation*}
  T = 1-\sqrt{1-4R},
\end{equation*}
since $T(0) = R(0) = 0$ determines the correct branch for the solution.
We can characterize the function $h(x)$ via
\begin{equation*}
  h(R) = - \frac{F}{R} = - \frac{T + \ln(1-\frac{T}{2})}{R} = - \frac{1-\sqrt{1-4R} + \ln(1-\frac{1-\sqrt{1-4R}}{2})}{R}.
\end{equation*}
Therefore, plugging this characterization of $h(x)$ into \eqref{eqn:Fzu_general}, the generating function $\tilde{F} = \tilde{F}(z,u)$ is given implicitly as follows:
\begin{equation}\label{eqn:Fzu_sol_complicated}
  uz e^{\tilde{F}} - \tilde{F} + 1-\sqrt{1-4ze^{\tilde{F}}} + \ln\left( 1-\frac{1-\sqrt{1-4ze^{\tilde{F}}}}{2}\right)  = 0.
\end{equation}
To get a more amenable representation we introduce $Q = Q(z,u)$ via
\begin{equation*}
  Q \colonequals \frac{1-\sqrt{1-4ze^{\tilde{F}}}}{2}.
\end{equation*}
First, we get
\begin{equation}\label{eqn:Fzu_Qzu_rel}
  e^{\tilde{F}} = \frac{Q (1-Q)}{z},
\end{equation}
and, after plugging this into \eqref{eqn:Fzu_sol_complicated},
\begin{equation*}
  \tilde{F} = uQ(1-Q) +2Q + \ln(1-Q).
\end{equation*}
Exponentiating the latter equation shows then the functional equation characterizing $Q$,
\begin{equation*}
  Q = z e^{uQ(1-Q) + 2Q},
\end{equation*}
finishing the proof.
\end{proof}

As for the case where the number of drivers coincides with the size of the tree, we do not extract coefficients at this point yet.
We will see in Theorem~\ref{thm:link_trees_mappings_general} that the numbers $F_{n,m}$ are again linked directly to the numbers $M_{n,m}$ counting mapping parking functions and we shall therefore content ourselves with extracting coefficients for the corresponding generating function $\tilde{M}(z,u)$.

\subsection{Mapping parking functions}

We continue our studies on mapping parking functions by considering the total number $M_{n,m}$ of $(n,m)$-mapping parking functions, i.e., the number of pairs $(f,s)$ with $f \in \mathcal{M}_{n}$ an $n$-mapping and $s \in [n]^{m}$ a parking sequence of length $m$ for the mapping $f$, such that all drivers are successful. 

As pointed out already in Section~\ref{ssec:Mapping_parking_functions_m=n}, it suffices to provide the relevant considerations for the subfamily $\mathcal{C}_{n}$ of connected $n$-mappings, since results for the general situation can then be deduced easily. Thus, let us introduce the total number $C_{n,m}$ of parking functions of length $m$ for connected $n$-mappings, i.e., the number of pairs $(f,s)$, with $f \in \mathcal{C}_{n}$ a connected $n$-mapping and $s \in [n]^{m}$ a parking sequence of length $m$ for $f$, such that all drivers are successful. Additionally, we require the numbers $F_{n}$, $C_{n}$ and $F_{n,m}$ as introduced in the Sections \ref{ssec:Tree_parking_functions_m=n}, \ref{ssec:Mapping_parking_functions_m=n} and \ref{ssec:Tree_parking_functions_general}, respectively.

Let us consider parking functions for connected mappings for the case that $\tilde{m} = n-m$ parking spaces remain free after all drivers have parked successfully. In what follows it is advantageous to define $\tilde{C}_{n,\tilde{m}} \colonequals C_{n,n-\tilde{m}}$ and also to use $\tilde{F}_{n,\tilde{m}} \colonequals F_{n,n-\tilde{m}}$ as done previously.
Then it holds that $\tilde{C}_{n,0} = C_{n}$ and $\tilde{F}_{n,0} = F_{n}$.
Let us assume that $1 \le \tilde{m} \le n$. To obtain a recursive description of the numbers $\tilde{C}_{n,\tilde{m}}$ we use the combinatorial decomposition of a connected mapping $f \in \mathcal{C}_{n}$ w.r.t.\ the free node which has the largest label amongst all $\tilde{m}$ empty nodes in the mapping graph.

Three situations may occur when using this decomposition: $(i)$ the empty node with largest label is the root node of the Cayley tree which forms a length-$1$ cycle (depicted on the left hand side of Figure~\ref{fig:tree_parking_n=m}), $(ii)$ the empty node with largest label is the root node of a Cayley tree forming a cycle of at least two trees (depicted on the left hand side of Figure~\ref{fig:mapping_parking_n=m}) and $(iii)$ the empty node with largest label is not a  cyclic node (depicted on the right hand side of Figure~\ref{fig:mapping_parking_n=m}).
Analogous considerations to the ones given for tree parking functions in Section~\ref{ssec:Tree_parking_functions_general} show the following recursive description of the number of parking functions for connected mappings for $1 \le \tilde{m} \le n$:
\begin{align}
  \tilde{C}_{n,\tilde{m}} & = \frac{1}{\tilde{m}} \sum_{r \ge 0} \frac{1}{r!} \sum_{k_{1}+\cdots+k_{r}=n-1} \sum_{\ell_{1}+\cdots+\ell_{r}=\tilde{m}-1}
  \tilde{F}_{k_{1},\ell_{1}} \cdot \tilde{F}_{k_{2},\ell_{2}} \cdots \tilde{F}_{k_{r},\ell_{r}} \cdot\notag\\
  & \quad \cdot \binom{n}{k_{1}, k_{2}, \dots, k_{r}} \binom{n-\tilde{m}}{k_{1}-\ell_{1}, k_{2}-\ell_{2}, \dots, k_{r}-\ell_{r}}\notag\\
  & \quad \mbox{} + \frac{1}{\tilde{m}} \sum_{r \ge 0} \frac{1}{r!} \sum_{k+k_{1}+\cdots+k_{r}=n-1} \sum_{\ell+\ell_{1}+\cdots+\ell_{r}=\tilde{m}-1}
  \tilde{F}_{k,\ell} \tilde{F}_{k_{1},\ell_{1}} \cdots \tilde{F}_{k_{r},\ell_{r}} \cdot\label{eqn:Cnm_rec}\\
  & \qquad \cdot \binom{n}{k, k_{1}, \dots, k_{r}} \binom{n-\tilde{m}}{k-\ell,k_{1}-\ell_{1}, \dots, k_{r}-\ell_{r}} \cdot k\notag\\
  & \quad \mbox{} + \frac{1}{\tilde{m}} \sum_{r \ge 0} \frac{1}{r!} \sum_{k+k_{1}+\cdots+k_{r}=n-1} \sum_{\ell+\ell_{1}+\cdots+\ell_{r}=\tilde{m}-1}
  \tilde{C}_{k,\ell} \tilde{F}_{k_{1},\ell_{1}} \cdots \tilde{F}_{k_{r},\ell_{r}} \cdot\notag\\
  & \qquad \cdot \binom{n}{k, k_{1}, \dots, k_{r}} \binom{n-\tilde{m}}{k-\ell,k_{1}-\ell_{1}, \dots, k_{r}-\ell_{r}} \cdot k, \notag
\end{align}
with initial values $\tilde{C}_{n,0} = C_{n}$. When introducing the generating function
\begin{equation}\label{eqn:Czu_def}
  \tilde{C}(z,u) \colonequals \sum_{n \ge 1} \sum_{\tilde{m} \ge 0} \tilde{C}_{n,\tilde{m}} \frac{z^{n} u^{\tilde{m}}}{n! (n-\tilde{m})!},
\end{equation}
recurrence~\eqref{eqn:Cnm_rec} yields the following first order linear partial differential equation for the function $\tilde{C}(z,u)$:
\begin{equation}\label{eqn:Czu_PDE}
  \tilde{C}_{u}(z,u) = z^{2} \tilde{C}_{z}(z,u) \exp(\tilde{F}) + z \exp(\tilde{F}) + z^{2} \tilde{F}_{z} \exp(\tilde{F}),
\end{equation}
with $\tilde{F} = \tilde{F}(z,u) = \sum_{n,\tilde{m}} \tilde{F}_{n,\tilde{m}} \frac{z^{n} u^{\tilde{m}}}{n! (n-\tilde{m})!}$ the corresponding generating function for the number of tree parking functions given in Proposition~\ref{prop:Fzu_sol}, and initial condition $\tilde{C}(z,0) = C(z)$, with $C(z) = \sum_{n \ge 1} C_{n} \frac{z^{n}}{(n!)^{2}}$ given by \eqref{eqn:Cz_sol}.
A suitable representation of the solution of the PDE \eqref{eqn:Czu_PDE} is given in the following proposition.

\begin{prop}\label{prop:Czu_sol}
  The generating function $\tilde{C}(z,u)$ defined in \eqref{eqn:Czu_def} is given as follows:
  \begin{equation*}
    \tilde{C}(z,u) = \ln\left( \frac{1}{(1-Q)(1-u Q)}\right) ,
  \end{equation*}
  where the function $Q = Q(z,u)$ is given implicitly as the solution of the following functional equation:
  \begin{equation*}
	  Q = z \cdot e^{Q \cdot \left(2+u(1-Q)\right)}.
  \end{equation*}
\end{prop}
\begin{proof}
  To solve equation \eqref{eqn:Czu_PDE} we first consider the partial derivatives of the function $Q=Q(z,u)$ occurring in the characterization of the function $\tilde{F} = \tilde{F}(z,u)$ given in Proposition~\ref{prop:Fzu_sol}. Starting with \eqref{eqn:Qzu_feq}, implicit differentiation yields
\begin{equation}\label{eqn:Qzu_partderiv}
  Q_{z} = \frac{Q}{z (1-2Q) (1-uQ)} \quad \text{and} \quad Q_{u} = \frac{Q^{2} (1-Q)}{(1-2Q) (1-uQ)}.
\end{equation}
Thus, due to \eqref{eqn:Fzu_Qzu_rel}, it holds
\begin{equation*}
  Q_{u}(z,u) = z^{2} Q_{z}(z,u) e^{\tilde{F}(z,u)},
\end{equation*}
i.e., $Q(z,u)$ solves the reduced PDE corresponding to \eqref{eqn:Czu_PDE}. This suggests the substitution $z = z(Q) \colonequals \frac{Q}{e^{Q(2+u(1-Q))}}$ and we introduce
\begin{equation*}
  \hat{C}(Q,u) \colonequals \tilde{C}(z(Q),u) = \tilde{C}\left(\frac{Q}{e^{Q(2+u(1-Q))}},u\right).
\end{equation*}
After straightforward computations, which are thus omitted, equation \eqref{eqn:Czu_PDE} reads as
\begin{equation*}
  \hat{C}_{u}(Q,u) = \frac{Q}{1-uQ}.
\end{equation*}
Thus, after backsubstituting $z$ and $\tilde{C}(z,u)$, the general solution of this equation is given by
\begin{equation}\label{eqn:Czu_sol_gen}
  \tilde{C}(z,u) = \ln\left(\frac{1}{1-u Q(z,u)}\right) + \tilde{h}(Q(z,u)),
\end{equation}
with an arbitrary differentiable function $\tilde{h}(x)$. To characterize it, we evaluate \eqref{eqn:Czu_sol_gen} at $u=0$ and use the initial condition $\tilde{C}(z,u) = C(z)$, with $C(z)$ given by \eqref{eqn:Cz_sol}. Using the abbreviation $\tilde{T} \colonequals \frac{T(2z)}{2}$, one easily gets $Q(z,0) = \tilde{T}$ and further
\begin{equation*}
  \tilde{h}(\tilde{T}) = \tilde{h}(Q(z,u)) = \tilde{C}(z,0) = C(z) = \ln\left(\frac{1}{1-\tilde{T}}\right),
\end{equation*}
which characterizes the function $\tilde{h}(x)$. The proposition follows immediately.
\end{proof}

We are now able to treat the total number $M_{n,m}$ of $(n,m)$-mapping parking functions. We introduce $\tilde{M}_{n,\tilde{m}} \colonequals M_{n,n-\tilde{m}}$ and the generating function
\begin{equation}\label{eqn:Mzu_def}
  \tilde{M}(z,u) \colonequals \sum_{n \ge 0} \sum_{\tilde{m} \ge 0} \tilde{M}_{n,\tilde{m}} \frac{z^{n} u^{\tilde{m}}}{n! (n-\tilde{m})!}.
\end{equation}
The decomposition of mapping parking functions into parking functions for their connected components immediately gives the relation
\begin{equation*}
  \tilde{M}(z,u) = \exp\left(\tilde{C}(z,u)\right)
\end{equation*}
for the respective generating functions. According to Proposition \ref{prop:Czu_sol} we obtain the following solution of $M(z,u)$.

\begin{prop}\label{prop:Mzu_sol}
  The generating function $\tilde{M}(z,u)$ defined in \eqref{eqn:Mzu_def} is given as follows:
  \begin{equation*}
    \tilde{M}(z,u) = \frac{1}{(1-Q)(1-u Q)},
  \end{equation*}
  where the function $Q = Q(z,u)$ is given implicitly as the solution of the following functional equation:
  \begin{equation*}
    Q = z \cdot e^{Q \cdot \left(2+u(1-Q)\right)}.
  \end{equation*}
\end{prop}

Using the representations of the generating functions $\tilde{F}(z,u)$ and $\tilde{M}(z,u)$ for the number of tree and mapping parking functions given in Proposition \ref{prop:Fzu_sol} and \ref{prop:Mzu_sol}, respectively, it can be shown easily how they are connected with each other. Namely, together with \eqref{eqn:Qzu_partderiv}, we obtain
\begin{align*}
  1+z \tilde{F}_{z}(z,u) & = 1 + z \left(\frac{1-2Q}{Q(1-Q)} Q_{z} - \frac{1}{z}\right) = \frac{1-2Q}{Q(1-Q)} \frac{Q}{(1-2Q)(1-uQ)} \\
	& = \frac{1}{(1-Q) (1-uQ)}
	 = \tilde{M}(z,u).
\end{align*}
Thus, at the level of their coefficients, we obtain the following simple relation between the total number of tree and mapping parking functions extending Theorem~\ref{thm:link_trees_mappings_m=n}.

\begin{theorem}\label{thm:link_trees_mappings_general}
For all $n \geq 1$ it holds that the total numbers $F_{n,m}$ and $M_{n,m}$ of $(n,m)$-tree parking functions and $(n,m)$-mapping parking functions, respectively, satisfy:
\begin{equation*}
  M_{n,m} = n \cdot F_{n,m}.
\end{equation*}
\end{theorem}
In Section~\ref{ssec:bijection_general} we will extend the considerations made in Section~\ref{ssec:bijection_m=n} for the particular case $m=n$ and provide a combinatorial proof of this relation.

Using Proposition~\ref{prop:Mzu_sol}, extracting coefficients leads to the following explicit formul{\ae} for the numbers $M_{n,m}$ and $F_{n,m}$. Note that specializing $m=n$ restates Theorem~\ref{thm:m=n_mappings} and Corollary~\ref{cor:m=n_trees}.
\begin{theorem}\label{thm:Mnm_formula}
  The total number $M_{n,m}$ of $(n,m)$-mapping parking functions is, for $0 \le m \le n$ and $n \ge 1$, given as follows:
  \begin{align*}
    M_{n,m} = \frac{(n-1)! m! n^{n-m}}{(n-m)!} \sum_{j=0}^{m} \binom{2m-n-j}{m-j} \frac{(2n)^{j} (n-j)}{j!}.
  \end{align*}
\end{theorem}
\begin{coroll}
  The total number $F_{n,m}$ of $(n,m)$-tree parking functions is, for $0 \le m \le n$ and $n \ge 1$, given as follows:
  \begin{align*}
    F_{n,m} = \frac{(n-1)! m! n^{n-m-1}}{(n-m)!} \sum_{j=0}^{m} \binom{2m-n-j}{m-j} \frac{(2n)^{j} (n-j)}{j!}.
  \end{align*}
\end{coroll}
\begin{proof}[Proof of Theorem~\ref{thm:Mnm_formula}]
  In view of the representation of $\tilde{M}(z,u)$ given in Proposition~\ref{prop:Mzu_sol} containing the function $Q=Q(z,u)$, we make a change of variables in order to extract coefficients. 
Using the functional equation \eqref{eqn:Qzu_feq} and the derivative \eqref{eqn:Qzu_partderiv} of $Q$ w.r.t.\ $z$, an application of the Cauchy integral formula (choosing as contour a suitable simple positively oriented closed curve around the origin) gives
	\begin{align*}
	  [z^{n}] \tilde{M}(z,u) & = \frac{1}{2 \pi i} \oint \frac{\tilde{M}(z,u)}{z^{n+1}} dz =
		\frac{1}{2 \pi i} \oint \frac{1}{z^{n+1}} \frac{1}{(1-Q) (1-uQ) } dz\\
		& =
		\frac{1}{2 \pi i} \oint \frac{e^{(uQ(1-Q)+2Q) (n+1)}}{(1-Q) (1-uQ) Q^{n+1}} \frac{(1-2Q)(1-uQ)}{e^{uQ(1-Q)+2Q}} dQ\\
		& = [Q^{n}] \frac{e^{n(uQ(1-Q)+2Q)} (1-2Q)}{1-Q}.
	\end{align*}
	Further, for $0 \le m \le n$,
	\begin{equation}\label{eqn:Mzu_coeff_Q}
	\begin{split}
	  [z^{n} u^{n-m}]\tilde{M}(z,u) & = [Q^{n} u^{n-m}] \frac{e^{u n Q (1-Q)} e^{2nQ} (1-2Q)}{1-Q}\\
		& = \frac{n^{n-m}}{(n-m)!} [Q^{m}] (1-Q)^{n-m-1} e^{2nQ} (1-2Q).
	\end{split}
	\end{equation}
	We get
	\begin{align}\label{eqn:Mnm_Q}
	  \notag M_{n,m} & = n! m! [z^{n} u^{n-m}] \tilde{M}(z,u) \\ 
	  & = \frac{n! m! n^{n-m}}{(n-m)!} [Q^{m}] (1-Q)^{n-m-1} e^{2nQ} (1-2Q) \\ \notag
	 & = \frac{n! m! n^{n-m}}{(n-m)!} \sum_{j=0}^{m} \binom{n-m-1}{j} (-1)^{j} [Q^{m-j}] e^{2nQ} (1-2Q)\\ \notag
		& = \frac{n! m! n^{n-m}}{(n-m)!} \sum_{j=0}^{m} \binom{n-m-1}{j} (-1)^{j} \frac{2 (n-m+j) (2n)^{m-j-1}}{(m-j)!}\\ \notag
		& = \frac{(n-1)! m! n^{n-m}}{(n-m)!} \sum_{j=0}^{m} \binom{j+m-n}{j} \frac{(n-m+j) \, (2n)^{m-j}}{(m-j)!}\\ \notag
		& = \frac{(n-1)! m! n^{n-m}}{(n-m)!} \sum_{j=0}^{m} \binom{2m-n-j}{m-j} \frac{(n-j) \, (2n)^{j}}{j!}.
	\end{align}
\end{proof}
From Theorem~\ref{thm:Mnm_formula} we can easily derive exact values for the total number of $(n,m)$-mapping parking functions for a moderate size of $n$.
However, due to the alternating sign of the summands in the explicit formula of $M_{n,m}$  that is inherent in the binomial coefficient, it is not well suited to deduce the asymptotic behaviour of these numbers and thus to give answers to questions concerning the probability $p_{n,m}$ that a random pair $(f,s)$ of an $n$-mapping $f$ and a sequence $s \in [n]^{m}$ of addresses of length $m$ is a parking function, when $n \to \infty$. Starting from \eqref{eqn:Mzu_coeff_Q}, such asymptotic considerations will be carried out in Section~\ref{ssec:asymptotics} using saddle point methods.

\subsection{Bijective relation between parking functions for trees and mappings\label{ssec:bijection_general}}

We will extend the bijection given in Theorem~\ref{thm:bijection_parking_m=n}, such that it also works for the general case and thus gives a bijective proof of Theorem~\ref{thm:link_trees_mappings_general}. The corresponding bijection $\varphi'$ is illustrated in Figure~\ref{fig:bij_tree_mapping_general}.
\begin{figure}
\begin{center}
\includegraphics[height=5cm]{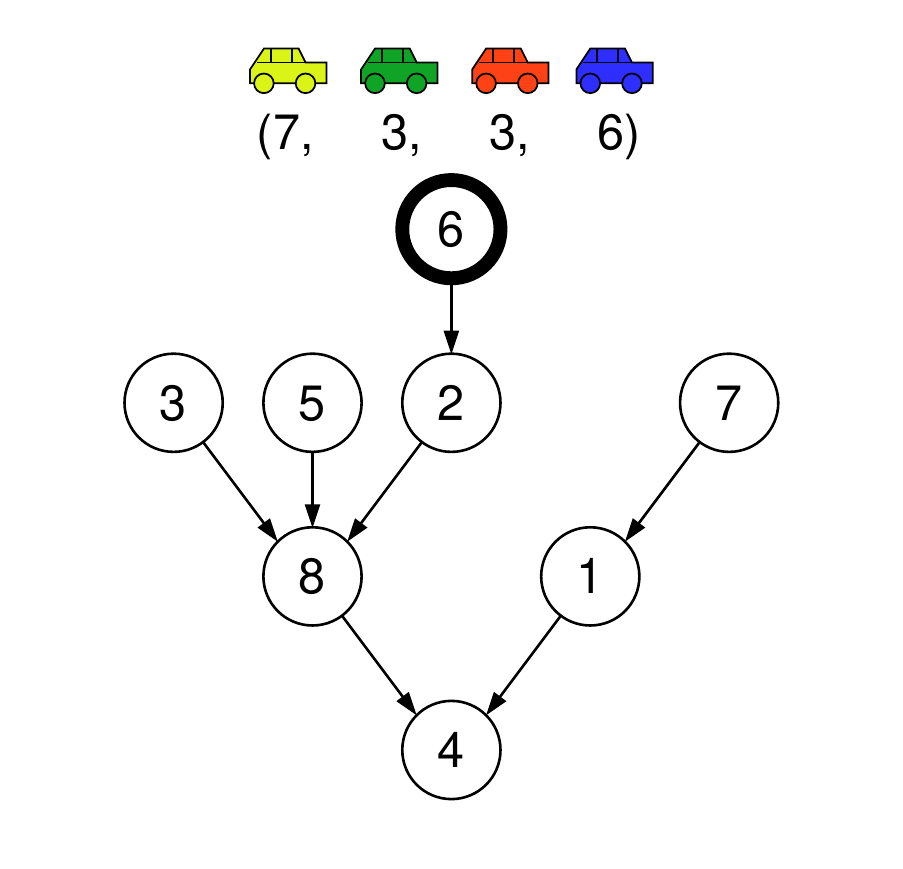}
\raisebox{1.6cm}{$\Rightarrow$}
\includegraphics[height=5cm]{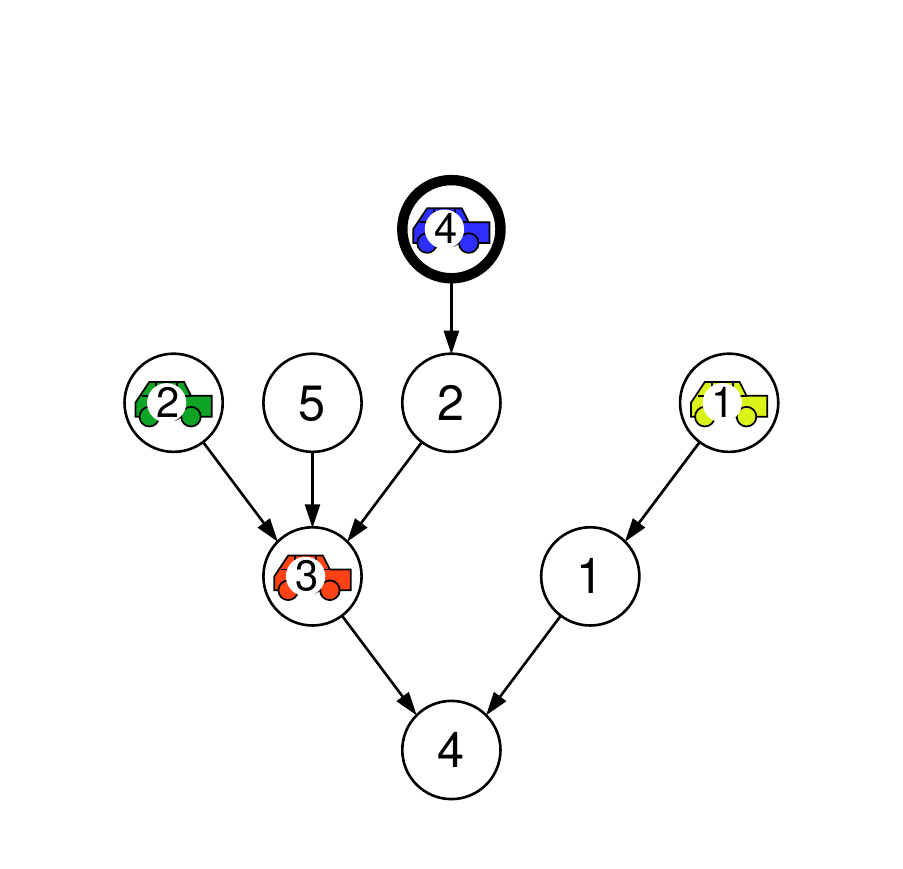}
\raisebox{1.6cm}{$\Rightarrow$}\\
\includegraphics[height=5cm]{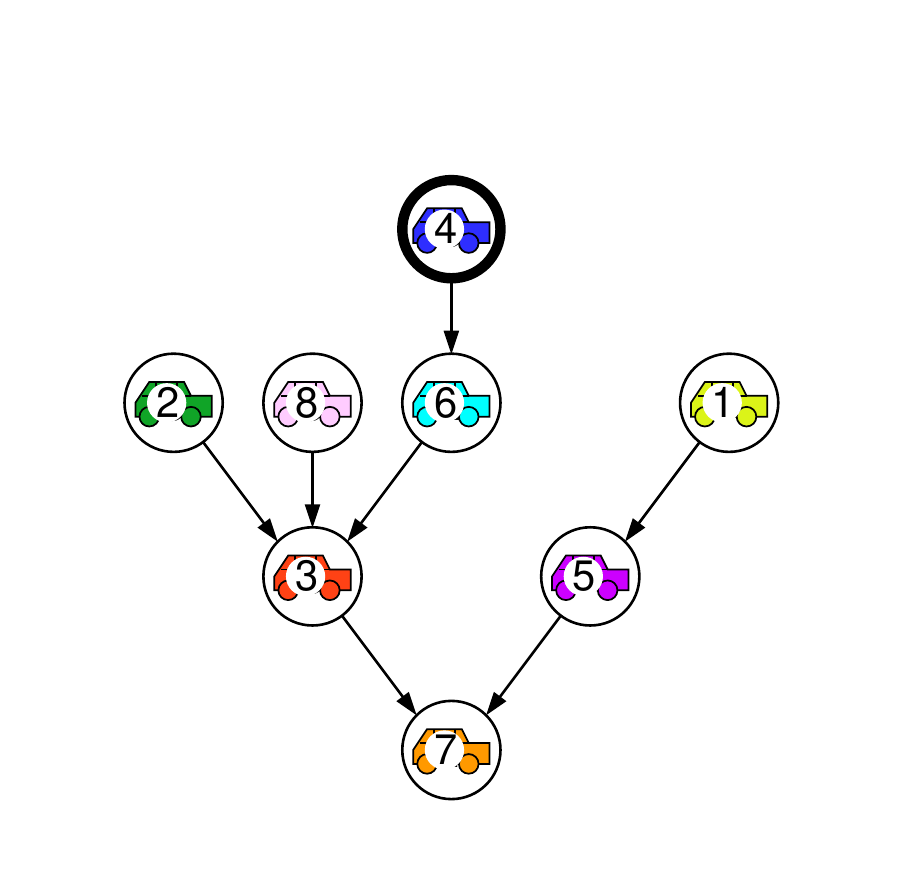}
\raisebox{1.6cm}{$\widehat{=}$}
\includegraphics[height=5cm]{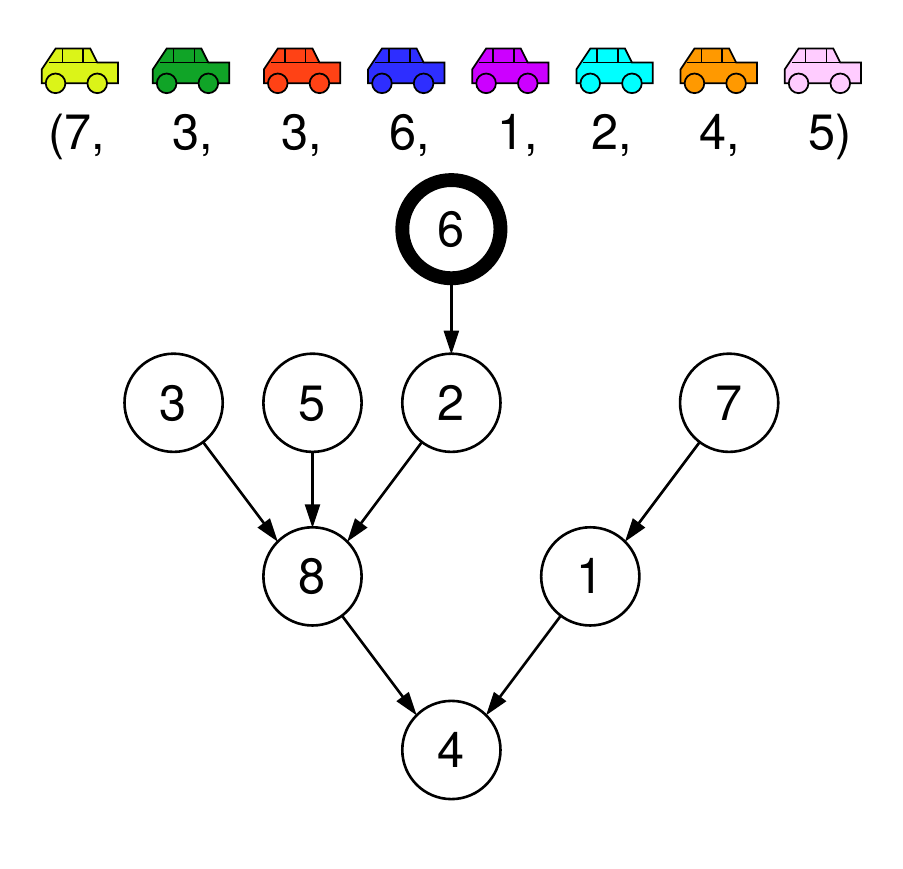}
\raisebox{1.6cm}{$\stackrel{\varphi}{\Rightarrow}$}\\
\raisebox{1.6cm}{$\stackrel{\varphi}{\Rightarrow}$}
\includegraphics[height=5cm]{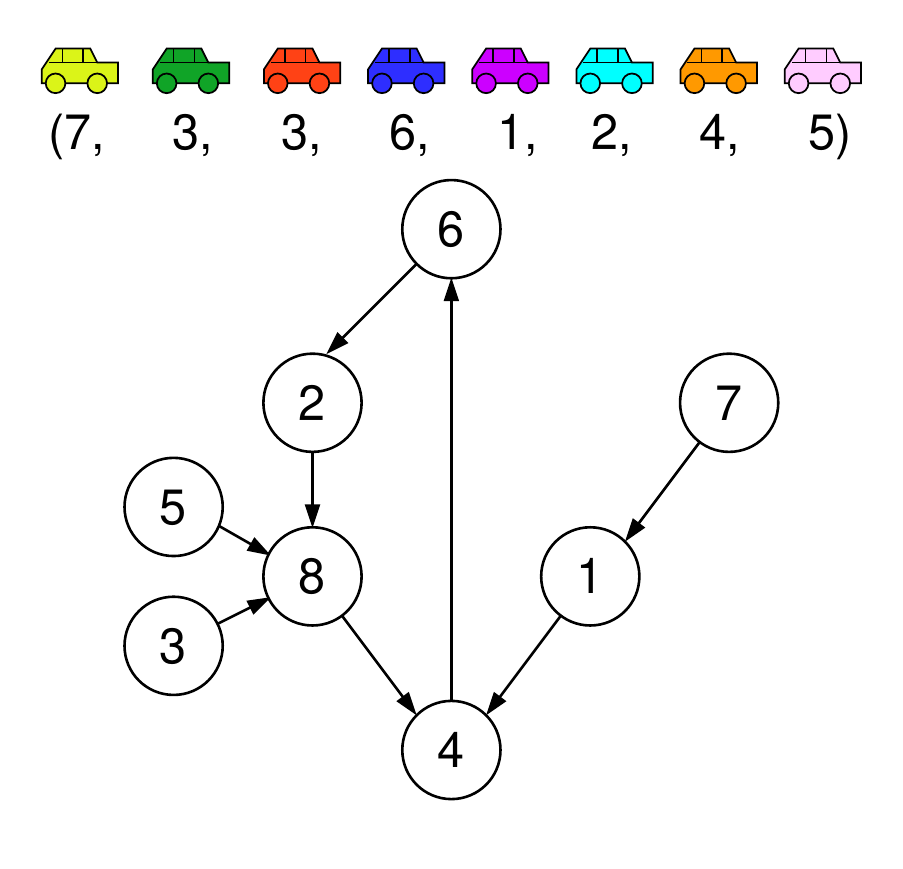}
\raisebox{1.6cm}{$\Rightarrow$}
\includegraphics[height=5cm]{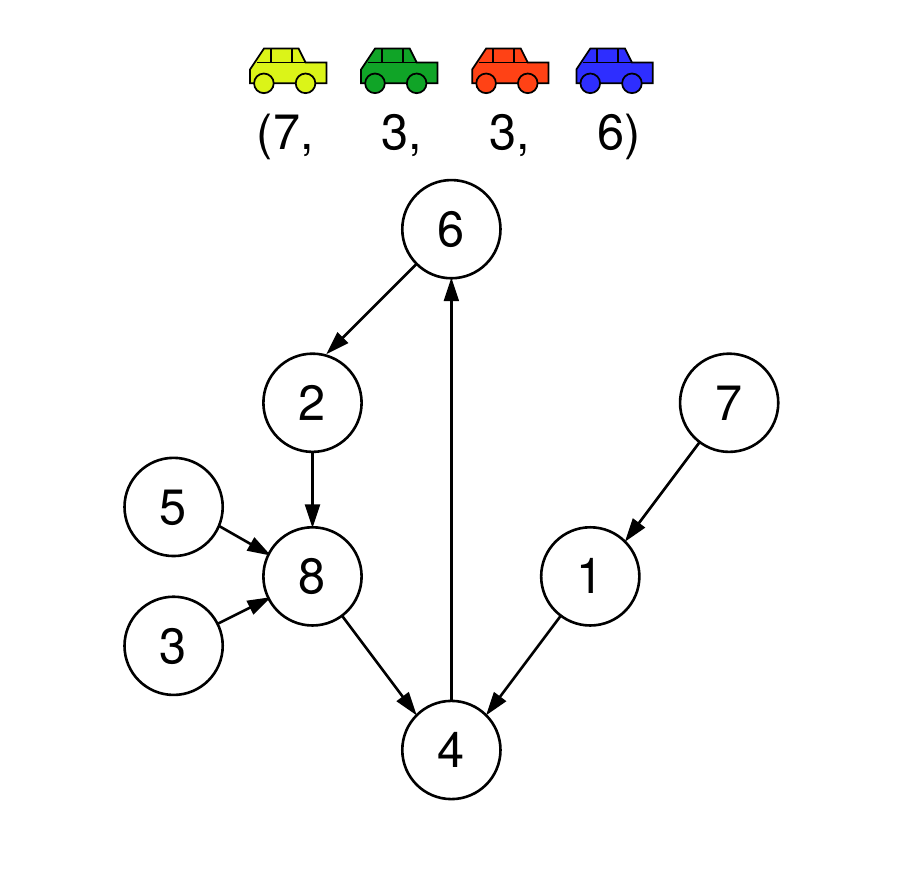}
\end{center}
\caption{The bijection $\varphi'$ described in Theorem~\ref{thm:bijection_parking_general} is applied to the triple $(T,s,w)$ with $T$ a size-$8$ tree, $s=(7,3,3,6)$ a parking function for $T$ with $4$ drivers and the node $w=6$ which is marked in $T$ in the top left corner.
It yields the mapping parking function $(f,s)$ with $f : [8] \to [8]$ an $8$-mapping represented in the bottom right corner. The function $\varphi$ denotes the bijection described in the proof of Theorem~\ref{thm:bijection_parking_m=n}.\label{fig:bij_tree_mapping_general}}
\end{figure}

\begin{theorem}\label{thm:bijection_parking_general}
  For $0 \le m \le n$ and $n \ge 1$, there exists a bijection $\varphi'$ from the set of triples $(T,s,w)$, with $T \in \mathcal{T}_{n}$ a tree of size $n$, $s \in [n]^{m}$ a parking function for $T$ with $m$ drivers, and $w \in T$ a node of $T$, to the set of pairs $(f,s)$ of $(n,m)$-mapping parking functions, i.e, $f \in \mathcal{M}_{n}$ an $n$-mapping and $s \in [n]^{m}$ a parking function for $f$ with $m$ drivers. Thus
\begin{equation*}
  n \cdot F_{n,m} = M_{n,m}, \quad \text{for $n \ge 1$}.
\end{equation*}
\end{theorem}

\begin{remark}
	It holds that the parking function $s$ remains unchanged under the bijection $\varphi'$. Thus, also the general case satisfies the relation
\begin{equation*}
  n \cdot \hat{F}_{n}(s) = \hat{M}_{n}(s), \quad \text{for $n \ge 1$},
\end{equation*}
where again $\hat{F}_{n}(s)$ and $\hat{M}_{n}(s)$ denote the number of trees $T \in \mathcal{T}_{n}$ and mappings $f \in \mathcal{M}_{n}$, respectively, such that a given sequence $s \in [n]^{m}$ is a parking function for $T$ and $f$, respectively.
\end{remark}

\begin{remark}
  The case $m=0$ gives one of the many bijective proofs of the relation $n \cdot T_{n} = M_{n}$, thus showing $T_{n} = n^{n-1}$.
\end{remark}

\begin{proof}[Proof of Theorem~\ref{thm:bijection_parking_general}.]
  In order to a establish a bijection $\varphi'$ from the set of triples $(T,s,w)$ to pairs $(f,s)$, we will first extend the tree parking function $s \in [n]^{m}$ to a tree parking function $s' \in [n]^{n}$ with $n$ drivers, then apply the bijection $\varphi$ described in the proof of Theorem~\ref{thm:bijection_parking_m=n}, and finally reduce $s'$ to the original parking function $s$. 
We only need to ensure that the extension from $s$ to $s'$ is done in such a way that the whole procedure can be reversed in a unique way. This can be done as follows.

Starting with a triple $(T,s,w)$, let us denote by $V_{\pi}$ the set of nodes which are occupied after the parking procedure, i.e., $V_{\pi} \colonequals \pi([m]) = \{\pi(k) : 1 \le k \le m\}$, where $\pi = \pi_{(T,s)}$ is the output-function of $(T,s)$. Let us arrange the $n-m$ free nodes in ascending order w.r.t.\ their labels: $V \setminus V_{\pi} = \{x_{1}, x_{2}, \dots, x_{n-m}\}$, with $x_{1} < x_{2} < \cdots < x_{n-m}$. Then we define the sequence $s' = (s_{1}', \dots, s_{n}') \in [n]^{n}$ as follows:
\begin{align*}
  s_{i}' \colonequals s_{i}, & \quad 1 \le i \le m, \\
	s_{m+i}' \colonequals x_{i}, & \quad 1 \le i \le n-m.
\end{align*}
Of course $s'$ is a parking function for $T$ since $s$ is a parking function for $T$ and every one of the drivers $m+1, \dots, n$ can park at their preferred parking space. Applying $\varphi$ from Theorem~\ref{thm:bijection_parking_m=n} gives an $n$-mapping $f$, such that $s'$ is a parking function for $f$. Thus the sequence $s = (s_{1}, \dots, s_{m}) = (s_{1}', \dots, s_{m}')$, which contains the preferences of the first $m$ drivers of $s'$, is a parking function for $f$.
We define the pair $(f,s)$ to be the outcome of $\varphi'$.

As for the case $m=n$, it holds that the parking paths of the drivers coincide for $T$ and $f$. 
In particular, it holds $\pi_{(f,s)} = \pi_{(T,s)}$ for the corresponding output-functions. Thus, $\varphi'$ can be reversed easily, since the extension from $s$ to $s'$ can also be constructed when starting with $f$.
\end{proof}

\subsection{Asymptotic considerations\label{ssec:asymptotics}}

Let us now turn to the asymptotic analysis of the number $M_{n,m}$ of $(n,m)$-mapping parking functions.
Due to Theorem~\ref{thm:link_trees_mappings_general}, our results for parking functions for mappings can automatically be translated  to results for parking functions for trees. In this context the following question will be of particular interest to us: How does the probability $p_{n,m} \colonequals \frac{M_{n,m}}{n^{n+m}}$ that a randomly chosen sequence of length $m$ on the set $[n]$ is a parking function for a randomly chosen $n$-mapping swap from being equal to $1$ (which is the case for $m=1$) to being close to $0$ (which is the case for $m=n$) when the ratio $\rho \colonequals \frac{m}{n}$ increases? 

In order to get asymptotic results for $M_{n,m}$ (and so for $F_{n,m}$ and $p_{n,m}$, too) we start with the representation \eqref{eqn:Mnm_Q}, which can be written as
\begin{equation}\label{eqn:relation_Mnm_Anm}
  M_{n,m} = \frac{n! m! n^{n-m}}{(n-m)!} A_{n,m},
\end{equation}
with
\begin{align}
  A_{n,m} &= [w^{m}] (1-2w) e^{2nw} (1-w)^{n-m-1} \notag\\
  & = \frac{1}{2 \pi i} \oint \frac{(1-2w) e^{2nw} (1-w)^{n-m-1}}{w^{m+1}} dw, \label{eqn:Anm_def}
\end{align}
where, in the latter expression, we choose as contour a suitable simple positively oriented closed curve around the origin, e.g., for each choice of $m$ and $n$, we may choose any such curve in the dotted disk $0 < |w| < 1$.

Next we will use the integral representation \eqref{eqn:Anm_def} of $A_{n,m}$ and apply saddle point techniques (see, e.g., \cite{DeBruijn1958,flajolet2009analytic} for instructive expositions of this method).
We write the integral as follows:
\begin{equation}\label{eqn:Anm_gwhw}
  A_{n,m} = \frac{1}{2 \pi i} \int\limits_{\Gamma} g(w) e^{n h(w)} dw,
\end{equation}
with $\Gamma$ a suitable contour and
\begin{align}\label{eqn:def_gw_hw}
  g(w) &\colonequals \frac{1-2w}{(1-w)w} \quad \text{and} \\
   h(w) = h_{n,m}(w) &\colonequals 2w + \left( 1-\frac{m}{n}\right)  \log(1-w) - \frac{m}{n} \log w. \notag
\end{align}
In the terminology of \cite{flajolet2009analytic} the integral \eqref{eqn:Anm_gwhw} has the form of a ``large power integral'' and saddle points of the relevant part $e^{n h(w)}$ of the integrand can thus be found as the zeros of the derivative $h'(w)$. The resulting equation
\begin{equation*}
  h'(w) = 2 - \left( 1-\frac{m}{n}\right)  \frac{1}{1-w} - \frac{m}{n} \frac{1}{w} = 0
\end{equation*}
yields the following two solutions:
\begin{equation*}
  w_{1} = \frac{m}{n} \quad \text{and} \quad w_{2} = \frac{1}{2}.
\end{equation*}
The idea of the saddle point method is to choose a suitable integration contour passing through (or at least passing close to) the saddle point lying closer to the origin, such that the main contribution of the integral comes from a small part of the curve containing the saddle point. Thus, one chooses the contour in such a way that, locally around the saddle point, it follows the steepest descent lines.

\begin{figure}
\begin{tabular}{ccc}
\includegraphics[width=0.32\textwidth]{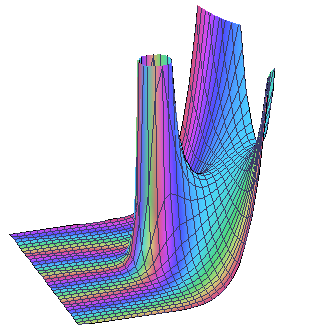} & \includegraphics[width=0.32\textwidth]{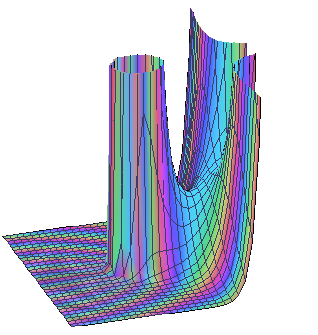} & \includegraphics[width=0.32\textwidth]{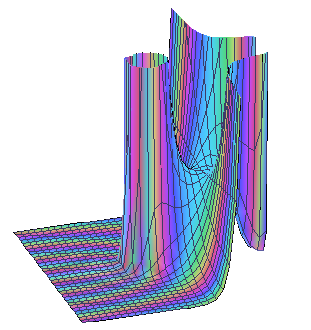} \\
\includegraphics[width=0.32\textwidth]{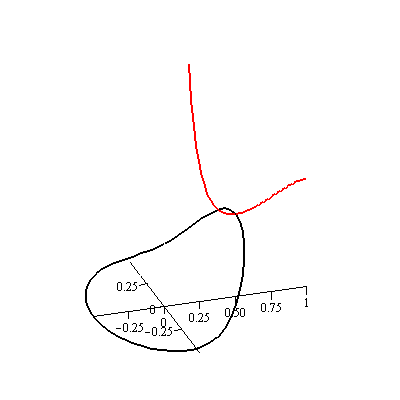} & \includegraphics[width=0.32\textwidth]{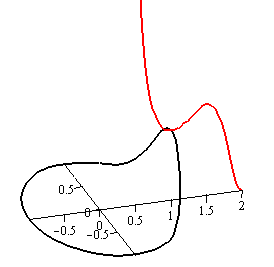} & \includegraphics[width=0.32\textwidth]{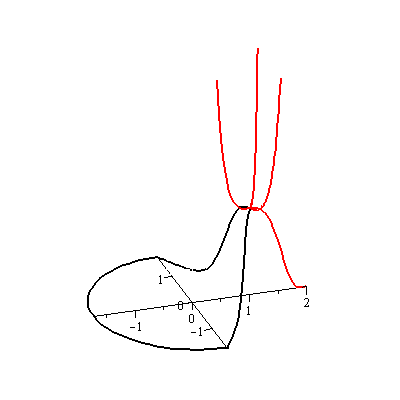} \\
\includegraphics[width=0.27\textwidth]{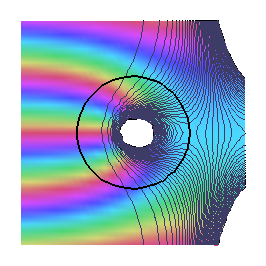} & \includegraphics[width=0.27\textwidth]{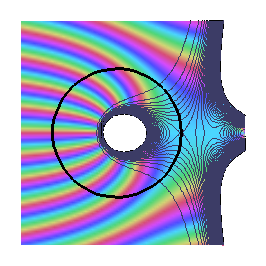} & \includegraphics[width=0.27\textwidth]{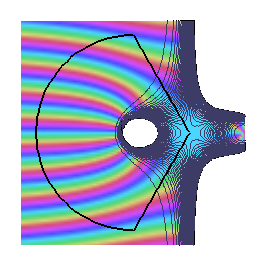} 
\end{tabular}
\caption{Plots of the modulus of the function $e^{nh(w)}$ near the saddle point for the three different regions; to the left the case $\rho<1/2$, in the middle $\rho>1/2$ and to the right $\rho=1/2$.
In the middle row: plots of the integration paths and steepest ascent/descent lines. The functions depicted here correspond to $n=12$ and $m=3$, $m=9$ and $m=6$ (from left to right).}
\label{fig:saddlepoints}
\end{figure}

The present situation is illustrated in Figure~\ref{fig:saddlepoints}.
In our asymptotic analysis we will have to distinguish whether $w_{1} < w_{2}$, $w_{1} > w_{2}$ or $w_{1} = w_{2}$. 
Actually, we will restrict ourselves to the cases $(i)$ $\rho = \frac{m}{n} \le \frac{1}{2} - \delta$ (with an arbitrary small, but fixed constant $\delta > 0$), $(ii)$ $\rho = \frac{m}{n} \ge \frac{1}{2} + \delta$, and $(iii)$ $\rho = \frac{m}{n} = \frac{1}{2}$, but remark that the transient behaviour of the sequences $M_{n,m}$, etc.\ for $m \sim \frac{n}{2}$ could be described via Airy functions as illustrated in \cite{BanFlaSchaefSor2001}.

We can sum up our results in the following theorem.
\begin{theorem}\label{thm:asymp_mappings}
The total number $M_{n,m}$ of $(n,m)$-mapping parking functions is asymptotically, for $n \to \infty$, given as follows (where $\delta$ denotes an arbitrary small, but fixed, constant):
\begin{align*}
M_{n,m} \sim \begin{cases} \frac{n^{n+m+\frac{1}{2}} \sqrt{n-2m}}{n-m}, & \text{for $1 \le m \le (\frac{1}{2} - \delta) n$},\\
\frac{\sqrt{2} \, 3^{\frac{1}{6}} \Gamma(\frac{2}{3}) n^{\frac{3n}{2}-\frac{1}{6}}}{\sqrt{\pi}}, & \text{for $m = \frac{n}{2}$},\\
\frac{m!}{(n-m)!} \cdot \frac{n^{2n-m+\frac{3}{2}} 2^{2m-n+1}}{(2m-n)^{\frac{5}{2}}}, & \text{for $(\frac{1}{2}+\delta) n \le m \le n$}. \end{cases}
\end{align*}
\end{theorem}

Let us fix the ratio $\rho=m/n$. 
This ratio can be interpreted as a ``load factor''--a term used in open addressing hashing.
Then the  asymptotic behaviour of the probabilities $p_{n,m}=p_{n,\rho n}$ follows immediately.
\newpage
\begin{coroll}\label{cor:asymptotic_pnrho}
The probability $p_{n,m}$ that a randomly chosen pair $(f,s)$, with $f$ an $n$-mapping and $s$ a sequence in $[n]^{m}$, represents a parking function is asymptotically, for $n \to \infty$ and $m = \rho n$ with $0 < \rho < 1$ fixed, given as follows:
\begin{equation*}
p_{n,m} \sim 
\begin{cases} C_<(\rho), & \text{for $0 < \rho < \frac{1}{2}$},\\
C_{1/2} \cdot n^{-1/6}, & \text{for $\rho = 1/2$},\\
C_{>}(\rho) \cdot n^{-1} \cdot (D_{>}(\rho))^n, & \text{for $1/2 < \rho < 1$}, 
\end{cases}
\end{equation*}
with
\begin{align*}
C_{<}(\rho) & = \frac{\sqrt{1-2\rho}}{1-\rho}, & C_{1/2} & = \sqrt{\frac{6}{\pi}}\frac{\Gamma(2/3)}{ 3^{1/3}} \approx 1.298\dots, \\
C_{>}(\rho) & = 2 \cdot \sqrt{\frac{\rho}{(1-\rho)(2\rho-1)^5}}, & D_{>}(\rho) & = \left( \frac{4 \rho}{e^2}\right)^\rho \frac{e}{2 (1-\rho)^{1-\rho}}.
\end{align*}
\end{coroll}

From Corollary~\ref{cor:asymptotic_pnrho} it follows that the limiting probability $L(\rho) \colonequals \lim_{n \to \infty} p_{n,\rho n}$ that   all drivers can park successfully for a load factor $\rho$ is given as follows :
\begin{equation*}
  L(\rho) = \begin{cases} \frac{\sqrt{1-2\rho}}{1-\rho}, & \quad \text{for $0 \le \rho \le \frac{1}{2}$},\\ 0, & \quad \text{for $\frac{1}{2} \le \rho \le 1$}.\end{cases}
\end{equation*}
See Figure~\ref{fig:asymptotics:prho} for an illustration: On the left hand-side the limiting distribution $L(\rho)$ is plotted and on the right hand-side the exact probabilities $p_{n,\rho n}$ for some values of $n$ can be found. 
\begin{figure}
\begin{minipage}[b]{0.4\linewidth}
\centering
\includegraphics[height=6cm]{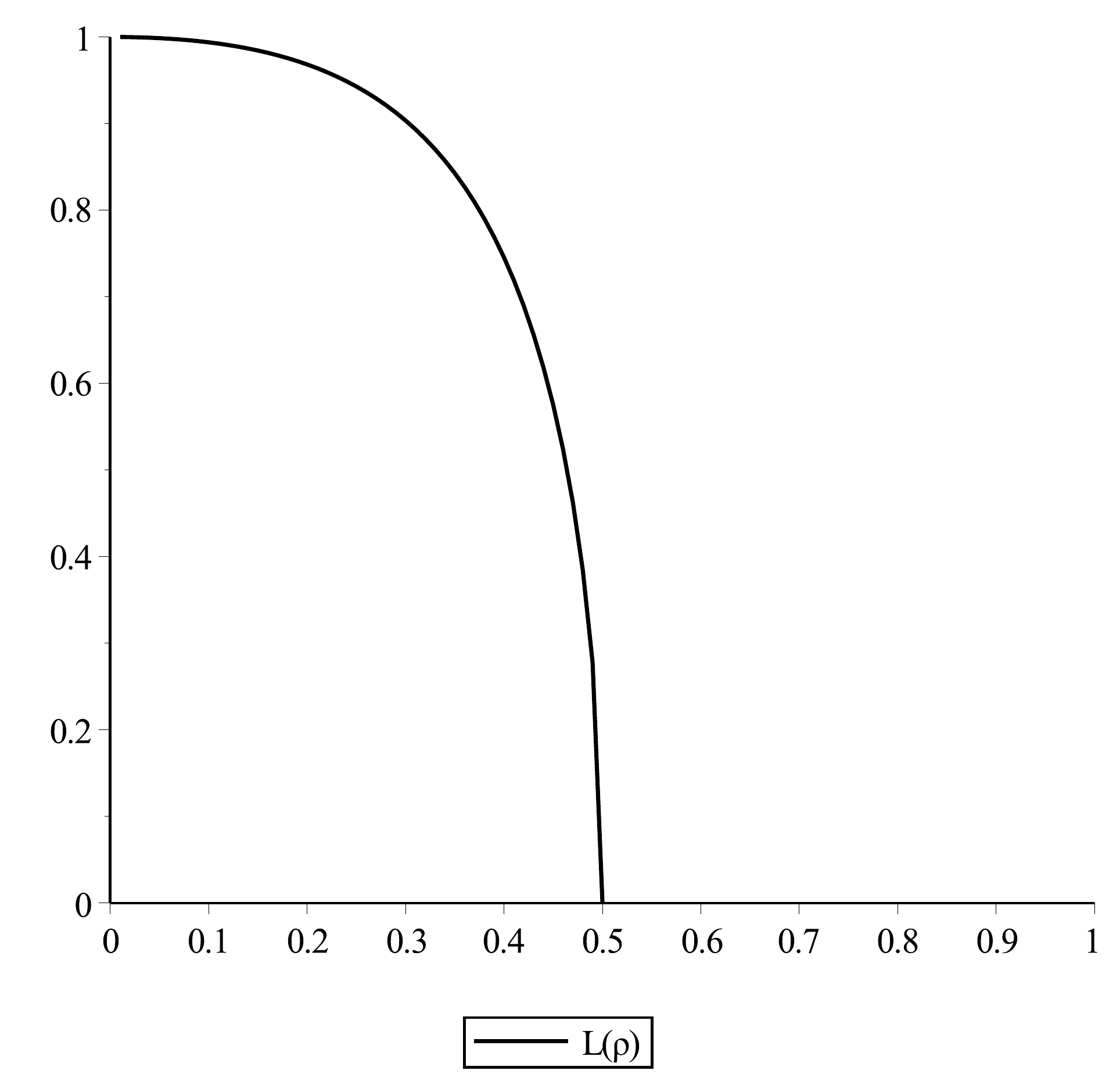}
\end{minipage}
\hspace{0.1cm}
\begin{minipage}[b]{0.55\linewidth}
\centering
\includegraphics[height=6cm]{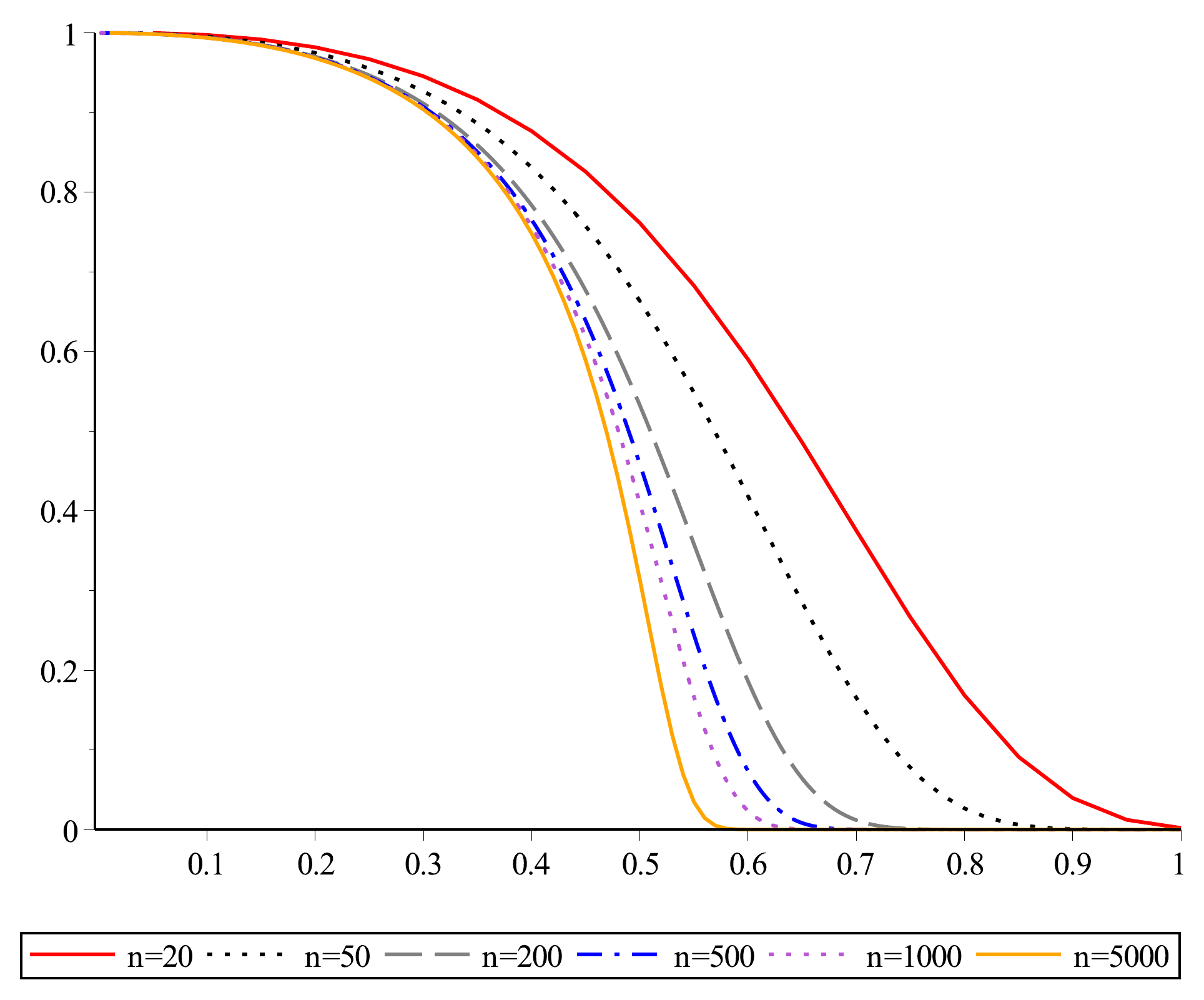}
\end{minipage}
\caption{To the left: The limiting probability $L(\rho)$ that all drivers are able to park successfully in a mapping, for a load factor $0 \le \rho \le 1$.
To the right: The exact probabilities $p_{n,\rho n}$ for $n=20, 50, 200, 500, 1000, 5000$.\label{fig:asymptotics:prho}}
\end{figure}

The proof of Theorem~\ref{thm:asymp_mappings} is given in Sections~\ref{sec:asymp_rho_small}-\ref{sec:asymp_rho_medium}.
As we can see from Theorem~\ref{thm:asymp_mappings} and Corollary~\ref{cor:asymptotic_pnrho}, the most interesting region for us is case $(i)$, i.e., where less than half of the parking spaces are occupied. We will thus provide the calculations for this region in detail, whereas the application of the saddle point method is only sketched for the other two regions. Before we continue with the computations, we want to comment on relations between mapping parking functions and ordered forests of unrooted trees.
\begin{remark}
  Let $G_{n,m}$ denote the number of ordered forests, i.e., sequences of $m$ unrooted labelled trees and comprised of $n$ nodes in total. The problem of evaluating $G_{n,m}$ asymptotically by using saddle point techniques has been considered in \cite[p.~603f.]{flajolet2009analytic}; as has been mentioned there, it is also relevant to the analysis of random graphs during the phase where a giant component has not yet emerged \cite{FlaKnuPit1989}.
	
	Since there are $n^{n-2}$ unrooted labelled trees of size $n \ge 1$, the exponential generating function $U(z) = \sum_{n \ge 1} n^{n-2} \frac{z^{n}}{n!}$ is given by $U(z) = T(z) - \frac{T^{2}(z)}{2}$, with $T(z)$ the tree-function.
	Thus  the numbers $G_{n,m}$ can be obtained as follows:
	\begin{align*}
	  G_{n,m} & = n! [z^{n}] U(z)^{m} = n! [z^{n}] \left(T(z)-\frac{T^{2}(z)}{2}\right)^{m} \\
	  & = \frac{n!}{2\pi i} \oint \left(T(z)-\frac{T^{2}(z)}{2}\right)^{m} \frac{dz}{z^{n+1}},
	\end{align*}
	by using a suitable contour around the origin. The substitution $z = \frac{T}{e^{T}}$ leads to 
	\begin{equation*}
	  G_{n,m} = \frac{n!}{2 \pi i} \oint \frac{e^{nT} \left(1-\frac{T}{2}\right)^{m} (1-T)}{T^{n-m+1}} dT.
	\end{equation*}
	After substituting $T = 2w$, one obtains
	\begin{align*}
	  G_{n,m} = & \frac{n!}{2 \pi i} \oint \frac{e^{2nw} \left(1-w\right)^{m} (1-2w)}{2^{n-m} w^{n-m+1}} dw \\
		= &\frac{n!}{2^{n-m}} [w^{n-m}] (1-2w) e^{2nw} (1-w)^{m},
	\end{align*}
	and so
	\begin{equation}\label{eqn:Gnm_otherexpression}
	  [w^{m}] (1-2w) e^{2nw} (1-w)^{n-m} = \frac{2^{m}}{n!} G_{n,n-m}.
	\end{equation}
	Comparing Equation~\eqref{eqn:Gnm_otherexpression} with Equation~\eqref{eqn:Anm_def} for $A_{n,m}$ suggests that the same phase change behaviour occurs for $A_{n,m}$ (and thus also $M_{n,m}$) and for $G_{n,n-m}$ as studied in \cite{{flajolet2009analytic}}.
\end{remark}
\begin{remark}
By  slightly adapting the considerations made in the previous remark we can even express the numbers $M_{n,m}$ directly via the number of ordered forests. Namely, let $\tilde{G}_{n,m}$ denote the number of ordered forests made of one rooted labelled tree followed by $m-1$ unrooted labelled trees and comprised of $n$ nodes in total. This yields
\begin{align*}
  \tilde{G}_{n,m} & = n! [z^{n}] T(z) U(z)^{m-1} = n! [z^{n}] T(z) \left(T(z)-\frac{T^{2}(z)}{2}\right)^{m-1} \\
	& = \frac{n!}{2\pi i} \oint T(z) \left(T(z) - \frac{T^{2}(z)}{2}\right)^{m-1} \frac{dz}{z^{n+1}},
\end{align*}
and, after the substitutions $z = \frac{T}{e^{T}}$ and $T=2w$, we end up with
\begin{align*}
  \tilde{G}_{n,m} & = \frac{n!}{2^{n-m}} \oint \frac{e^{2nw} (1-w)^{m-1} (1-2w)}{w^{n-m+1}} dw \\
  &= \frac{n!}{2^{n-m}} [w^{n-m}] (1-2w) e^{2nw} (1-w)^{m-1}.
\end{align*}
We get
\begin{equation*}
  A_{n,m} = \frac{2^{m}}{n!} \tilde{G}_{n,n-m},
\end{equation*}
and thus we are able to express $M_{n,m}$ by counting a certain number of ordered forests:
\begin{equation*}
  M_{n,m} = \frac{m! 2^{m} n^{n-m}}{(n-m)!} \tilde{G}_{n,n-m}.
\end{equation*}
\end{remark}

\subsubsection{The region $\rho \le \frac{1}{2} - \delta$\label{sec:asymp_rho_small}}

The geometry of the modulus of the integrand of \eqref{eqn:Anm_gwhw} as depicted in Figure~\ref{fig:saddlepoints} is easily described: There is a simple dominant saddle point at $w=w_{1}$, where the surface resembles an ordinary horse saddle or a mountain pass.
There are two steepest descent/steepest ascent lines: one following the real axis and one parallel to the imaginary axis. 
It is thus natural to adopt an integration contour that lies close to the steepest ascent and steepest descent line perpendicular to the real axis. 
In Equation~\eqref{eqn:Anm_gwhw} we thus choose the contour $\Gamma$ to be a circle centered at the origin and passing through the dominant saddle point $w_{1}$, i.e., it has radius $r = \rho$.

Using the parametrization $\Gamma = \{w = \rho e^{i \phi} : \phi \in [-\pi,\pi]\}$, we obtain from \eqref{eqn:Anm_gwhw} the representation
\begin{equation}\label{eqn:Anm_startI}
  A_{n,m} = \frac{1}{2 \pi} \int_{-\pi}^{\pi} \rho e^{i \phi} g(\rho e^{i \phi}) e^{n h(\rho e^{i \phi})} d\phi.
\end{equation}
Next we want to find a suitable splitting of the integral into the central approximation and the remainder.
That is, we need to choose a proper value $\phi_{0} = \phi_{0}(n,m)$ to write the contour as $\Gamma = \Gamma_{1} \cup \Gamma_{2}$, with $\Gamma_{1} \colonequals \{w = \rho e^{i \phi} : \phi \in [-\phi_{0}, \phi_{0}]\} $ and $\Gamma_{2} \colonequals \{w = \rho e^{i \phi} : \phi \in [-\pi,-\phi_{0}] \cup [\phi_{0},\pi]\}$ yielding the representation $A_{n,m} = I_{n,m}^{(1)} + I_{n,m}^{(2)}$, such that $I_{n,m}^{(2)} = o\left( I_{n,m}^{(1)}\right) $, where
\begin{align*}
  I_{n,m}^{(1)} & \colonequals \frac{1}{2 \pi} \int_{\Gamma_{1}} \rho e^{i \phi} g(\rho e^{i \phi}) e^{n h(\rho e^{i \phi})} d\phi \quad \text{ and } \\ 
	I_{n,m}^{(2)} & \colonequals \frac{1}{2 \pi} \int_{\Gamma_{2}} \rho e^{i \phi} g(\rho e^{i \phi}) e^{n h(\rho e^{i \phi})} d\phi.
\end{align*}
To do this we consider the local expansion of the integral around $\phi=0$; the following results are obtained by straightforward computations, which are thus omitted:
\begin{align*}
  \rho e^{i \phi} g(\rho e^{i \phi})  = &\frac{1-\frac{2m}{n} e^{i \phi}}{1-\frac{m}{n} e^{i \phi}} = \frac{1-\frac{2m}{n}}{1-\frac{m}{n}} \cdot \left(1+\mathcal{O}\left( \frac{m \phi}{n}\right) \right), \\
	n h(\rho e^{i \phi})  = & n \left(\frac{2m}{n} e^{i \phi} + \left( 1-\frac{m}{n}\right)  \log\left( 1-\frac{m}{n}e^{i \phi}\right)  - \frac{m}{n} \log\left( \frac{m}{n} e^{i \phi}\right)\right)  \\
	 = &2m + (n-m) \log\left( 1-\frac{m}{n}\right)  - m \log\left( \frac{m}{n}\right)  \\
	 & -\frac{m(n-2m)}{2(n-m)}\phi^{2} + \mathcal{O}(m \phi^{3}),
\end{align*}
yielding
\begin{align*}
\rho e^{i \phi} g(\rho e^{i \phi}) e^{n h(\rho e^{i \phi})} & = \left( 1-\frac{2m}{n}\right)  \left( \frac{n}{m}\right) ^{m} e^{2m} \left( 1-\frac{m}{n}\right) ^{n-m-1} e^{-\left( \frac{m(n-2m)}{2(n-m)}\right)  \phi^{2}} \cdot\\
& \qquad \cdot \left(1+\mathcal{O}\left( m \phi^{3}\right)  + \mathcal{O}\left( \frac{m \phi}{n}\right) \right).
\end{align*}

From the latter expansion we obtain that we shall choose $\phi_{0}$, such that $m \phi_{0}^{2} \to \infty$ (then the central approximation contains the main contributions) and $m \phi_{0}^{3} \to 0$ (then the remainder term is asymptotically negligible). E.g., we may choose $\phi_{0} = m^{-\frac{1}{2} + \frac{\epsilon}{3}}$, for a constant $0 < \epsilon < \frac{1}{2}$.
With such a choice of $\phi_{0}$ we obtain for the integral $I_{n,m}^{(1)}$ the following asymptotic expansion:
\begin{align*}
  I_{n,m}^{(1)}  = &\frac{e^{2m}}{2\pi} \left( 1-\frac{2m}{n}\right)  \left( \frac{n}{m}\right) ^{m}  \left( 1-\frac{m}{n}\right) ^{n-m-1} \cdot \int_{-\phi_{0}}^{\phi_{0}} e^{-\left(\frac{m(n-2m)}{2(n-m)}\right) \phi^{2}} d\phi \cdot \\
  & \cdot \left(1+\mathcal{O}\left(m^{-\frac{1}{2}+\epsilon}\right)\right) \\
	 = & \frac{e^{2m}}{2\pi} \left( 1-\frac{2m}{n}\right)  \left( \frac{n}{m}\right) ^{m}  \left( 1-\frac{m}{n}\right) ^{n-m-1} \frac{1}{\sqrt{m}} \cdot \int_{-m^{\frac{\epsilon}{3}}}^{m^{\frac{\epsilon}{3}}} e^{- \left(\frac{n-2m}{2(n-m)}\right) t^{2}} dt \cdot \\
	 & \cdot \left(1+\mathcal{O}\left(m^{-\frac{1}{2}+\epsilon}\right)\right),
\end{align*}
where we used the substitution $\phi = \frac{t}{\sqrt{m}}$ for the latter expression.

For the so-called tail completion we use that
\begin{equation*}
  \int_{c}^{\infty} e^{- \alpha t^{2}} dt = \mathcal{O}\left(e^{-\alpha c^{2}}\right), \quad \text{for $c > 0$ and $\alpha > 0$},
\end{equation*}
which can be shown, e.g., via
\begin{equation*}
\int_{c}^{\infty} e^{- \alpha t^{2}} dt \le \sum_{j=0}^{\infty} e^{-\alpha (c+j)^{2}} = e^{-\alpha c^{2}} \cdot \sum_{j=0}^{\infty} e^{- \alpha j (2c+j)} = \mathcal{O}\left(e^{-\alpha c^{2}}\right),
\end{equation*}
since $\sum_{j=0}^{\infty} e^{-\alpha (c+j)^{2}}$ converges.

Thus we obtain
\begin{equation*}
  \int_{m^{\frac{\epsilon}{3}}}^{\infty} e^{-\left(\frac{n-2m}{2(n-m)}\right) t^{2}} dt = \mathcal{O}\left(e^{-\left(\frac{n-2m}{2(n-m)}\right) m^{\frac{2 \epsilon}{3}}}\right),
\end{equation*}
which yields a subexponentially small and thus negligible error term. Using this, we may proceed in the asymptotic evaluation of $I_{n,m}^{(1)}$ and get
\begin{align*}
  I_{n,m}^{(1)} = & \frac{e^{2m}}{2 \pi \sqrt{m}} \left( 1-\frac{2m}{n}\right)  \left( \frac{n}{m}\right) ^{m} \left( 1-\frac{m}{n}\right) ^{n-m-1}  \\
  & \cdot\int_{-\infty}^{\infty} e^{-\left(\frac{n-2m}{2(n-m)}\right) t^{2}} dt \cdot \left(1+\mathcal{O}\left(m^{-\frac{1}{2}+\epsilon}\right)\right).
\end{align*}
Using
\begin{equation*}
  \int_{-\infty}^{\infty} e^{-\alpha t^{2}} dt = \frac{\sqrt{\pi}}{\sqrt{\alpha}}, \quad \text{for $\alpha > 0$},
\end{equation*}
the Gaussian integral occurring can be evaluated easily, which yields
\begin{align*}
  I_{n,m}^{(1)}  = &\frac{e^{2m}}{2 \pi \sqrt{m}} \left( 1-\frac{2m}{n}\right)  \left( \frac{n}{m}\right) ^{m} \left( 1-\frac{m}{n}\right) ^{n-m-1} \frac{\sqrt{2\pi}}{\sqrt{1-\frac{m}{n-m}}} \\
  & \cdot \left(1+\mathcal{O}\left(m^{-\frac{1}{2}+\epsilon}\right)\right) \\
	 = & \frac{e^{2m} (n-m)^{n-m-\frac{1}{2}} \sqrt{n-2m} }{\sqrt{2\pi} \, m^{m+\frac{1}{2}} n^{n-2m}} \cdot \left(1+\mathcal{O}\left(m^{-\frac{1}{2}+\epsilon}\right)\right).
\end{align*}

Next we consider the remainder integral
\begin{align*}
  I_{n,m}^{(2)} = & \frac{1}{2\pi} \left( 1-\frac{2m}{n}\right)  \cdot \\
  & \cdot \int_{\Gamma_{2}} \frac{1-\frac{2m}{n}e^{i \phi}}{1-\frac{m}{n}e^{i \phi}} e^{n \left(2 \frac{m}{n} e^{i \phi} + (1-\frac{m}{n}) \log(1-\frac{m}{n} e^{i \phi}) - \frac{m}{n} \log(\frac{m}{n}) - \frac{m}{n} i \phi\right)} d\phi.
\end{align*}
To estimate the integrand we use the obvious bounds
\begin{equation*}
  \left\vert 1-\frac{2m}{n} e^{i \phi}\right\vert \le 1+\frac{2m}{n} \quad \text{and} \quad \frac{1}{\left\vert 1-\frac{m}{n} e^{i \phi} \right\vert} \le \frac{1}{1-\frac{m}{n}},
\end{equation*}
as well as the following:
\begin{align*}
  & \left\vert e^{n \left(2 \frac{m}{n} e^{i \phi} + (1-\frac{m}{n}) \log(1-\frac{m}{n} e^{i \phi}) - \frac{m}{n} \log(\frac{m}{n}) - \frac{m}{n} i \phi\right)} \right\vert \\
  = & \left( \frac{n}{m}\right) ^{m} e^{n \left(2 \rho \cos \phi + \frac{1-\rho}{2} \log(1-2\rho \cos \phi + \rho^{2})\right)}.
\end{align*}
This yields
\begin{equation*}
  \left\vert I_{n,m}^{(2)} \right\vert \le \frac{1}{2\pi} \frac{(1-\frac{2m}{n}) (1+\frac{2m}{n})}{1-\frac{m}{n}} \cdot \left( \frac{n}{m}\right) ^{m} \cdot \int_{\Gamma_{2}} e^{n \left(2 \rho \cos \phi + \frac{1-\rho}{2} \log(1-2\rho \cos \phi + \rho^{2})\right)} d\phi.
\end{equation*}
Considering the function
\begin{equation*}
  \tilde{H}(x) \colonequals 2 \rho x + \frac{1-\rho}{2} \log(1-2\rho x + \rho^{2}),
\end{equation*}
it can be shown by applying standard calculus that $\tilde{H}(x)$ is a monotonically increasing function for $x \in [-1,1]$. Setting $x=\cos \phi$ it follows that amongst all points of the contour $\Gamma_{2}$ the integrand reaches its maximum at $\phi = \phi_{0}$.
Thus, we obtain
\begin{align*}
  \left\vert I_{n,m}^{(2)} \right\vert  \le & \frac{(1-\frac{2m}{n}) (1+\frac{2m}{n})}{1-\frac{m}{n}} \cdot \left( \frac{n}{m}\right)^{m} \cdot e^{2m \cos \phi_{0} + \frac{n-m}{2} \log\left( 1-\frac{2m}{n} \cos \phi_{0} + (\frac{m}{n})^{2}\right) } \\
	 \le & 2 \cdot \left(\frac{n}{m}\right)^{m} \cdot e^{2m \cos \phi_{0} + \frac{n-m}{2} \log\left( 1-\frac{2m}{n} \cos \phi_{0} + (\frac{m}{n})^{2}\right) } \\
	 =	&2 \left(\frac{n}{m}\right)^{m} e^{2m} e^{\frac{n-m}{2} \log\left( 1-\frac{2m}{n} + \left(\frac{m}{n}\right)^{2}\right) } \cdot \\
	& \cdot e^{2m(\cos \phi_{0} -1) + \frac{n-m}{2} \log\left(\frac{1-\frac{2m}{n} \cos \phi_{0} + \left(\frac{m}{n}\right)^{2}}{1-\frac{2m}{n}+\left(\frac{m}{n}\right)^{2}}\right)}\\
	 = & 2 \left(\frac{n}{m}\right)^{m} e^{2m} \left( 1-\frac{m}{n}\right) ^{n-m} \cdot e^{2m(\cos \phi_{0} -1)+\frac{n-m}{2} \log\left(1-\frac{\frac{2m}{n}(\cos \phi_{0}-1)}{(1-\frac{m}{n})^{2}}\right)}.
\end{align*}

Using the estimates
\begin{equation*}
  \log(1-x) \le -x, \quad \text{for $x <1$}, \quad \text{and} \quad \cos x - 1 \le -\frac{x^{2}}{6}, \quad \text{for $x \in [-\pi,\pi]$},
\end{equation*}
we may proceed as follows:
\begin{align*}
  |I_{n,m}^{(2)}| & \le 2 \left(\frac{n}{m}\right)^{m} e^{2m} \left( 1-\frac{m}{n}\right) ^{n-m} \cdot e^{m(\cos \phi_{0} -1) \cdot \left( 2-\frac{1}{1-\frac{m}{n}}\right) } \\
  & \le 2 \left(\frac{n}{m}\right)^{m} e^{2m} \left( 1-\frac{m}{n}\right) ^{n-m} \cdot e^{-\frac{\phi_{0}^{2}}{6} m \left( 2-\frac{1}{1-\frac{m}{n}}\right) } \\
	& = 2 \left(\frac{n}{m}\right)^{m} e^{2m} \left( 1-\frac{m}{n}\right) ^{n-m} \cdot e^{-\frac{1}{6} \left( 2-\frac{1}{1-\frac{m}{n}}\right)  m^{\frac{2\epsilon}{3}}}.
\end{align*}
Thus we obtain
\begin{equation*}
  |I_{n,m}^{(2)}| = |I_{n,m}^{(1)}| \cdot \mathcal{O}\left(\sqrt{m} \, e^{-c m^{\frac{2\epsilon}{3}}}\right), \quad \text{with } c=\frac{1}{6} \left( 2-\frac{1}{1-\frac{m}{n}}\right) ,
\end{equation*}
i.e., $I_{n,m}^{(2)}$ is subexponentially small compared to $I_{n,m}^{(1)}$.

Combining these results we get
\begin{equation*}
  A_{n,m} = \frac{(n-m)^{n-m-\frac{1}{2}} \sqrt{n-2m} \, e^{2m}}{\sqrt{2\pi} \, m^{m+\frac{1}{2}} n^{n-2m}} \cdot \left(1+\mathcal{O}\left(m^{-\frac{1}{2}+\epsilon}\right)\right)
\end{equation*}
and, by using \eqref{eqn:relation_Mnm_Anm} and applying Stirling's approximation formula for the factorials, 
\begin{align}
  M_{n,m} & = \frac{n! m! n^{n-m} (n-m)^{n-m-\frac{1}{2}} \sqrt{n-2m} \, e^{2m}}{\sqrt{2\pi} \, (n-m)! m^{m+\frac{1}{2}} n^{n-2m}} \cdot \left(1+\mathcal{O}\left(m^{-\frac{1}{2}+\epsilon}\right)\right) \notag\\
	& = \frac{n^{n+m} \sqrt{1-\frac{2m}{n}}}{1-\frac{m}{n}} \cdot \left(1+\mathcal{O}\left(m^{-\frac{1}{2}+\epsilon}\right)\right).
	\label{eqn:Mnm_asymp_res}
\end{align}

Note that according to the remainder term in \eqref{eqn:Mnm_asymp_res} we have only shown the required result for $m \to \infty$. However, again starting with \eqref{eqn:Anm_startI}, we can easily show a refined bound on the error term for small $m$. Namely, we may write the integral as follows:
\begin{equation*}
  A_{n,m} = \frac{1}{2\pi} \left(\frac{n}{m}\right)^{m} \cdot \int_{-\pi}^{\pi} \frac{e^{2m e^{i\phi}} (1-\frac{m}{n}e^{i\phi})^{n-m-1} (1-\frac{2m}{n} e^{i\phi})}{e^{im\phi}} d\phi,
\end{equation*}
and use for $m=o(\sqrt{n})$ the expansions
\begin{equation*}
  \left( 1-\frac{m}{n}e^{i\phi}\right) ^{n-m-1} = e^{-me^{i\phi}} \cdot \left(1+\mathcal{O}\left(\frac{m^{2}}{n}\right)\right), \quad 1-\frac{2m}{n} e^{i\phi} = 1 + \mathcal{O}\left( \frac{m}{n}\right) ,
\end{equation*}
which gives
\begin{equation*}
  A_{n,m} = \frac{1}{2\pi} \left(\frac{n}{m}\right)^{m} \cdot \int_{-\pi}^{\pi} \frac{e^{me^{i\phi}}}{e^{im\phi}} d\phi \cdot \left(1+\mathcal{O}\left(\frac{m^{2}}{n}\right)\right).
\end{equation*}
Using the substitution $z=e^{i\phi}$ this yields for $m=o(\sqrt{n})$
\begin{equation*}
  A_{n,m} = \frac{1}{2\pi i} \left(\frac{n}{m}\right)^{m} \cdot \oint \frac{e^{mz}}{z^{m+1}} dz \cdot \left(1+\mathcal{O}\left(\frac{m^{2}}{n}\right)\right) = \frac{n^{m}}{m!} \cdot \left(1+\mathcal{O}\left(\frac{m^{2}}{n}\right)\right)
\end{equation*}
and, again by using \eqref{eqn:relation_Mnm_Anm} and applying Stirling's approximation formula for the factorials, furthermore
\begin{align*}
  M_{n,m} & = \frac{n! n^{n}}{(n-m)!} \cdot \left(1+\mathcal{O}\left(\frac{m^{2}}{n}\right)\right) = n^{n+m} \cdot \left(1+\mathcal{O}\left(\frac{m^{2}}{n}\right)\right)\\
  &  = \frac{n^{n+m} \sqrt{1-\frac{2m}{n}}}{1-\frac{m}{n}} \cdot \left(1+\mathcal{O}\left(\frac{m^{2}}{n}\right)\right).
\end{align*}

\subsubsection{The region $\rho \ge \frac{1}{2} + \delta$\label{sec:asymp_rho_large}}

For this region we choose in \eqref{eqn:Anm_gwhw} the contour $\Gamma$ to be a circle centered at the origin and passing through the dominant saddle point $w_{2}$, i.e., it has radius $r = \frac{1}{2}$.
Using the parametrization $\Gamma = \{w = \frac{1}{2} e^{i \phi} : \phi \in [-\pi,\pi]\}$, we obtain from \eqref{eqn:Anm_gwhw} the representation
\begin{equation}\label{eqn:Anm_startII}
  A_{n,m} = \frac{1}{2 \pi} \int_{-\pi}^{\pi} \frac{1}{2} e^{i \phi} g(\frac{1}{2} e^{i \phi}) e^{n h(\frac{1}{2} e^{i \phi})} d\phi,
\end{equation}
with functions $g(w)$ and $h(w)$ defined in \eqref{eqn:def_gw_hw}.

As in the previous region we expand the integrand in \eqref{eqn:Anm_startII} around $\phi=0$ to find a suitable choice for $\phi_{0}$ to split the integral. However, due to cancellations, we require a more refined expansion which  can again be obtained by straightforward computations. Namely, we obtain
\begin{align*}
  \frac{e^{i \phi}}{2} g\left( \frac{e^{i \phi}}{2} \right)   = & -2i \phi + 3 \phi^{2} + \mathcal{O}(\phi^{3}), \\
	n h\left( \frac{e^{i \phi}}{2} \right)   = & n + (2m-n) \log 2 - \left( m-\frac{n}{2}\right)  \phi^{2} + i \left( \frac{5n}{6} - m\right)  \phi^{3} + \mathcal{O}(n \phi^{4}),
\end{align*}
which gives
\begin{align*}
  \frac{e^{i \phi}}{2}  g\left( \frac{e^{i \phi}}{2} \right)  e^{n h\left( \frac{e^{i \phi}}{2} \right) }  = &e^{n} 2^{2m-n} e^{-(m-\frac{n}{2})\phi^{2}} \cdot \\
	& \cdot \Bigg(-2i \phi +3\phi^{2} + \frac{5n-6m}{3} \phi^{4} + \\
	& \quad + \mathcal{O}(\phi^{3}) + \mathcal{O}(n\phi^{5}) + \mathcal{O}(n^{2} \phi^{7}) \Bigg).
\end{align*}

Thus we may choose $\phi_{0} = n^{-\frac{1}{2}+\epsilon}$, with $0 < \epsilon < \frac{1}{6}$ to split the contour $\Gamma = \Gamma_{1} \cup \Gamma_{2}$, with $\Gamma_{1} = \{w = \frac{1}{2} e^{i \phi} : \phi \in [-\phi_{0},\phi_{0}]\} $ and $\Gamma_{2} = \{w = \frac{1}{2} e^{i \phi} : \phi \in [-\pi,-\phi_{0}] \cup [\phi_{0},\pi]\} $.
Let us again denote by $I_{n,m}^{(1)}$ and $I_{n,m}^{(2)}$ the contribution of the integral in the representation \eqref{eqn:Anm_startII} over $\Gamma_{1}$ and $\Gamma_{2}$, respectively.

For $I_{n,m}^{(1)}$ we use the above expansion for the integrand and obtain after simple manipulations
\begin{equation*}
  I_{n,m}^{(1)} = c_{n,m}  \int_{-\phi_{0}}^{\phi_{0}} e^{-(m-\frac{n}{2})\phi^{2}}  \left(-2i \phi + 3 \phi^{2} + \frac{5n-6m}{3} \phi^{4} + \mathcal{O}(n^{-\frac{3}{2}+7\epsilon})\right) d\phi,
\end{equation*}
where the multiplicative factor $c_{n,m}$ is equal to $\frac{e^{n} 2^{2m-n}}{2\pi}$.
Again it holds that completing the tails only gives a subexponentially small error term and we obtain
\begin{align*}
  I_{n,m}^{(1)} & = c_{n,m}  \int_{-\infty}^{\infty} e^{-(m-\frac{n}{2})\phi^{2}}  \left(-2i \phi + 3 \phi^{2} + \frac{5n-6m}{3} \phi^{4} + \mathcal{O}(n^{-\frac{3}{2}+7\epsilon})\right) d\phi\\
	& =   \frac{c_{n,m}}{\sqrt{n}} \int_{-\infty}^{\infty} e^{-(\frac{m}{n}-\frac{1}{2})t^{2}} \left(-\frac{2it}{\sqrt{n}}+\frac{3t^{2}}{n} + \frac{5n-6m}{3} \frac{t^{4}}{n^{2}} + \mathcal{O}\left(n^{-\frac{3}{2}+7\epsilon}\right)\right) dt,
\end{align*}
where we used the substitution $\phi = \frac{t}{\sqrt{n}}$ to get the latter expression.

Using the integral evaluations (with $\alpha > 0$):
\begin{equation*}
  \int_{-\infty}^{\infty} t e^{-\alpha t^{2}} dt = 0, \quad \int_{-\infty}^{\infty} t^{2} e^{-\alpha t^{2}} dt = \frac{\sqrt{\pi}}{2 \alpha^{\frac{3}{2}}}, \quad \int_{-\infty}^{\infty} t^{4} e^{-\alpha t^{2}} dt = \frac{3 \sqrt{\pi}}{4 \alpha^{\frac{5}{2}}},
\end{equation*}
we obtain
\begin{align*}
  I_{n,m}^{(1)} & = \frac{1}{2\pi} e^{n} 2^{2m-n} \frac{1}{n^{\frac{3}{2}}} \left(\frac{3 \sqrt{\pi}}{2\left( \frac{m}{n}-\frac{1}{2}\right) ^{\frac{3}{2}}} + \frac{\left( 5-\frac{6m}{n}\right)  \sqrt{\pi}}{4 \left( \frac{m}{n}-\frac{1}{2}\right) ^{\frac{5}{2}}}\right) \cdot \left(1+\mathcal{O}\left(n^{-\frac{1}{2}+7\epsilon}\right)\right)\\
	& = \frac{e^{n} 2^{2m-n+1}}{\sqrt{2 \pi} \, n^{\frac{3}{2}} \left( \frac{2m}{n}-1\right) ^{\frac{5}{2}}} \cdot \left(1+\mathcal{O}\left(n^{-\frac{1}{2}+7\epsilon}\right)\right).
\end{align*}

Again, it can be shown that the main contribution of $A_{n,m}$ comes from $I_{n,m}^{(1)}$, i.e., that it holds
\begin{equation*}
  I_{n,m}^{(2)} = \frac{1}{2 \pi} \int_{\Gamma_{2}} \frac{1}{2} e^{i \phi} g\left( \frac{1}{2} e^{i \phi}\right)  e^{n h\left( \frac{1}{2} e^{i \phi}\right) } d\phi
	= o\left( I_{n,m}^{(1)}\right) ,
\end{equation*}
but here we omit these computations.
Thus we obtain
\begin{equation*}
  A_{n,m} = \frac{e^{n} 2^{2m-n+1}}{\sqrt{2 \pi} \, n^{\frac{3}{2}} \left( \frac{2m}{n}-1\right) ^{\frac{5}{2}}} \cdot \left(1+\mathcal{O}\left(n^{-\frac{1}{2}+7\epsilon}\right)\right),
\end{equation*}
and
\begin{equation*}
  M_{n,m} = \frac{m!}{(n-m)!} \cdot \frac{n^{2n-m-1} 2^{2m-n+1}}{\left( \frac{2m}{n}-1\right) ^{\frac{5}{2}}} \cdot \left(1+\mathcal{O}\left(n^{-\frac{1}{2}+7\epsilon}\right)\right).
\end{equation*}

\subsubsection{The monkey saddle for $\rho = 1/2$\label{sec:asymp_rho_medium}}

For $\rho = \frac{m}{n} = \frac{1}{2}$, the situation is slightly different to the previous regions since the two otherwise distinct saddle points coalesce to a unique double saddle point.
The difference in the geometry of the surface, i.e., of the modulus of the large power $e^{n \cdot h(w)}$ in \eqref{eqn:Anm_gwhw}, is that there are now three steepest descent lines and three steepest ascent lines departing from the saddle point (in contrast to two steepest descent lines and two steepest ascent lines for the case of a simple saddle point).
This explains why such saddle points are also referred to as ``monkey saddles'': they do not only offer space for two legs but also for a tail.
In this particular case the three steepest descent and steepest ascent lines departing from the saddle point at $w=w_{1}=w_{2}=\frac{1}{2}$ have angles $0, 2 \pi/3$ and $-2 \pi/3$ as can be seen in the middle right of Figure~\ref{fig:saddlepoints}.
This also follows from a local expansion of $h(w)$ as defined in \eqref{eqn:def_gw_hw} around $w=\frac{1}{2}$:
\begin{equation*}
  h(w) = 1 - \frac{8}{3} \left( w-\frac{1}{2}\right) ^{3} + \mathcal{O}\left(\left( w-\frac{1}{2}\right) ^{4}\right).
\end{equation*}

Thus, we may choose as integration contour two line segments joining the point $w=\frac{1}{2}$ with the imaginary axis at an angle of $-2 \pi/3$ and $2 \pi/3$, respectively, as well as a half circle centered at the origin and joining the two line segments. See the bottom right of Figure~\ref{fig:saddlepoints}. This yields $\Gamma = \Gamma_{1} \cup \Gamma_{2} \cup \Gamma_{3}$ and $A_{n,m} = I_{n,m}^{(1)} + I_{n,m}^{(2)} + I_{n,m}^{(3)}$ for the corresponding integrals, where we use the parametrizations $\Gamma_{1} \colonequals \left\lbrace \frac{1}{2}-e^{-\frac{2\pi i}{3}} t : t \in [-1,0]\right\rbrace $, $\Gamma_{2} \colonequals \left\lbrace \frac{1}{2}+e^{\frac{2\pi i}{3}} t : t \in [0,1]\right\rbrace $, and $\Gamma_{3} \colonequals \left\lbrace \frac{\sqrt{3}}{2} e^{i t} : t \in [\frac{\pi}{2}, \frac{3\pi}{2}]\right\rbrace $.

We first treat
\begin{align*}
  I_{n,m}^{(1)} & = \frac{1}{2\pi i} \int_{-1}^{0} \left( -e^{-\frac{2\pi i}{3}}\right)  g\left( \frac{1}{2}-e^{-\frac{2\pi i}{3}} t\right)  e^{n h\left( \frac{1}{2}-e^{-\frac{2\pi i}{3}} t\right) } dt\\
	& = \frac{1}{2\pi i} \int_{0}^{1} \left( -e^{-\frac{2\pi i}{3}}\right)  g\left( \frac{1}{2}+e^{-\frac{2\pi i}{3}} t\right)  e^{n h\left( \frac{1}{2}+e^{-\frac{2\pi i}{3}} t\right) } dt.
\end{align*}
In order to find a suitable choice $t_{0}$ for splitting the integral for the central approximation and the remainder we consider the expansion of the integrand around $t=0$, which can be obtained easily:
\begin{align}\label{eqn:Inm1_integrand_expansion}
& \frac{-e^{-\frac{2\pi i}{3}}}{2\pi i} g\left( \frac{1}{2}+e^{-\frac{2\pi i}{3}} t\right)  e^{n h\left( \frac{1}{2}+e^{-\frac{2\pi i}{3}} t\right) } \\
= & \frac{4 e^{n} e^{-\frac{4 \pi i}{3}}}{\pi i} t e^{-\frac{8}{3} n t^{3}} \cdot \left(1+\mathcal{O}(t^{2}) + \mathcal{O}(n t^{5})\right). \notag
\end{align}
Thus we obtain the restrictions $n t_{0}^{3} \to \infty$ and $n t_{0}^{5} \to 0$ which are, e.g., satisfied when choosing $t_{0} = n^{-\frac{1}{4}}$. This splitting yields $I_{n,m}^{(1)} = I_{n,m}^{(1,1)} + I_{n,m}^{(1,2)}$, for the integration paths $t \in [0,t_{0}]$ and $t \in [t_{0},1]$, respectively.

Using the local expansion of the integrand \eqref{eqn:Inm1_integrand_expansion} as well as the before-mentioned choice for $t_{0}$, the central approximation $I_{n,m}^{(1,1)}$ gives
\begin{align*}
  I_{n,m}^{(1,1)} = & \frac{4 e^{n} e^{-\frac{4 \pi i}{3}}}{\pi i} \int_{0}^{t_{0}} t e^{-\frac{8}{3} n t^{3}} dt \cdot \left(1+\mathcal{O}(n^{-\frac{1}{4}})\right) \\
  = & \frac{4 e^{n} e^{-\frac{4 \pi i}{3}}}{\pi i} \int_{0}^{\infty} t e^{-\frac{8}{3} n t^{3}} dt \cdot \left(1+\mathcal{O}(n^{-\frac{1}{4}})\right),
\end{align*}
since one can show easily that completing the integral only yields a subexponentially small error term. Moreover, also the remainder
\begin{equation*}
  I_{n,m}^{(1,2)} = \frac{1}{2\pi i} \int_{t_{0}}^{1} \left( -e^{-\frac{2\pi i}{3}}\right)  g\left( \frac{1}{2}+e^{-\frac{2\pi i}{3}} t\right)  e^{n h\left( \frac{1}{2}+e^{-\frac{2\pi i}{3}} t\right) } dt
\end{equation*}
only yields a subexponentially small error term compared to $I_{n,m}^{(1,1)}$. Thus, we get the contribution
\begin{equation*}
  I_{n,m}^{(1)} = \frac{4 e^{n} e^{-\frac{4 \pi i}{3}}}{\pi i} \int_{0}^{\infty} t e^{-\frac{8}{3} n t^{3}} dt \cdot \left(1+\mathcal{O}(n^{-\frac{1}{4}})\right).
\end{equation*}

The integral
\begin{equation*}
  I_{n,m}^{(2)} = \frac{1}{2\pi i} \int_{0}^{1} \left( e^{\frac{2\pi i}{3}}\right)  g\left( \frac{1}{2}+e^{\frac{2\pi i}{3}} t\right)  e^{n h\left( \frac{1}{2}+e^{\frac{2\pi i}{3}} t\right) } dt,
\end{equation*}
can be treated in an analogous manner which gives the contribution
\begin{equation*}
  I_{n,m}^{(2)} = - \frac{4 e^{n} e^{\frac{4 \pi i}{3}}}{\pi i} \int_{0}^{\infty} t e^{-\frac{8}{3} n t^{3}} dt \cdot \left(1+\mathcal{O}(n^{-\frac{1}{4}})\right).
\end{equation*}

Moreover, one can show that the contribution of 
\begin{equation*}
  I_{n,m}^{(3)} = \frac{1}{2\pi i} \int_{\frac{\pi}{2}}^{\frac{3\pi}{2}} \frac{\sqrt{3}}{2} i e^{it} g\left( \frac{\sqrt{3}}{2} e^{i t}\right)  e^{n h\left( \frac{\sqrt{3}}{2} e^{it}\right) } dt
\end{equation*}
is asymptotically negligible compared to $I_{n,m}^{(1)}$ and $I_{n,m}^{(2)}$.

Collecting the contributions and evaluating the integral yields
\begin{align*}
  A_{n,m} & \sim \frac{4 e^{n}}{\pi i} \left(e^{-\frac{4 \pi i}{3}} - e^{\frac{4 \pi i}{3}}\right) \cdot \int_{0}^{\infty} t e^{-\frac{8}{3} n t^{3}} dt = \frac{4 \sqrt{3} \, e^{n}}{\pi} \int_{0}^{\infty} t e^{-\frac{8}{3} n t^{3}} dt \\
	& = \frac{3^{\frac{1}{6}} e^{n} n^{-\frac{2}{3}}}{\pi} \Gamma\left(\frac{2}{3}\right),
\end{align*}
and thus by using \eqref{eqn:relation_Mnm_Anm}:
\begin{equation*}
  M_{n,m} \sim \frac{\sqrt{2} \, 3^{\frac{1}{6}} \Gamma(\frac{2}{3}) n^{\frac{3n}{2}}}{\sqrt{\pi} \, n^{\frac{1}{6}}}.
\end{equation*}

\section{Conclusion\label{sec:further_research}}

This paper constitutes the  first treatment of parking functions for trees and mappings.
Several possible further research directions arise; we mention a few of them in the following.

\begin{enumerate}
\item Given a tree $T$ or a mapping $f$, we obtained general, but simple bounds for the number of tree and mapping parking functions $S(T,m)$ and $S(f,m)$, respectively. It is possible to obtain explicit formul{\ae} for some simple classes of trees (or mappings), e.g., for ``chain-like'' trees with only few branchings. However, the following question remains open: Is it possible in general to give some ``simple characterization'' of the numbers $S(T,m)$ and $S(f,m)$, respectively?

\item With the approach presented, one can also study the total number of parking functions for other important tree families as, e.g., labelled binary trees or labelled ordered trees. We already performed some preliminary work for these tree families and according to our considerations, the results are considerably more involved than for labelled unordered trees and mappings. However, they do not seem to lead to new phase change phenomena. Thus we did not comment on these studies here.

\item In contrast to the previous comment, the problem of determining the total number of parking functions seems to be interesting for so-called increasing (or decreasing) tree families (see, e.g., \cite{Drmota2009,PanPro2007}). That is, the labels along all leaf-to-root-paths form an increasing (or, decreasing) sequence. For so-called recursive trees (see, e.g., \cite{MahSmy1995,FuchHwaNei2006}), i.e., unordered increasing trees, the approach presented could be applied, but the differential equations occurring do not seem to yield ``tractable'' solutions. For such tree families quantities such as the ``sums of parking functions'' as studied in \cite{KunYan2003} could be worthwhile treating as well.

\item As for ordinary parking functions one could analyse important quantities for tree and mapping parking functions. E.g., the so-called ``total displacement'' (which is of particular interest in problems related to hashing algorithms, see \cite{FlaPobVio1998,Janson2001}), i.e., the total driving distance of the drivers, or ``individual displacements'' (the driving distance of the $k$-th driver, see \cite{Janson2005,Viola2005}) seem to lead to interesting questions.

\item A refinement of parking functions can be obtained by studying what has been called ``defective parking functions'' in \cite{CamJohPreSchwe2008}, or ``overflow'' in \cite{GonMun1984}, i.e., pairs $(T,s)$ or $(f,s)$, such that exactly $k$ drivers are unsuccessful. Preliminary studies indicate that the approach presented is suitable to obtain results in this direction as well.

\item Again, as for ordinary parking functions, one could consider enumeration problems for some restricted parking functions for trees (or mappings). E.g., we call $(T,s)$, with $T$ a size-$n$ tree and $s \in [n]^{n}$ a parking function for $T$, an ordered tree parking function, if $\pi^{-1}(v) < \pi^{-1}(w)$, whenever $v \prec w$ (i.e., if the $k$-th driver parks at parking space $w$, all predecessors of $w$ are already occupied by earlier drivers). Then it is easy to show that for any size-$n$ tree $T$ there are exactly $n!$ ordered tree parking functions.

\item Let us denote by $X_{n}$ the random variable measuring the number of parking functions $s$ with $n$ drivers for a randomly chosen labelled unordered tree $T$ of size $n$. Then, due to our previous results, we get the expected value of $X_{n}$ via
\begin{equation*}
  \mathbb{E}(X_{n}) = \frac{F_{n}}{T_{n}} \sim \frac{\sqrt{2 \pi} \, 2^{n+1} n^{n-\frac{1}{2}}}{e^{n}}.
\end{equation*}
However, with the approach presented here, it seems that we are not able to obtain higher moments or other results on the distribution of $X_{n}$.
\end{enumerate}

\bibliographystyle{abbrv}
\bibliography{lit}

\begin{thebibliography}{10}

\bibitem{BanFlaSchaefSor2001}
C.~Banderier, P.~Flajolet, G.~Schaeffer, and M.~Soria.
\newblock Random maps, coalescing saddles, singularity analysis, and {Airy}
  phenomena.
\newblock {\em Random Structures \& Algorithms}, 19:194--246, 2001.

\bibitem{BlaKon1977}
I.~F. Blake and A.~G. Konheim.
\newblock Big buckets are (are not) better!
\newblock {\em Journal of the Association for Computing Machinery},
  24:591--606, 1977.

\bibitem{CamJohPreSchwe2008}
P.~J. Cameron, D.~Johannsen, T.~Prellberg, and P.~Schweitzer.
\newblock Counting defective parking functions.
\newblock {\em Electronic Journal of Combinatorics}, 15:Research Paper~92,
  2008.
\newblock 15 pp.

\bibitem{DeBruijn1958}
N.~de~Bruijn.
\newblock {\em Asymptotic methods in analysis}.
\newblock North-Holland Publishing Co., Amsterdam, 1958.

\bibitem{Drmota2009}
M.~Drmota.
\newblock {\em Random trees}.
\newblock Springer, Wien, 2009.

\bibitem{FlaKnuPit1989}
P.~Flajolet, D.~E. Knuth, and B.~Pittel.
\newblock The first cycles in an evolving graph.
\newblock {\em Discrete Mathematics}, 75:167--215, 1989.

\bibitem{FlaPobVio1998}
P.~Flajolet, P.~Poblete, and A.~Viola.
\newblock On the analysis of linear probing hashing.
\newblock {\em Algorithmica}, 22:490--515, 1998.

\bibitem{flajolet2009analytic}
P.~Flajolet and R.~Sedgewick.
\newblock {\em Analytic combinatorics}.
\newblock Cambridge University press, 2009.

\bibitem{FuchHwaNei2006}
M.~Fuchs, H.-K. Hwang, and R.~Neininger.
\newblock Profiles of random trees: limit theorems for random recursive trees
  and binary search trees.
\newblock {\em Algorithmica}, 46:367--407, 2006.

\bibitem{GonMun1984}
G.~H. Gonnet and J.~I. Munro.
\newblock The analysis of linear probing sort by the use of a new mathematical
  transform.
\newblock {\em Journal of Algorithms}, 5:451--470, 1984.

\bibitem{Janson2001}
S.~Janson.
\newblock Asymptotic distribution for the cost of linear probing hashing.
\newblock {\em Random Structures \& Algorithms}, 19:438--471, 2001.

\bibitem{Janson2005}
S.~Janson.
\newblock Individual displacements for linear probing hashing with different
  insertion policies.
\newblock {\em ACM Transactions on Algorithms}, 1:177--213, 2005.

\bibitem{KonWei1966}
A.~G. Konheim and B.~Weiss.
\newblock An occupancy discipline and applications.
\newblock {\em SIAM Journal on Applied Mathematics}, 14:1266--1274, 1966.

\bibitem{KubPan2010}
M.~Kuba and A.~Panholzer.
\newblock Enumeration results for alternating tree families.
\newblock {\em European Journal of Combinatorics}, 7:1751--1780, 2010.

\bibitem{KunYan2003}
J.~P. Kung and C.~Yan.
\newblock Exact formulas for moments of sums of classical parking functions.
\newblock {\em Advances in Applied Mathematics}, 31:215--241, 2003.

\bibitem{PanPro2007}
A.~Panholzer and H.~Prodinger.
\newblock Level of nodes in increasing trees revisited.
\newblock {\em Random Structures \& Algorithms}, 31:203--226, 2007.

\bibitem{PosSha2004}
A.~Postnikov and B.~Shapiro.
\newblock Trees, parking functions, syzygies, and deformations of monomial
  ideals.
\newblock {\em Transactions of the American Mathematical Society},
  356:3109--3142, 2004.

\bibitem{MahSmy1995}
R.~T. Smythe and H.~M. Mahmoud.
\newblock A survey of recursive trees.
\newblock {\em Theory of Probability and Mathematical Statistics}, 51:1--27,
  1996.

\bibitem{Stanley1997}
R.~Stanley.
\newblock {\em Enumerative combinatorics}, volume I \& II.
\newblock Cambridge University press, 1997 \& 1999.

\bibitem{PitSta2002}
R.~P. Stanley and J.~Pitman.
\newblock A polytope related to empirical distributions, plane trees, parking
  functions, and the associahedron.
\newblock {\em Discrete \& Computational Geometry}, 27:603--634, 2002.

\bibitem{Viola2005}
A.~Viola.
\newblock Exact distribution of individual displacements in linear probing
  hashing.
\newblock {\em ACM Transactions on Algorithms}, 1:214--242, 2005.

\bibitem{Yan2001}
C.~H. Yan.
\newblock Generalized parking functions, tree inversions, and multicolored
  graphs.
\newblock {\em Advances in Applied Mathematics}, 27:641--670, 2001.

\end{thebibliography}

\end{document}